\documentclass[11pt,letterpaper]{amsart}
\usepackage[utf8]{inputenc}
\usepackage{notes-style}

\usepackage[backend=biber,maxbibnames=6]{biblatex}
\bibliography{./ref}
\title{Vertex functions for bow varieties and their Mirror Symmetry}

\author{T. M. Botta$^\diamond$, H. Dinkins$^\star$}
\date{\today}
\email{tommaso.botta@columbia.edu}

\email{hunte864@mit.edu}

\address{$^\diamond$ Department of Mathematics, Columbia University, New York, USA \\
$^\star$ Department of Mathematics, MIT, Cambridge, MA, USA}

\begin{document}

\begin{abstract}
    In this paper, we study the vertex functions of finite type $A$ bow varieties. Vertex functions are $K$-theoretic analogs of $I$-functions, and 3d mirror symmetry predicts that the $q$-difference equations satisfied by the vertex functions of a variety and its 3d mirror dual are the same after a change of variable swapping the roles of the various parameters. Thus the vertex functions are related by a matrix of elliptic functions, which is expected to be the elliptic stable envelope of M. Aganagic and A. Okounkov. We prove all of these statements.

    The strategy of our proof is to reduce to the case of cotangent bundles of complete flag varieties, for which the $q$-difference equations can be explicitly identified with Macdonald difference equations. A key ingredient in this reduction, of independent interest, involves relating vertex functions of the cotangent bundle of a partial flag variety with those of a ``finer" flag variety. Our formula involves specializing certain K\"ahler parameters (also called Novikov parameters) to singularities of the vertex functions. In the $\hbar \to \infty$ limit, this statement is expected to degenerate to an analogous result about $I$-functions of flag varieties.
    
\end{abstract}

\maketitle

\onehalfspacing
\setcounter{tocdepth}{1} 
\tableofcontents

\section{Introduction}

\subsection{Motivation from 3d mirror symmetry}

Three dimensional supersymmetric quantum field theories have been a key source of inspiration for algebraic geometry and representation theory. One aspect of these theories, called \textit{3d mirror symmetry}, predicts deep relationships between different algebraic varieties. Given such a theory $\mathcal{T}$, one can construct a Higgs branch $X_{0}$ and a Coulomb branch $X^{!}_{0}$, which are Poisson affine algebraic varieties and are expected to have symplectic singularities. 

We assume both branches have symplectic resolutions, denoted by $X$ and $X^{!}$ respectively. We will refer to $X$ and $X^{!}$ throughout this paper as \textit{3d mirror dual varieties}. One basic expectation is that these spaces are equipped with actions of tori $\Tt$ and $\Tt^{!}$ such that $|X^{\Tt}|=|(X^{!})^{\Tt^{!}}|$; in particular, these sets are both finite. See \cite{kamnitzer} for a survey of other statements. A much more sophisticated expectation, stemming from the insights of Okounkov \cite{okounkov2017enumerative, Okounkov_video}, asserts that particular curve counts in $X$ and $X^{!}$ known as \textit{vertex functions} are ``equivalent". This is the focus of the present paper.

When $\mathcal{T}$ is a quiver gauge theory, $X$ is a Nakajima quiver variety \cite{Nakajimaquiver}, and vertex functions were defined in \cite{okounkov2017enumerative} using quasimap technology developed in \cite{qm}. Vertex functions are generating functions recording equivariant counts of quasimaps from a parameterized $\mathbb{P}^{1}$ to $X$. The ``pole subtraction" property proven in \cite{aganagic2016elliptic} suggests that vertex functions of $X$ should be equivalent to vertex functions of $X^{!}$. Unfortunately, since quasimap technology only applies to varieties with a GIT quotient presentation, it is not known how to define vertex functions of $X^{!}$.

Nevertheless, another expectation from physics \cite{IntSei} is that there should exist a dual theory $\mathcal{T}^{!}$ such that the Higgs branch of $\mathcal{T}^{!}$ is the Coulomb branch of $\mathcal{T}$, and vice versa. If this is true, then $X^{!}$ should also have a GIT quotient presentation leading to a definition of vertex functions. Although it is not known how to construct $\mathcal{T}^{!}$ from $\mathcal{T}$ in general, there are some cases in which it can be done. Strictly speaking, the theories themselves are not well-defined mathematical objects, but the Higgs and Coulomb branches are. Bow varieties, which we turn to now, provide a rich class of examples in which both $X$ and $X^{!}$ can be given a uniform and rigorous description. 

\subsection{Bow varieties}

\textit{Bow varieties}, originally defined by Cherkis \cite{cherkis3, cherkis2, cherkis1} and later given a quiver description by Nakajima and Takayama \cite{Nakajima_Takayama}, are a class of smooth holomorphic symplectic varieties that fit into the above discussion. In particular if $X$ is a Nakajima quiver variety of type $A$, then both $X$ and $X^{!}$ can be described as bow varieties. More generally, the dual of any bow variety is another bow variety, although the same statement is generally false for a quiver variety\footnote{Whether a bow variety is isomorphic to a Nakajima quiver variety can be established directly from the brane diagram. We refer to \cite[\S 5.1]{rimanyi2020bow} for the precise statement and examples.}.

A bow variety $X$ is determined by a combinatorial object called a \textit{brane diagram} $\D$, like the following 
\begin{equation*}
\D=
\begin{tikzpicture}[baseline=0,scale=.2]
\draw [thick,red] (0.5,0) --(1.5,2); 
\draw[thick] (1,1)--(2.5,1) node [above] {$2$} -- (31,1);
\draw [thick,blue](4.5,0) --(3.5,2);  
\draw [thick](4.5,1)--(5.5,1) node [above] {$2$} -- (6.5,1);
\draw [thick,red](6.5,0) -- (7.5,2);  
\draw [thick](7.5,1) --(8.5,1) node [above] {$2$} -- (9.5,1); 
\draw[thick,blue] (10.5,0) -- (9.5,2);  
\draw[thick] (10.5,1) --(11.5,1) node [above] {$4$} -- (12.5,1); 
\draw [thick,red](12.5,0) -- (13.5,2);   
\draw [thick](13.5,1) --(14.5,1) node [above] {$3$} -- (15.5,1);
\draw[thick,red] (15.5,0) -- (16.5,2);  
\draw [thick](16.5,1) --(17.5,1) node [above] {$3$} -- (18.5,1);  
\draw [thick,red](18.5,0) -- (19.5,2);  
\draw [thick](19.5,1) --(20.5,1) node [above] {$4$} -- (21.5,1);
\draw [thick,blue](22.5,0) -- (21.5,2);
\draw [thick](22.5,1) --(23.5,1) node [above] {$3$} -- (24.5,1);  
\draw[thick,red] (24.5,0) -- (25.5,2); 
\draw[thick] (25.5,1) --(26.5,1) node [above] {$2$} -- (27.5,1);
\draw [thick,blue](28.5,0) -- (27.5,2);  
\draw [thick](28.5,1) --(29.5,1) node [above] {$2$} -- (30.5,1);   
\draw [thick,blue](31.5,0) -- (30.5,2);   

\end{tikzpicture}.
\end{equation*}
The horizontal black segments are called D5 branes, the red segments are called NS5 branes, and the blue ones are called D5 branes. The non-negative integers lying above the D3 branes are the bow variety analogs of the dimensions assigned to gauge vertices of a quiver variety. 

Given $\D$, it is trivial to construct the brane diagram $\D^{!}$ for the mirror dual $X^{!}$: it suffices to swap D5 branes (written as $\bs$) and NS5 branes (written as $\fs$):
\begin{equation*}
\D^!=
\begin{tikzpicture}[baseline=0,scale=.2]
\draw [thick,blue] (1.5,0) --(0.5,2); 
\draw[thick] (1,1)--(2.5,1) node [above] {$2$} -- (31,1);
\draw [thick,red](3.5,0) --(4.5,2);  
\draw [thick](4.5,1)--(5.5,1) node [above] {$2$} -- (6.5,1);
\draw [thick,blue](7.5,0) -- (6.5,2);  
\draw [thick](8.5,1) node [above] {$2$} --(7.5,1)  -- (9.5,1); 
\draw[thick,red] (9.5,0) -- (10.5,2);  
\draw[thick] (11.5,1) node [above] {$4$} --(10.5,1)  -- (12.5,1); 
\draw [thick,blue](13.5,0) -- (12.5,2);   
\draw [thick](14.5,1) node [above] {$3$} --(13.5,1)  -- (15.5,1);
\draw[thick,blue] (16.5,0) -- (15.5,2);  
\draw [thick](17.5,1)  node [above] {$3$} --(16.5,1) -- (18.5,1);  
\draw [thick,blue](19.5,0) -- (18.5,2);  
\draw [thick](19.5,1) --(20.5,1) node [above] {$4$} -- (21.5,1);
\draw [thick,red](21.5,0) -- (22.5,2);
\draw [thick](23.5,1) node [above] {$3$} --(22.5,1)  -- (24.5,1);  
\draw[thick,blue] (25.5,0) -- (24.5,2); 
\draw[thick] (26.5,1) node [above] {$2$} --(25.5,1)  -- (27.5,1);
\draw [thick,red](27.5,0) -- (28.5,2);  
\draw [thick](29.5,1) node [above] {$2$} --(28.5,1)  -- (30.5,1);   
\draw [thick,red](30.5,0) -- (31.5,2);   

\end{tikzpicture}.
\end{equation*}
Rimanyi and Shou \cite{rimanyi2020bow} showed that $X$ (resp. $X^{!}$) is equipped with the action of a torus $\Tt$ (resp. $\Tt^{!}$) with finitely many fixed points. They also describe explicitly the expected bijection 
\[
X^{\Tt} \ni f \leftrightarrow f^{!} \in (X^{!})^{\Tt^{!}}
\]
between fixed points. 

For a bow variety $X$ and some choices of additional data, one can define a certain elliptic characteristic class called the \textit{elliptic stable envelope}, denoted $\MSstab^{X}(f)$, for each fixed point $f \in X^{\Tt}$. By restricting to torus fixed points, one obtains a matrix of elliptic functions $\MSstab^{X}_{gf}:=\MSstab^{X}(f)|_{g}$ depending on variables $Q_1,\ldots,Q_{m-1},u_1,\ldots,u_{n-1},\hbar,q$\footnote{Here $m$ and $n$ are the number of NS5 and D5 branes in $\D$, respectively.}. Three dimensional mirror symmetry of bow varieties from the perspective of elliptic stable envelopes was studied by Rimanyi and the first author in \cite{BR}.

\subsection{Main theorem}

The vertex function of $X$ is defined as the formal power series 
\[
\ver^{X}:=\sum_{d} \ev_{p,*}\left(\vrs^{d} \right) Q^{d} \in K_{\Tt\times \mathbb{C}^{\times}_{q}}(X)_{loc}[[Q]],
\]
where $\ev_{p}: \qm^{d}_{\ns p}(X) \to X$ and $\qm^{d}_{\ns p}(X)$ is the moduli space of stable degree $d$ quasimaps from $\mathbb{P}^{1}$ to $X$ nonsingular at a chosen point $p \in \mathbb{P}^{1}$. We refer to \S\ref{sec: quasimaps} for further explanation of the notation. Although we only write explicitly the dependence on $Q$, the \emph{K\"ahler parameters}, the restriction of the vertex function to a fixed point $f \in X^{\Tt}$ is an element of
\[
\ver^{X}_{f}(Q):=\ver^{X}(Q)|_{f} \in \mathbb{Q}(u_1,\ldots,u_{n-1},\hbar,q)[[Q_{1},\ldots,Q_{m-1}]].
\]
In \eqref{eq: def of MSver}, we specify a particular normalization of the vertex which we denote by $\MSver^{X}(Q)$.

Similarly, we have the vertex functions of $X^{!}$:
\[
\MSver^{X^{!}}_{f^{!}}(Q^{!}) \in \mathbb{Q}(u_1^{!},\ldots,u_{m-1}^{!},\hbar^{!},q^{!})[[Q_{1}^{!},\ldots,Q_{n-1}^{!}]].
\]
Notice that the roles of $m$ and $n$ have switched. 

The choice of character for the GIT quotient used to construct $X^{!}$ provides a chamber $\chamb$ for the torus action on $X$, and vice versa. With respect to this chamber, we can decompose the tangent space at a fixed point into attracting and repelling directions:
\[
T_{f} X= N_{f}^{-}+N_{f}^{+}.
\]
We need one more transcendental function, the reciprocal of the $q$-gamma function, defined on nonzero torus characters by
\[
\Phi(x):=\prod_{i=0}^{\infty}(1- x q^i).
\]
and extended to sums and differences by the rule $\Phi(x+y)=\Phi(x)\Phi(y)$. Then we define
\[
\Phi^{\pm}_{f}=\Phi((q-\hbar^{-1})N_{f}^{\pm}).
\]

Our main theorem, mirror symmetry of vertex functions, is the following.

\begin{theorem}[Theorem \ref{thm: mirsym}]\label{thm: mirsym intro}
Let $X$ be a bow variety with dual $X^{!}$. For each $f \in X^{\Tt}$, we have
\[
\mirmap\left( \Phi^{+}_{f^{!}} \MSver^{X^{!}}_{f^!}(Q^{!}) \right)=\sum_{g\in X^{\Tt}}\MSstab^X_{gf} \Phi^{-}_{g} \MSver^{X}_{g}(Q^{-1}).
\]
where the identification of parameters $\mirmap$ is defined by
\[
q^{!} \mapsto q \qquad \hbar^{!} \mapsto \frac{1}{\hbar q} \qquad (q^{!})^{-\w(\Zb_{i}^{!})} Q_{i}^{!}\mapsto u_{i} \qquad u_{i}^{!} \mapsto  q^{\w(\Zb_{i})} Q_{i}
\]
and $\w(\cdot)$ is the weight of the brane.
\end{theorem}

At first glance, Theorem \ref{thm: mirsym intro} does not appear symmetric. By using mirror symmetry of stable envelopes from \cite{BR}, we show in Theorem \ref{thm: opmirsym} how to rewrite mirror symmetry of vertex functions with the roles of $X$ and $X^{!}$ swapped.

It is instructive to consider a degeneration of the previous theorem. Upon specializing $\hbar=1$ (equivalently, $\hbar^{!} q^{!}=1$), the various terms degenerate as follows:
\begin{align*}
    \MSver^{X}_{g}(Q^{-1}) \to 1 \qquad \MSstab^{X}_{gf} \to \delta_{g,f}\qquad 
    \Phi^{+}_{f^{!}} \to 1 \qquad \Phi_{g}^{-} \to \prod_{\substack{\text{weights $w$} \\ \text{of $N_{g}^{-}$}}}(1-w)^{-1}
\end{align*}
So we obtain the equality
\[
\mirmap\left( \MSver^{X^{!}}_{f^{!}}(Q^{!})|_{\hbar^{!} q^{!} =1}\right)= \prod_{\substack{\text{weights $w$} \\ \text{of $N_{g}^{-}$}}}(1-w)^{-1}
\]
allowing one to extract the weights of $TX$ at fixed points from the vertex functions of $X^{!}$. Conversely, we immediately see that $\MSver^{X^{!}}_{f^{!}}(Q^{!})|_{\hbar^{!} q^{!} =1}$ is the power series expansion of a rational function of $Q^{!}$. 

Other interesting and less-extreme degenerations of Theorem \ref{thm: mirsym intro} were studied by Smirnov and the second author in \cite{dinksmir4} and \cite{dinksmir}. The techniques of \cite{dinksmir4} immediately apply to our setting, giving an interpretation of $K$-theoretic stable envelopes of $X$ in terms of the ``index limits" of vertex functions of $X^{!}$. As in the $\hbar=1$ degeneration, rationality of these index limits, a nontrivial and new result, follows immediately. 

3d mirror symmetry for vertex functions of cotangent bundles of Grassmannians and their duals was studied from a combinatorial perspective in \cite{dinkms1}. There it is shown how a special case of mirror symmetry is equivalent to an explicit description of eigenvectors and eigenvalues of the Ruijsenaars-Macdonald operators of row type defined in \cite{NoumiSano}.

\subsection{Strategy of the proof}

The proof of Theorem \ref{thm: mirsym intro} parallels the strategy used in \cite{BR} to prove mirror symmetry of elliptic stable envelopes. Namely, we prove the statement by induction on the \textit{weights} of the D5 and NS5 branes in $\D$, with the base case being when all the branes have weight 1, see \S\ref{sec: bow} for the definition of weight. When all the branes have weight 1, the bow variety is known to be isomorphic to the cotangent bundle of a complete flag variety for which mirror symmetery was proven by the second author in \cite{dinkms2} using the fact that the $q$-difference equations satisfied by vertex functions in both the $Q$ and $u$ variables are Macdonald difference equations.

\subsection{Reduction of D5 weights}

The reduction of D5 brane weights takes the following form. Let $\D$ be a brane diagram and let $\Ab$ be a D5 brane in $\D$ of weight $\w\geq 2$. Choosing a splitting $\w=\w'+\w''$, we consider the brane diagram $\wt \D$ in which $\Ab$ has been replaced by two D5 branes $\Ab'$ and $\Ab''$ of weights $\w'$ and $\w''$. Let $X$ and $\wt X$ be the corresponding bow varieties which are equipped with actions of tori $\Tt$ and $\wt \Tt$, respectively. As we will review in \S\ref{sec: d5 resolutions}, \cite{BR} constructs a closed embedding $j: X \hookrightarrow \wt X$ equivariant along an inclusion $\varphi:\Tt \hookrightarrow \wt \Tt$. We prove that the corresponding pullback in $K$-theory preserves vertex functions:

\begin{theorem}[Theorem \ref{thm: d5 vertex}]\label{thm: intro d5 vertex}
    The following holds:
    \[
     j^* \varphi^* \MSver^{\wt X}= \MSver^{X}.
    \]
\end{theorem}
Intuitively, this result holds because of the following property of $j$. For any $x_1,x_2 \in X$, any morphism $f: \mathbb{P}^{1} \to \wt X$ satisfying $f(0)=j(x_1)$, $f(\infty)=j(x_2)$ must lie entirely in $j(X)$. Due to the presence of singularities, quasimaps to $\wt X$ need not literally have the form $f:\mathbb{P}^{1} \to \wt X$. Nevertheless, we are able to extend the above observation to quasimaps compatibly with the obstruction theories. 

\subsection{Reduction of NS5 weights}

The reduction of NS5 brane weights is more complicated and takes the following form. Let $\D$ be a brane diagram and let $\Zb$ be an NS5 brane in $\D$ of weight $\w\geq 2$. Let $\ol \D$ be the brane diagram obtained by replacing $\Zb$ by two NS5 branes $\Zb'$ and $\Zb''$ of weights $\w'$ and $\w''$, where $\w=\w'+\w''$. In this case, the two bow varieties $X$ and $\ol X$ are related by a Lagrangian correspondence:
 \begin{equation}\label{eq: ns5 diagram intro}
         \begin{tikzcd}
       \overline X & L \arrow[l, swap, "j"] \arrow[r, "p"] & X.
       \end{tikzcd}
     \end{equation}
where $j$ is a closed immersion and $p$ is a Grassmannian fibration. The two maps are equivariant for natural actions a single torus $\Tt$ on each term. It was proven in \cite{BR} that convolution in the above diagram relates stable envelopes of $X$ to those of $\ol X$. There is an extra K\"ahler parameter $Q'$ for $\ol X$, and this must be specialized as $\psi^{*} Q'=\hbar^{k}$ for some $k \in \mathbb{Z}$. In view of this, one might expect the vertex functions to be related by 
\[
\MSver^{X}(Q) \overset{?}{=} p_{*}\psi^{*}j^{*} \MSver^{\ol X}(Q,Q')
\]
Here we encounter a new difficulty with vertex functions: $\ver^{\ol X}$ is a formal power series in the variables $(Q,Q')$, and the specialization $\psi^{*}$ is not well-defined on power series. 

A first attempt to correct this issue proceeds as follows. The power series defining a vertex function is known to be convergent in a certain region in the K\"ahler parameters. So one may hope that $\psi^{*}$ may nevertheless be convergent. Unfortunately this does not happen either. Instead, the power series defining $\MSver^{\ol X}$ diverges at $\psi^{*}$. Moreover, $\psi^{*}$ is precisely a point on the \emph{boundary} of the radius of convergence.

As a second attempt, one may try to apply $\psi^{*}$ to the meromorphic function of $(Q,Q')$ defined by an analytic continuation of $\MSver^{\ol X}$. Since vertex functions are known to satisfy certain scalar $q$-difference equations \cite[\S8.3]{Okounkov_lectures}, they can automatically be analytically continued. At first glance, this idea also fails: the specialization $\psi^{*}$ really is a singularity of this meromorphic function. Nevertheless, we prove that the singularity at $\psi^{*}$ is a simple pole. By multiplying by an appropriate factor to cancel this pole, we are able to achieve our comparison.

\begin{theorem}[Theorems \ref{thm: ns5vertex} and \ref{thm: cosepNS5vertex}] \label{thm: intro ns5vertex}
    Let $f \in X^{\Tt}$ and let $\ol f \in L^{\Tt}$ such that $p(\ol f)=f$. Then 
\begin{equation*}
\MSver^{X}_{f}(Q)=\Phi_{\chamb_{\ol f},Y}\psi^{*}\left( \MSflaglim^{Y}(Q'^{-1})^{-1} \MSver^{\ol X}_{\ol f}(Q,Q')\right)
\end{equation*}
\end{theorem}
For the notation in this theorem, we refer to \S\ref{sec: NS5 and vertex}. It is interesting to consider a few special cases of Theorem \ref{thm: intro ns5vertex}. For a certain choice of brane diagram $\D$, $X=T^{*}\Fl_{d_1,d_2,\ldots,d_m}$, the cotangent bundle to the partial flag variety parameterizing quotients
\[
\Fl_{d_1,d_2,\ldots,d_m}=\{\mathbb{C}^{d_{m}}=V_{m} \twoheadrightarrow V_{m-1} \twoheadrightarrow \ldots \twoheadrightarrow V_{1} \, \mid \, \dim V_{i}=d_{i}\}
\]
Then $\ol X=T^{*}\Fl_{d_1,\ldots,d_{k-1},d,d_{k},\ldots,d_{m}}$, the cotangent bundle to a flag variety ``refining" that of $X$\footnote{In terms of brane weights as above, $\w=d_{k}-d_{k-1}$, $\w'=d-d_{k-1}$, and $\w''=d_{k}-d$.}. Theorem \ref{thm: intro ns5vertex} recovers the vertex functions of $X$ from those of $\ol X$. We are not aware of other places where curve counts, for example $I$-functions, of such varieties are related. In the $\hbar \to \infty$ limit, these vertex functions are expected degenerate to the $K$-theoretic $I$-functions of the corresponding flag varieties, see \cite[\S4.4]{KorZeit}.

Iterating the procedure, one can take $X=T^{*} \Fl_{1,n}$ and $\ol X= T^{*} \Fl_{1,2,\ldots,n}$. The vertex functions of $X$ are equal to the classical hypergeometric functions $_{n}\phi_{n-1}$, see \cite{Gasper_Rahman_2004}, for certain values of the parameters. The vertex functions of $\ol X$ are the ``Macdonald functions" (also called ``non-stationary Ruijsenaars functions") studied, for example, in \cite{BFS, LMS,NSmac}. Then Theorem \ref{thm: intro ns5vertex} roughly states that the basic hypergeometric function is a specialization of the Macdonald function. 

As a final example, one could take $X=T^{*} \Fl_{0,d_m}=\{pt\}$ and $\ol X =T^{*} \Fl_{0,d_1,d_2,\ldots,d_m}$. In this case, $\MSver^{X}=1$ and we obtain an explicit and nontrivial evaluation formula for $\MSver^{\ol X}(Q_1,\ldots,Q_{m})$ at $(Q_1,\ldots,Q_{m})=(1,\hbar^{-d_1},\hbar^{-d_2},\ldots,\hbar^{-d_{m-1}})$. There is nothing special about $T^{*}\Fl$ in this example and one could similarly obtain evaluation formulas for vertex functions of any bow variety.

\subsection{Connection with quantum groups actions}

Theorem \ref{thm: intro d5 vertex} and Theorem \ref{thm: intro ns5vertex} can be thought of as the enumerative manifestation of a classical representation theoretic construction known as fusion of $R$-matrices, which, albeit not directly used in the paper, has been an important source of inspiration. 

Up to isomorphism, a bow variety $X$ with $m$ NS5 and $n$ D5 branes is uniquely determined by a pair of dimension vectors $r\in \NN^m$ and $c\in \NN^n$ recording the charges of the branes. Accordingly, we will sometimes denote it by $X(r,c)$. 
The equivariant cohomology (resp. $K$-theory or elliptic cohomology) of the disjoint union $\sqcup_{r} X(r,c)$ admits a natural action of the Yangian\footnote{This action can be realized via stable envelopes and the RTT formalism (cf. \cite{FRT}) using the logic of \cite{maulik2012quantum} or via cohomological Hall algebras. The equivalence of the two approaches follows from \cite{BDCohaYangians, SV} and the fusion procedure described here.} (resp. quantum affine algebra or elliptic quantum group) of $\mathfrak{sl}_m$ and, with respect to this action, is identified with the evaluation representation
\[
\Lambda^{c_1}\C^m(a_1)\otimes \dots\otimes \Lambda^{c_n}\C^m(a_n).
\]
The braidings for the Yangian action can be reconstructed geometrically via the stable envelopes, whose transition matrix 
\[
R(a_1-a_2):=(\Stab_{\lbrace a_1<a_2 \rbrace })^{-1}\circ \Stab_{\lbrace a_1>a_2\rbrace }\in \End(\Lambda^{c_1}\C^m(a_1)\otimes \Lambda^{c_2}\C^m(a_2))
\]
provides a solution of the spectral Yang-Baxter equation 
\[
R^{(12)}(a_1-a_2)R^{(13)}(a_1-a_3)R^{(23)}(a_2-a_3)=R^{(23)}(a_2-a_3)R^{(13)}(a_1-a_3)R^{(12)}(a_1-a_2).
\] 
It is well known, see for example \cite{fusion}, that the $R$-matrices of all exterior power representations can be uniquely reconstructed from those of a suitable tensor product of fundamental evaluation representations $\Lambda^1\C^m(a)=\C^m(a)$. Tensor products of the latter correspond geometrically to those bow varieties that are isomorphic to the cotangent bundle of a partial flag variety.

This classical construction can be geometrically realized by means of the embeddings $j: X\to \wt X$, which, inductively, embed any bow variety in the cotangent bundle of a partial flag variety. In fact, it is shown in \cite[\S6.6]{BR} that this embedding preserves the stable envelopes of the bow varieties $X$ and $\wt X$, and induces fusion for the associated $R$-matrices. 

In conclusion, Theorem \ref{thm: intro d5 vertex} can be seen as an enumerative analog of fusion of $R$-matrices. Likewise, Theorem \ref{thm: intro ns5vertex} can be seen as the mirror dual statement, which was not known classically.

\subsection{Extended worked example}

To aid the reader, we have written all of our results explicitly in the simplest nontrivial example, $X=T^{*}\mathbb{P}^{1}$, in Appendix \ref{appendix MS}. We explicitly write out the mirror symmetry statements, including Theorem \ref{thm: mirsym intro}. We also write out the statements of Theorems \ref{thm: intro d5 vertex} and \ref{thm: intro ns5vertex}. Because the vertex functions of $X$ can be written in terms of the basic hypergeometric function $_{2} \phi_{1}$ introduced by Heine \cite{Heine}, we are also able to prove all the statements directly and explicitly using well-known properties of the basic hypergeometric function. We encourage the reader unfamiliar with vertex functions and stable envelopes to start here before proceeding to the rest of the paper.

\subsection{Outline of the paper}
In \S\ref{sec: bow}, we review the definition and basic properties of bow varieties, and we fix notation that will be used throughout the paper. 

In \S\ref{sec: stab}, we provide a brief review of elliptic stable envelopes and fix the normalizations that will be used in our main theorem. We also review the mirror symmetry theorem for stable envelopes proven in \cite{BR}.

In \S\ref{sec: quasimaps}, we review the theory of quasimaps to GIT quotients. We apply this to bow varieties. In particular, we provide an explicit combinatorial formula for vertex functions, study their invariance under Hanany-Witten transition, and fix the normalizations that appear in Theorem \ref{thm: mirsym intro}. When $X=T^{*}\Fl$, we give a refined formula which parallels that of the ``Macdonald function" of \cite{NSmac}. We also discuss various analytic properties of vertex functions.

In \S\ref{sec: d5 resolutions}, we prove Theorem \ref{thm: intro d5 vertex}.

In \S\ref{sec: ns5 resolutions}, we review the construction of \eqref{eq: ns5 diagram intro} and the corresponding compatibility of stable envelopes.

In \S\ref{sec: NS5 and vertex}, we prove Theorem \ref{thm: intro ns5vertex}. Our proof takes several steps. First, we review the explicit $q$-difference equations in the $u$ variables for the vertex functions of cotangent bundles of partial flag varieties identified in \cite{KorZeit}. A straightforward calculation shows the compatibility of these equations under $\psi^{*}$. Since these equations have unique solutions up to leading term, we are able to prove Theorem \ref{thm: intro ns5vertex} for these special cases. The subtle aspect of the proof involves justifying that the specialization $\psi^{*}$ is well-defined in the right hand side. We do that inductively by making use of results from \cite{NSmac} about the singularities of the Macdonald function. Then we prove Theorem \ref{thm: intro ns5vertex} in full generality by combining this partial result with Theorem \ref{thm: intro d5 vertex}.

In \S\ref{sec: mirsym}, we assemble all of our results together to prove mirror symmetry of vertex functions.

\subsection{Further directions}

One interesting manifestation of 3d mirror symmetry was explored recently by Smirnov and Varchenko, who use 3d mirror symmetry in cohomology to study $p$-adic approximations of vertex functions, \cite[Theorem 3.4]{SV}. It would be interesting to understand how the $K$-theoretic mirror symmetry studied here degenerates to cohomology. In $K$-theory, vertex functions can be written as $q$-integrals, whereas the cohomological vertex functions are expected to be legitimate contour integrals. In particular, a new phenomenon in cohomology is the appearance of multi-valued functions.

In this paper, we succeed in proving 3d mirror symmetry of vertex functions of finite type $A$ bow varieties. It is of great interest to extend our techniques to affine type $A$ bow varieties. In particular, the Hilbert scheme of points on $\mathbb{C}^{2}$ is self dual, and its quasimap vertex function is the same as the 1-leg PT vertex \cite{PT}. Similarly, vertex functions of Hilbert scheme of points on other resolutions of type $A$ surface singularities are related to curve counting in 3-folds, see \cite{Liuqmstab}. Liu has shown in \cite{Liuqmstab} that the fully equivariant crepant resolution conjecture for sheaf counting theories on $3$-folds, see \cite{BCR,CRC}, is a simple corollary of 3d mirror symmetry of vertex functions. We are unsure which, if any, of the techniques used here can be extended to affine type $A$.

\subsection{Conventions and notation}

All schemes and stacks are defined over $\C$. 

In this article, we will constantly compare certain objects associated to mirror dual bow varieties $X$ and $X^!$. To clarify which side of the mirror symmetry an object belongs to, we will sometime use superscripts. For instance, although our standard notation for the stable envelope is $\Stab$, we will sometimes write $\Stab^X$ or $\Stab^{X^!}$ to stress whether $\Stab$ is the stable envelope of $X$ or its dual $X^!$. 

We also have two notations for both the stable envelope ($\Stab$ and $\MSstab$) and the vertex function $\ver$ and $\MSver$. While $\Stab$ and $V$ are the ``bare'' stable envelopes and vertex functions, namely the ones most commonly defined in the literature, $\MSstab$ and $\MSver$ are those normalized to make mirror symmetry more natural. 

\subsection{Acknowledgments}

We would like to thank Davesh Maulik, Leonardo Mihalcea, Andrei Okounkov, Rich\'ard Rim\'anyi, Jun'ichi Shiraishi, and Andrey Smirnov for helpful conversations related to this project. The authors are particularly grateful to Rich\'ard Rim\'anyi for sharing parts of his code and for his involvement in the early stages of this project.

T.M. Botta was supported by grant P500PT-222219 of the Swiss National Science Foundation and by a Postdoctoral Fellowship at Columbia University. H. Dinkins was supported by NSF grant DMS-2303286 at MIT and the NSF RTG grant Algebraic Geometry and  Representation Theory at Northeastern University DMS–1645877.

\section{Bow varieties}\label{sec: bow}

\subsection{Brane diagrams} \label{sec: brane diagrams}
Combinatorial objects like $\D=\ttt{\fs 2\bs 2\fs 2\bs 4\fs 3\fs 3\fs 4\bs 3\fs 2\bs 2\bs}$ will be called (type A) brane diagrams. The red forward-leaning lines are called NS5 branes, denoted by $\Zb$. The blue backward-leaning lines are called D5 branes, denoted by $\Ab$. The positions between 5-branes are called D3 branes, denoted by $\Xb$, and the integer sitting there is called its multiplicity or dimension and will be written $d_{\Xb}$. 

If all NS5 branes are to the left of all D5 branes, we call the diagram {\em separated}. If all NS5 branes are to the right of all D5 branes, we call the diagram {\em co-separated}.

The {\em charge} of an NS5  brane $\Zb$ or a D5 brane $\Ab$ is defined by
\begin{align*}
\ch(\Zb)= & (d_{\Zb^+}-d_{\Zb^-})+|\{\text{D5 branes left of $\Zb$}\}|,
\\ 
\ch(\Ab)= & (d_{\Ab^-}-d_{\Ab^+})+|\{\text{NS5 branes right of $\Ab$}\}|.
\end{align*}
The superscripts $+, -$ refer to the branes directly to the right and left, respectively. 
We define the {\em local charge} (or ``{\em weight}'') of 5-branes by  $\w(\Zb)=|d_{\Zb^+}-d_{\Zb^-}|$, $\w(\Ab)=|d_{\Ab^+}-d_{\Ab^-}|$. For an NS5 brane $\Zb$, let $\ell(\Zb)$ denote the number of D5 branes left of $\Zb$. 

Branes of the same type are labelled from left to right with increasing positive integers. Accordingly, we collect the D5 (resp. NS5) charges in a vector $c=(c_1, \dots c_n)$ (resp. $r=(r_1, \dots r_n)$) where $r_i=\ch(\Ab_i)$ (resp. $c_i=\ch(\Zb_i)$).

\subsection{The bow variety} \label{sec:def of bow variety}

To a D3 brane $\Xb$, we associate a complex vector space $W_{\Xb}$ of dimension $d_{\Xb}$. To a D5 brane $\Ab$ we associate a one-dimensional space $\C_{\Ab}$ with the standard $\GL(\C_{\Ab})$ action and the ``three-way part''
\begin{align*}
\MM_{\Ab}= &\Hom(W_{{\Ab}^+},W_{{\Ab}^-}) \oplus
\hbar\Hom(W_{{\Ab}^+},\C_{\Ab}) \oplus \Hom(\C_{\Ab},W_{{\Ab}^-}) \\
&  \oplus\hbar\End(W_{{\Ab}^-}) \oplus \hbar\End(W_{{\Ab}^+}),
\end{align*} 
with elements denoted $(A_{\Ab}, b_{\Ab}, a_{\Ab}, B_{\Ab}, B'_{\Ab})$, and $\NN_{\Ab}=\hbar\Hom(W_{{\Ab}^+},W_{{\Ab}^-})$. To an NS5 brane $\Zb$ we associate the ``two-way part''
\[
\MM_{\Zb}= \hbar\Hom(W_{{\Zb}^+},W_{{\Zb}^-}) \oplus
\Hom(W_{{\Zb}^-},W_{{\Zb}^+}),
\]
whose elements will  be denoted by $(C_{\Zb}, D_{\Zb})$. To a D3 brane $\Xb$ we associate $\NN_{\Xb}=\hbar\End(W_{\Xb})$. In these formulas, the $\hbar$ factor means that the corresponding term is endowed with the weight 1 action of a one dimensional torus denoted by $\Cs_{\hbar}$.
Let 
\begin{align}
    \MM&=\bigoplus_{\Ab} \MM_{\Ab} \oplus \bigoplus_{\Zb} \MM_{\Zb}   \label{eq: MM}
    \\
    \NN&=\bigoplus_{\Ab} \NN_{\Ab} \oplus \bigoplus_{\Xb} \NN_{\Xb}.  \label{eq: NN}
\end{align}
We define a map $\mu:\MM \to \NN$ componentwise as follows.
\begin{itemize}
\item 
The $\NN_{\Ab}$-component of $\mu$ is 
$B_{\Ab} A_{\Ab} -A_{\Ab} B'_{\Ab}+a_{\Ab} b_{\Ab}$.
\item
The $\NN_{\Xb}$-components of $\mu$ depend on the diagram:
\begin{itemize}
\item[\ttt{\bs -\bs}] If $\Xb$ is in between two D5 branes then it is $B'_{\Xb^-}-B_{\Xb^+}$. 
\item[\ttt{{\fs}-{\fs}}] If $\Xb$ is in between two NS5 branes then it is $C_{\Xb^+}D_{\Xb^+}$ $-D_{\Xb^-}C_{\Xb^-}$.
\item[\ttt{{\fs}-\bs}] If $\Xb^-$ is an NS5 brane and $\Xb^+$ is a D5 brane then it is $-D_{\Xb^-}C_{\Xb^-}$ $-B_{\Xb^+}$.
\item[\ttt{\bs -{\fs}}] If $\Xb^-$ is a  D5 brane and $\Xb^-$ is an NS5 brane then it is $C_{\Xb^+}D_{\Xb^+}+B'_{\Xb^-}$.
\end{itemize}
\end{itemize}
Let $M$ consist of points of $\mu^{-1}(0)\subset \MM$ for which the stability conditions  
\begin{itemize}
\item[(S1)]  if $S\leq W_{{\Ab}^+}$ is a subspace with $B'_{\Ab}(S)\subset S$, $A_{\Ab}(S)=0$, $b_{\Ab}(S)=0$ then $S=0$,
\item[(S2)]  if $T \leq W_{{\Ab}^-}$ is a subspace with $B_{\Ab}(T)\subset T$, $Im(A_{\Ab})+Im(a_{\Ab})\subset T$ then $T=W_{\Ab^-}$
\end{itemize}
hold for all D5 branes $\Ab$. Although the stability conditions (S1) and (S2) are manifestly open conditions, the variety ${M}$ is affine \cite[\S2]{takayama_2016}.

Let $G=\prod_{\Xb} \GL(W_\Xb)$ and consider the character 
\begin{equation}\label{eq:character}
\chi: G\ \to \C^{\times}, 
\qquad\qquad
(g_\Xb)_{\Xb} \mapsto \prod_{\Xb'} \det(g_{\Xb'}),
\end{equation}
where the product in the definition of the character runs over D3 branes $\Xb'$ such that $(\Xb')^-$ is an NS5 brane (in picture: \ttt{{\fs}}$\!\Xb'$). The bow variety $X(\D)$ is defined as the GIT quotient of the affine variety ${M}$ with respect to the action of $G$ and the character $\chi$: 
\[
X(\D):=M\git^{\chi} G.
\]

The resulting variety $X(\D)$ is smooth and holomorphic symplectic. It comes with the action of the torus $\Tt=\At\times \C_{\hbar}^\times$, where $\At=\prod_{\Ab} \GL(\C_{\Ab})$ and the action of the subtorus $\Cs_{\hbar}$ is induced by the $\hbar$-weighting of $\MM$. The vector spaces $W_{\Xb}$ associated with the D3 branes induce ``tautological'' bundles $\tb_{\Xb}$ (of the given rank).

Since there are no strictly semistable points and all the stabilizers of the stable points are trivial, the bow variety is isomorphic to the quotient stack ${M}^{\chi\semis}/G$. As a consequence, we have an open embedding 
\[
X(\D)\hookrightarrow \FX(\D),
\]
where $\FX(\D)$ is the Artin stack $M/G$. Whenever the brane diagram is understood, we will drop it from the notation and write $X$ and $\FX$ in place of $X(\D)$ and $\FX(\D)$, respectively.

\subsection{Hanany-Witten isomorphims}
\label{subsec: HW iso}

The local surgery 
\[
\ttt{$d_1$\fs $d_2$\bs $d_3$} \leftrightarrow \ttt{$d_1$\bs $d_1+d_3-d_2+1$\fs $d_3$}
\] 
on diagrams is called Hanany-Witten transition. It is a fact \cite[Prop. 8.1]{Nakajima_Takayama} that under such transition the associated bow varieties are isomorphic.

The charges of 5-branes are invariant under HW transition, and in fact, the vectors of NS5 and D5 charges $r$ and $c$ are a complete invariant of a HW equivalence class. Any brane diagram can be made separated or co-separated by a sequence of HW moves. We will call a bow variety $X=X(\D)$ separated (resp. co-separated) if $\D$ is separated (resp. co-separated).

\subsection{Affinization and handsaw variety} 
\label{sec:Affinization}

Assume that $\D$ is either separated or co-separated. As recalled above, the variety $M$ is affine. As a consequence, bow varieties come with a projective morphism
\[
\pi:X(\D)\to X_0(\D):=\text{Spec}(\C[{M}]^{G})
\]
to an affine variety. By neglecting the two-way parts of the diagram, we get a map 
\[
\mu_{HS}:(\oplus_{\Ab} \MM_{\Ab})\to (\oplus_{\Ab} \NN_{\Ab}) \oplus (\oplus_{\Xb} \NN_{\Xb}).
\]
The subscript refers to the shape of the quiver associated to the domain, which resembles a handsaw. Set $G_{HS}=\prod_{\Ab}\GL(W_{\Ab_-})$.

Let $M_{HS}$ denote the subvariety of $\mu_{HS}^{-1}(0)$ satisfying the conditions (S1) and (S2). It is also affine and its quotient 
\[
HS(\D):=\text{Spec}(\C[M_{HS}]^{G_{HS}})
\]
is called a handsaw variety. As shown in \cite[Prop. 2.9 and Cor. 2.21]{takayama_2016}, the $G_{HS}$ action on ${M}$ is free and all its orbits are closed, so $HS(\D)$ is just the orbit space ${M}_{HS}/G_{HS}$. The natural projection ${M}\to {M}_{HS}$ descends to an affine morphism $\rho: X_0(\D)\to HS(\D)$. Including the stack $\FX(\D)$ to our picture, we have a commutative diagram 
\begin{equation}
    \label{maps from bow to affine bow and handsaw}
    \begin{tikzcd}
    \FX(\D)\arrow[r, "p"] & X_0(\D) \arrow[r, "\rho"] & HS(\D)\\
    X(\D) \arrow[u, hookrightarrow]\arrow[ur, "\pi"]
\end{tikzcd}
\end{equation}

\subsection{Tangent space and tautological bundles}
\label{sec: taut bundles and tangent}

Let $X(\D)$ be a bow variety. By definition, a polarization for $X$ is a K-theory class $T^{1/2}X(D)\in K_{\Tt}(X(D))$ such that 
\begin{equation}
    \label{eq: alpha is pol}
    TX=T^{1/2}X(\D) + \hbar(T^{1/2}X(\D))^\vee\in K_{\Tt}(X(\D)).
\end{equation}
Let $\alpha=\alpha_{\D} \in K_{\Tt}(X(\D))$ be the class
\begin{equation}
\label{eq: class alpha for general bow variety}
\alpha:=\hbar \left( 
\bigoplus_{\Ab} \Hom( \tb_{\Ab^+}, \tb_{\Ab^-} ) \oplus \Hom( \C_{\Ab}, \tb_{\Ab^-} ) 
\oplus
\bigoplus_{\Zb} \Hom( \tb_{\Zb^-}, \tb_{\Zb^+} )
\ominus 
\bigoplus_{\tb} \Hom( \tb, \tb)
\right)^\vee
\end{equation}

It is proven in \cite{Shou} that for separated or co-separated $X(\D)$, there exists a class $\beta \in K_{\mathbb{C}^{\times}_{\hbar}}(\pt)$ such that $\alpha+\beta$ is a polarization in the sense of \eqref{eq: alpha is pol}.
Therefore, if $X(\D)$ is separated or coseparated, we will make a slight abuse of terminology and refer to $\alpha$ itself as a polarization.

\begin{remark}
This result together with the Hananay-Witten isomorphism, cf. \S\ref{subsec: HW iso}, implies that any bow variety (possibly not separated nor coseparated) admits a polarization. However, we stress that this polarization may not be directly related to the class $\alpha$ in \eqref{eq: class alpha for general bow variety} because the Hanany-Witten isomorphism is not $\Tt$ equivariant, but only equivariant along a given automorphism of $\Tt$, see \cite[\S3.3]{rimanyi2020bow}.
\end{remark}

\subsection{Mirror dual bow varieties}

In physics, bow varieties arise naturally as (resolutions of) Higgs and Coulomb branches of $3d$ $\mathcal{N}=4$ quiver gauge theories of type $A$ \cite{Nakajima_Takayama}. Throughout this article, we refer to a pair of bow varieties that describe, respectively, the Higgs and Coulomb branch of a given gauge theory as ``dual varieties'' and denote them by $X$ and $X^!$. From a combinatorial perspective, the dual of a bow variety $X=X(\D)$ is obtained by simply swapping NS5 and D5 branes in $\D$. For example $\D=\ttt{\fs 1\fs 2\bs 2\fs 2\fs 3\bs 3\bs 3\fs 2\bs 2\fs 1\bs 1\fs}$ and $\D^{!}=\ttt{\bs 1\bs 2\fs 2\bs 2\bs 3\fs 3\fs 3\bs 2\fs 2\bs 1\fs 1\bs}$ are 3d mirror dual brane diagrams and we have $X^!=X(\D^!)$.

\subsection{Fixed points of bow varieties and their mirror symmetry}\label{sec:fixedpoints}

Let $X(\D)$ be a bow variety with $m$ NS5 and $n$ D5 branes and consider the action of the torus $\At$. Its torus fixed points are isolated and can be combinatorially described as follows:
\begin{theorem}[{\cite{rimanyi2020bow}}]
\label{thm: fixed points bow}
    The $\At$-fixed points in $X(\D)$ are in one-to-one correspondence with any of the following equivalent combinatorial objects:
    \begin{itemize}
        \item Tie diagrams over $\D$.
        \item $m \times n$ binary contingency tables (BCTs) with sums of the rows (resp. columns) equal to the charge vector $r$ (resp. $c$).
        \item Butterfly diagrams over $\D$, which are certain colored acyclic quivers $B^{f}=(B^{f}_{0},B^{f}_{1},\col)$ with vertices $B^{f}_{0}$, arrows $B^{f}_{1}$, and coloring map $\col: B^{f}_{0} \to \{\text{D3 branes}\}\cong\{1,2,\ldots,m+n-1\}$ satisfying $|\col^{-1}(\Xb)|=d_{\Xb}$ for all D3 branes $\Xb$. In other words, there are exactly $d_{\Xb}$ vertices of $B^{f}$ of color $\Xb$.
    \end{itemize}
\end{theorem}

We refer to the article in \emph{loc. cit.} for the definitions of tie diagrams, BCTs, and butterfly diagrams. As an example, consider the bow variety $X(\D)$ where $\D=\ttt{\fs 2\bs 2\fs 2\bs 4\fs 3\fs 3\fs 4\bs 3\fs 2\bs 2\bs}$. Its charge vectors are $r=(2,1,1,2,3,2)$ and $c=(5,2,2,0,2)$. One fixed point in $X(\D)$ is represented by the tie diagram
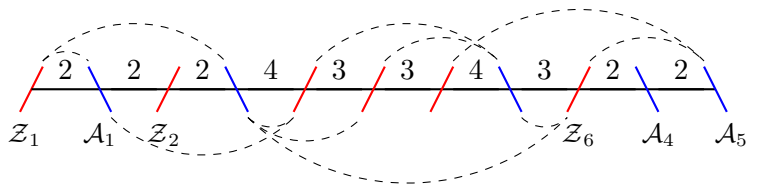
\begin{equation*}
\begin{tikzpicture}[baseline=0,scale=.3]
\draw [thick,red] (0.5,0) --(1.5,2); 
\draw[thick] (1,1)--(2.5,1) node [above] {$2$} -- (31,1);
\draw [thick,blue](4.5,0) --(3.5,2);  
\draw [thick](4.5,1)--(5.5,1) node [above] {$2$} -- (6.5,1);
\draw [thick,red](6.5,0) -- (7.5,2);  
\draw [thick](7.5,1) --(8.5,1) node [above] {$2$} -- (9.5,1); 
\draw[thick,blue] (10.5,0) -- (9.5,2);  
\draw[thick] (10.5,1) --(11.5,1) node [above] {$4$} -- (12.5,1); 
\draw [thick,red](12.5,0) -- (13.5,2);   
\draw [thick](13.5,1) --(14.5,1) node [above] {$3$} -- (15.5,1);
\draw[thick,red] (15.5,0) -- (16.5,2);  
\draw [thick](16.5,1) --(17.5,1) node [above] {$3$} -- (18.5,1);  
\draw [thick,red](18.5,0) -- (19.5,2);  
\draw [thick](19.5,1) --(20.5,1) node [above] {$4$} -- (21.5,1);
\draw [thick,blue](22.5,0) -- (21.5,2);
\draw [thick](22.5,1) --(23.5,1) node [above] {$3$} -- (24.5,1);  
\draw[thick,red] (24.5,0) -- (25.5,2); 
\draw[thick] (25.5,1) --(26.5,1) node [above] {$2$} -- (27.5,1);
\draw [thick,blue](28.5,0) -- (27.5,2);  
\draw [thick](28.5,1) --(29.5,1) node [above] {$2$} -- (30.5,1);   
\draw [thick,blue](31.5,0) -- (30.5,2);   

\draw [dashed, black](4.5,-.25) to [out=-45,in=225] (12.5,-.25);
\draw [dashed, black](10.5,-.25) to [out=-45,in=225] (12.5,-.25);
\draw [dashed, black](10.5,-.25) to [out=-45,in=225] (15.5,-.25);
\draw [dashed, black](10.5,-.25) to [out=-45,in=225] (24.5,-.25);
\draw [dashed, black](22.5,-.25) to [out=-45,in=225] (24.5,-.25);

\draw [dashed, black](1.5,2.25) to [out=45,in=-225] (3.5,2.25);
\draw [dashed, black](1.5,2.25) to [out=45,in=-225] (9.5,2.25);
\draw [dashed, black](13.5,2.25) to [out=45,in=-225] (21.5,2.25);
\draw [dashed, black](16.5,2.25) to [out=45,in=-225] (21.5,2.25);
\draw [dashed, black](19.5,2.25) to [out=45,in=-225] (30.5,2.25);
\draw [dashed, black](25.5,2.25) to [out=45,in=-225] (30.5,2.25);
\node at (.7,-1) {\small $\Zb_1$};
\node at (6.8,-1) {\small $\Zb_2$};
\node at (4,-1) {\small $\Ab_1$};
\node at (31.7,-1) {\small $\Ab_5$};
\node at (28.5,-1) {\small $\Ab_4$};
\node at (25,-1) {\small $\Zb_6$};

\end{tikzpicture},
\end{equation*}
or, equivalently, by the following BCT:
\[
\begin{tikzpicture}[baseline=1.4 cm, scale=.5]
\draw[ultra thin] (0,0) -- (5,0);
\draw[ultra thin]  (0,1) -- (5,1);
\draw[ultra thin]  (0,2) -- (5,2);
\draw[ultra thin]  (0,3) -- (5,3);
\draw[ultra thin]  (0,4) -- (5,4);
\draw[ultra thin]  (0,5) -- (5,5);
\draw[ultra thin]  (0,6) -- (5,6);
\draw[ultra thin]  (0,0) -- (0,6);
\draw[ultra thin]  (1,0) -- (1,6);
\draw[ultra thin]  (2,0) -- (2,6);
\draw[ultra thin]  (3,0) -- (3,6);
\draw[ultra thin]  (4,0) -- (4,6);
\draw[ultra thin]  (5,0) -- (5,6);
\node at (.5,6.4) {\small $5$}; \node at (1.5,6.4) {\small $2$}; \node at (2.5,6.4) {\small $2$}; \node at (3.5,6.4) {\small $0$}; \node at (4.5,6.4) {\small $2$};
\node at (.5,7.4) {\small $\Ab_1$}; \node at (1.5,7.4) {\small $\Ab_2$}; \node at (2.5,7.4) {\small $\Ab_3$}; \node at (3.5,7.4) {\small $\Ab_4$}; \node at (4.5,7.4) {\small $\Ab_5$};
\node at (-.5,0.5) {\small $2$}; \node at (-.5,1.5) {\small $3$}; \node at (-.5,2.5) {\small $2$}; \node at (-.5,3.5) {\small $1$}; \node at (-.5,4.5) {\small $1$}; \node at (-.5,5.5) {\small $2$};  
\node at (-1.8,0.5) {\small $\Zb_6$}; \node at (-1.8,1.5) {\small $\Zb_5$}; \node at (-1.8,2.5) {\small $\Zb_4$}; \node at (-1.8,3.5) {\small $\Zb_3$}; \node at (-1.8,4.5) {\small $\Zb_2$}; \node at (-1.8,5.5) {\small $\Zb_1$};  
\node[violet] at (0.5,0.5) {\small $1$};\node[violet] at (1.5,0.5) {\small $0$};\node[violet] at (2.5,0.5) {\small $0$};\node[violet] at (3.5,0.5) {\small $0$};\node[violet] at (4.5,0.5) {\small $1$};
\node[violet] at (0.5,1.5) {\small $1$};\node[violet] at (1.5,1.5) {\small $1$};\node[violet] at (2.5,1.5) {\small $0$};\node[violet] at (3.5,1.5) {\small $0$};\node[violet] at (4.5,1.5) {\small $1$};
\node[violet] at (0.5,2.5) {\small $1$};\node[violet] at (1.5,2.5) {\small $0$};\node[violet] at (2.5,2.5) {\small $1$};\node[violet] at (3.5,2.5) {\small $0$};\node[violet] at (4.5,2.5) {\small $0$};
\node[violet] at (0.5,3.5) {\small $0$};\node[violet] at (1.5,3.5) {\small $0$};\node[violet] at (2.5,3.5) {\small $1$};\node[violet] at (3.5,3.5) {\small $0$};\node[violet] at (4.5,3.5) {\small $0$};
\node[violet] at (0.5,4.5) {\small $1$};\node[violet] at (1.5,4.5) {\small $0$};\node[violet] at (2.5,4.5) {\small $0$};\node[violet] at (3.5,4.5) {\small $0$};\node[violet] at (4.5,4.5) {\small $0$};
\node[violet] at (0.5,5.5) {\small $1$};\node[violet] at (1.5,5.5) {\small $1$};\node[violet] at (2.5,5.5) {\small $0$};\node[violet] at (3.5,5.5) {\small $0$};\node[violet] at (4.5,5.5) {\small $0$};
\end{tikzpicture}.
\]
We also refer to \cite[Figure 3]{rimanyi2020bow} for the pictorial presentation of the corresponding butterfly diagram.

In the whole paper, we assume that $\D$ is such that $X(\D)$ has at least one fixed point. The Gale-Ryser theorem \cite[\S6.2.4]{BrualdiRyser} is a combinatorial characterization of the charge vectors for which this condition holds. It is expected that $X^{!}$ is empty (because there are no stable points) if $X$ has no torus fixed points. 

Fixed points of dual bow varieties are manifestly in one-to-one correspondence. Explicitly, given a bow variety $X$ and a fixed point $f\in X^{\At}$, we denote by $f^!\in (X^!)^{\At^!}$ the fixed point whose tie diagram is obtained by reflecting the tie diagram of $f$ along the $x$-axis of the page.

\section{Elliptic stable envelopes}\label{sec: stab}

\subsection{Recollection of elliptic cohomology}

Fix an elliptic curve $E=\Cs/q^{\ZZ}$. Throughout this article, we consider $G$-equivariant elliptic cohomology as a scheme-valued covariant functor
\[
\Ell_G(-):G\text{-spaces}\to \text{Schemes}_{\Ell_{G}}.
\]
We refer to \cite{aganagic2016elliptic, botta2021shuffle, BR} for details. Here, $\Ell_G:=\Ell_G(\pt)$ is the elliptic cohomology of the one-point space. If $G$ is either a rank $n$ torus $T$ or $\GL(n)$, then $\Ell_{T}=\text{Cochar}(T)\otimes_{\ZZ} E \cong  E^{n}$ 
and $\Ell_{GL(n)}= E^{(n)}$, the symmetrized $n$-{th} Cartesian power of $E$.

The Chern class of a rank $r$ vector bundle $V\to X$ is
 a morphism
 \begin{equation}
     \label{eq: chern class ell} c(V):\Ell_G(X)\to \Ell_{\GL(r)}(\pt)=E^{(r)},
 \end{equation}
called the elliptic characteristic class of $V$, see \cite[\S1.8]{ginzburg1995elliptic} and \cite[\S5]{Ganter_2014}. The coordinates in the target are called the elliptic Chern roots of $V$. Associated with $V$ is an invertible sheaf, the Thom sheaf of $V$
\[
\Theta(V):=c(V)^*\Sh{O}(D),
\]
obtained by pulling back the line bundle associated to the effective divisor $D=\lbrace0\rbrace+E^{(r-1)}\subset E^{(r)}$. We denote its canonical global section by $\vartheta(V)$, which is the fundamental example of an elliptic cohomology ``class". For computations, it is useful to pull-back $\vartheta(V)$ to $K$-theory via the Chern character map 
\[
\chern: \Spec(K_G(X)\otimes \mathbb{C})^{\text{an}}\to \Ell_G(X).
\]
Indeed, the pullback $\chern^*\Theta(V)$ is trivial for all $V$ and hence we can treat $\vartheta(V)$ as a function on $ \Spec(K_G(X) \otimes \mathbb{C} )^{\text{an}}$. In the case $X=\pt$ and $G=\Cs$ the Chern character is simply given by the quotient map $\Cs\to E=\Cs/q^{\ZZ}$, and we can choose a trivialization of $\chern^*\Sh{O}(0)$ that identifies its canonical section with the odd Jacobi theta function
\[
    \vartheta(x)=(x^{1/2}-x^{-1/2})\prod_{n>0} (1-q^nx)(1-q^nx^{-1}).
\]
Its quasi-periods are 
\begin{equation}
    \label{eq: quasi-period theta function}
    \vartheta(q^kx)= (-1)^k q^{-k^2/2}x^{-k}\vartheta(x) \qquad k\in \ZZ.
\end{equation}
It follows from the compatibility of the Chern character with the map \eqref{eq: chern class ell} that the trivialization of an arbitrary line bundle $\chern^*\Theta(V)$ can be uniquely determined by the previous universal case. With this convention, the pullback of $\vartheta(V)$ by $\chern$ is the function
\[
\prod_{i=1}^r\vartheta(t_i)
\]
in the Chern roots $t_i$ of $V$.

Fix now a one dimensional torus torus $\Cs_z$ acting trivially on $X$ and a vector bundle $V$. We define 
\begin{equation}
\label{eq: U-bundle}
    \Sh{U}(V,z):=\Theta\left((V-1^{\oplus \rk(V)})(z-1)\right)\in \Pic{}{\Ell_{G\times \Cs_z}(X)}.
\end{equation}
It admits a canonical meromorphic section, which is the pullback by $c(V)$ of the canonical section
\begin{equation}
    \label{eq: delta function}
    \prod_{i=1}^{r}\delta(t_i,z)=\prod_{i=1}^r \frac{\vartheta(t_iz)}{\vartheta(t_i)\vartheta(z)}
\end{equation}
of the line bundle $\bigotimes_{i=1}^r\Theta((t_i-1)(z-1))$ on $\Ell_{\GL(r)\times \Cs_{z}}(\pt)$. Notice that 
\begin{equation}
    \label{eq: quasiperiod delta}
     \delta(q^kx, y)= y^{-k}\qquad \delta(x,q^k y)= x^{-k}\qquad k \in \ZZ.
\end{equation}
A list of properties of $\Sh{U}(V,z)$ can be found in \cite[App. 1]{okounkov2020inductiveI} and \cite[Lemma 4.1]{BR}.

\subsection{Functorality}\label{ell functoriality}

Given a $G$-variety $X$, its elliptic cohomology $\Ell(X)$ can be defined as the relative spectrum of a sheaf of algebras $\Sh{O}_{\Ell(X)}$ over the variety $\Ell_G$, cf. \cite{Grojnowski2007EllipticCD}. Such sheaf of algebras can be thought as the analog of the (singular) cohomology ring $H_G(X)$, which is an algebra over $H_G(\pt)$. The standard singular cohomology $H_G(\pt)$ is contravariant with respect to any equivariant continuous map $f: X\to M$ or a change of group map $\varphi: H\to G$, and covariant for a proper complex oriented map $g: X\to M$. Elliptic cohomology shares similar properties, which are best to rephrase in sheaf theoretic language. Explicitly, given a continuous $G$-equivariant map $f: X\to Y$, and hence induced maps of schemes
\[
\begin{tikzcd}
    \Ell_G(X)\arrow[rr, "\Ell_G(f)"] \arrow[dr, swap, "q_X"]& & \Ell_G(Y)\arrow[dl, "q_Y"]
    \\ 
    & \Ell_G & 
\end{tikzcd}
\]
the adjunction $id \to (\Ell_G(f))_* (\Ell_G(f))^*$ induces a morphism of $\Sh{O}_{\Ell_G}$-modules
\[
f^{\oast}: (q_{Y})_{*}\Sh{F}\to (q_X)_*\Ell_G(f)^*\Sh{F},
\]
Here, $\Sh{F}$ is an arbitrary quasi-coherent sheaf of $\Ell_G(Y)$. 
This is the elliptic analog of pullback in singular equivariant cohomology. 

The elliptic change of group map is defined similarly: for every homomorphism $\varphi: H\to G$ and $G$-variety $X$ we have a functorial map $r: \Ell_H(X)\to \Ell_G(X)$ and the adjunction 
\[
\varphi^{\oast}: \Sh{F}\to r_*r^*\Sh{F}
\]
is the change of group map in elliptic cohomology. Finally, the pushforward in elliptic cohomology associated to a proper map $g: X\to Y$ and a locally free sheaf $\Sh{E}$ on $\Ell_G(Y)$ is a morphism 
\[
g_{\oast}: \Ell_G(g)_*(\Theta(TX-f^*TY)\otimes \Ell_G(g)^* \Sh{E})\to \Sh{E}.
\]
We refer to \cite[\S4]{BR} for a more detailed treatment.
Here, we only make the following important observation: if $p:X\to \pt $ is proper, then pushforward in elliptic cohoology gives a well-defined map
\[
\Theta(TX)\xrightarrow[]{p_{\oast}} \Sh{O}_{E_G(\pt)}.
\]
As a consequence, for any line bundle $\Sh{L}$ on $E_G(X)$, it is natural to define its dual by $\Sh{L}^\triangledown:= \Sh{L}^{-1}\otimes \Theta(TX)$ \cite{okounkov2020inductiveI}. Then, whenever $p:X\to \pt $ is proper (or, more generally, equivariantly proper, via localization),  we get a natural map
\[
\Sh{L}^\triangledown\otimes \Sh{L}=\Theta(TX)\xrightarrow[]{p_{\oast}} \Sh{O}_{E_G(\pt)}.
\]

\subsection{Elliptic stable envelopes--basics}\label{sec: attracting loci}

In the following two sections, we recall the definition and the main features of stable envelopes on bow varieties, cf. \cite{BR, okounkov2020inductiveI}. 
Let $X$ be a smooth variety equipped with the action of a torus $T$. Fix a subtorus $A\subset T$. The stable envelopes are certain distinguished correspondences depending on a choice of generic cocharacter $\lambda: \Cs\to A$, i.e. a cocharacter such that $X^{\lambda(\Cs)}=X^A$. Notice that a character is not generic if and only if it lies on some codimension one hyperplane in $\Cochar(A)\otimes_{\ZZ}\RR$ given by the zero locus of some $A$-weight of the normal bundle of $X^{A}$. We denote a connected component of the resulting chamber arrangement in $\Cochar(A)\otimes_{\ZZ}\RR$ by $\chamb$.

Intuitively, stable envelopes encode the geometry of the attractors 
\[
\Att_{\chamb}(X^{\Cs})=\{x\in X\; |\; \lim_{t\to 0} \lambda(t) \cdot x \in X^{A} \text{ for any (equiv. every) } \lambda \in \chamb\} .
\]

Notice that the latter is reducible. Its irreducible components are given by the closures of the subsets $\Att_{\chamb}(F)$, where $F\subset X^{A}$ is a connected component of $X^{A}$. In fact, $\Att_{\chamb}(F)$ is an affine bundle over $F$. If the variety $X$ admits a proper $A$-equivariant map to an affine variety $X\to X_0$, then the transitive closure of the relation 
\[
\Att_{\chamb}(F)\cap F'\neq \emptyset \implies F\geq F'
\]
induces a partial order on $\pi_0(X^{\Cs})$, sometimes known as a ``generalized Bruhat order''. 

For a fixed component $F$, consider the three $T$-invariant subvarieties
\begin{equation}
    \label{eq: spaces in the definition of stable envelopes}
    \begin{split}
    & \Att_{\chamb}^{\leq}(F):= \bigcup_{F'\leq F} F\times \Att_{\chamb}(F'),
    \qquad X^{\geq F}=X\setminus \bigcup_{F'< F} \Att_{\chamb}(F'), 
    \\
    & X^{> F} = X\setminus \bigcup_{F'\leq F} \Att_{\chamb}(F').
    \end{split}
\end{equation}
Notice that $X^{> F}\subseteq X^{\geq F}$ and the inclusion $j:  \Att_{\chamb}(F)\hookrightarrow F\times X^{\geq F}$ is well defined and closed. Informally speaking, the stable envelope of a fixed component $F$ is a specific correspondence on $F\times X$ that is supported on $\Att_{\chamb}^{\leq}(F)$ and restricts to the graph of $j$ on $F \times X^{\geq  F}$. Technically, to make this precise in elliptic cohomology one needs to fix a distinguished line bundle $\Sh{L}_{X}$ on $\Ell_T(X)$, known as an attractive line bundle. By definition, it is a line bundle such that 
\begin{equation}
    \label{defining equation attractive line bundle}
    \deg_A(\Ell_T(i_F)^*\Sh{L}_X)=\deg_A(\Theta(N^-_{F/X})),
\end{equation}
for all components $F\subset X^A$, where $i_{F}:F \hookrightarrow X$ is the inclusion. See \cite[App. B.3]{okounkov2020inductiveI} for the definition of $\deg_A(\cdot)$. 

\subsection{Elliptic stable envelopes--definition}

Let 
$ 
\Sh{L}_{A,F}=i_F^*\Sh{L}_X\otimes \Theta(-N_{F/X}^-)
$ and consider the exterior tensor product 
\[
(\Sh{L}_{A,F})^{\triangledown}\boxtimes \Sh{L}_X,
\]
which is a bundle bundle on the elliptic cohomology scheme $\Ell_T(F\times X)$. Here, the upperscript in $(\Sh{L}_{A,F})^{\triangledown}$ is the duality functor defined in \S\ref{ell functoriality}.

The following is Okounkov's definition of the elliptic stable envelope.

\begin{theorem} [{{\cite{okounkov2020inductiveI}}}]
\label{thm: definition elliptic stab}
Let $\mC$ be a chamber for the action of $A$ on $X$ and let $\Sh{L}_X$ be an attractive line bundle. For every connected component $F\subset X^{A}$, there exists a unique section
\[
\Stab_{\chamb}(F)\in \Gamma\left((\Sh{L}_{A,F})^{\triangledown}\boxtimes \Sh{L}_X(\infty \Delta)\right)
\]
proper over $X$, such that

\begin{enumerate}
    \item[(i)] (The diagonal axiom) The restriction of $\Stab_{\chamb}(F)$ on $F\times X^{\geq F}$ is equal to $\left[ \Att_{\chamb}(F)\right]$.
    \item[(ii)] (The support axiom) The restriction of $\Stab_{\chamb}(F)$ to $F\times X\setminus \Att_{\chamb}^\leq(F)$ is zero.
\end{enumerate}
\end{theorem}
 The notation $\infty\Delta$ stands for a singular behavior of the stable envelope on the ``resonant locus'' $\Delta\subset \Ell_T(\pt)$. Informally, sections of $\Sh{G}(\infty\Delta)$ are sections of $\Sh{G}$ with arbitrary singularities on $\Delta$. For the precise definition and the geometric interpretation of $\Delta$, we refer to \cite[
\S2.3]{okounkov2020inductiveI} and \cite[\S5.2]{BR}. For stable envelopes of bow  varieties, which are the main focus of the paper, $\Delta$ will be a divisor of $\Ell_T(\pt)$.

Consider now the canonical maps
\[
\begin{tikzcd}
    \Ell_T(F) & \Ell_T(X^A\times X)\arrow[l, swap, "p_A"]\arrow[r, "p"] & \Ell_T(X).
\end{tikzcd}
\]
Since $\Stab_{\chamb}(F)$ proper over $X$, it induces a morphism via convolution (cf. \cite[\S4.7]{BR})
\begin{equation*}
    \label{elliptic stable enveliopes map}
    \Stab_{\chamb}(F):(p_A)_{\ast}(\Sh{L}_{A,F})\to p_{\ast}\Sh{L}_X(\infty \Delta) \qquad \gamma\mapsto p_{\oast}\left(\Stab_{\chamb}(F)\otimes (p_A)^{\oast}(\gamma)\right)
\end{equation*}
which is commonly denoted in the same way.

\begin{remark}
\label{rem: full att set}
    The space $\Att_{\chamb}^{\leq}(F)$ is a singular closed subspace of $F\times X$. It contains the so-called full attracting set $\Attfull{C}(F)$, which is also $T$-invariant and is defined as the set of pairs $(f,x)\in F\times X$ connected via a chain of closures of attracting orbits. It is shown in \cite[\S2.5]{okounkov2020inductiveI} that the stable envelopes are actually supported on $\Attfull{C}(F)$.
\end{remark}

\subsection{Stable envelopes of bow varieties}

Let $X$ be a bow variety with $m$ NS5 branes (ordered from left to right), $n$ D5 branes, and the torus $\Tt=\At\times \Cs_{\hbar}$ acting on it. We introduce the K\"ahler torus $Z=(\Cs)^{m}$ and let $(z_1,\dots z_m)$ be a set of coordinates spanning $Z$. We also set $T=\Tt\times Z=\At\times \Cs_{\hbar}\times Z$. It may be beneficial to think about the subtorus $\At\times Z$ as the ambient space of two tuples of parameters $(a_1, \dots a_n)$ and $(z_1,\dots z_m)$ that are attached to the similarly named five-branes.

Set $\Base:=\Ell_{\Tt\times Z}(\pt)$ and let $\eta_k$ be one (any) of the bundles in between the NS5 branes $\Zb_k$ and $\Zb_{k+1}$. The construction will eventually be independent of such choice. Consider the line bundle on $\Ell_{\Tt\times Z}(X)$ given by
\begin{align}
\label{eq: ubund}
\Sh{U}:=\bigotimes_{k=1}^{m-1} \Sh{U}\left(\eta_k,\frac{z_k}{z_{k+1}}\hbar^{\ell(\Zb_{k+1})-\ch(\Zb_k)}\right)
\end{align}
where $\Sh{U}(\eta, x)$ is the line bundle defined in \eqref{eq: U-bundle} and the integers $\ell(\Zb_{k})$, $\ch(\Zb_k)$ were defined in \S\ref{sec: brane diagrams}. Recalling the definition of the polarization $\alpha$, see \S\ref{sec: taut bundles and tangent}, we further define 
    \begin{equation}
    \label{good line bundle bow variety}
    \Sh{L}_X:=\Theta(\alpha)\otimes \Sh{U}.
\end{equation}

\begin{proposition}[{\cite[\S5.5]{BR}}]
\label{prop: line bundle is attractive}
    The line bundle $\Sh{L}_X$ is attractive for the action of $\At$ on $X$. 
    Moreover, the following statements hold:
    \begin{enumerate}
        \item Assume that $X_1\cong X_2$ are Hanany-Witten isomorphic bow varieties and let $F_1\subset (X_1)^{\At}$ and $F_2\subset (X_2)^{\At}$ be isomorphic fixed components.
        Then we have an isomorphism 
        \[(\Sh{L}_{A,F_1})^{\triangledown}\boxtimes \Sh{L}_{X_1}(\infty \Delta)\cong (\Sh{L}_{A,{F_2}})^{\triangledown}\boxtimes \Sh{L}_{X_2}(\infty \Delta).
        \]
        \item The line bundle $\Sh{L}_X$ is non-degenerate, i.e. the complement of resonant locus $\Delta$ in $\Ell_{ \Cs_{\hbar}\times Z}(\pt)$ is open and dense. Furthermore, $\Delta$ is contained in the complement of the union of codimension-one abelian varieties of the form $\{z_i/z_j\hbar^{\alpha_{ij}}=1\}$ for certain $\alpha_{ij}\in \ZZ$.
    \end{enumerate}
\end{proposition}

The following result is a direct consequence of Okounkov's general theory and the previous proposition. 

\begin{theorem}[{\cite[Thm. 5.10]{BR}}]
\label{thm: stab exists}
Fix an arbitrary subtorus $A\subseteq \At$, a chamber $\chamb$ for the $A$-action, and a connected component $F\subset X^{A}$. Then the stable envelope
\[
\Stab_{\chamb}(F)\in \Gamma( (\Sh{L}_{A,F})^{ \triangledown}\boxtimes \Sh{L}_X(\infty\Delta))
\]
exists and is unique. Furthermore, The elliptic stable envelopes of Hanany-Witten equivalent bow varieties match. 
\end{theorem}
\subsection{Matrix stable envelope}

Let $X$ be a bow variety with $m$ NS5 branes and $n$ D5 branes. We now focus on the action of the torus $\At$ acting on all $n$ D5 branes. The tangent weights for the $\At$ action on $X$ are of the form $a_i/a_j$, with $1 \leq i,j \leq n$ \cite[\S4]{rimanyi2020bow}. 
Hence, a permutation of the $a_i$ variables determines a chamber. By slight abuse of notation we will write 
$\lbrace a_{\sigma(1)}< \dots <a_{\sigma(n)}\rbrace$ for the chamber it determines, in which weights of the form $a_{\sigma(i)}/a_{\sigma(j)}$ for $i<j$ are repelling. The chamber
$\lbrace a_{1}< \dots <a_{n}\rbrace$ associated to the identity permutation will be called the {\em standard chamber} for the action of $\At$.

As discussed in \S\ref{sec:fixedpoints}, the fixed locus $X^{\At}$ is finite. Given an arbitrary chamber $\chamb$ for the $\At$-action, we define the matrix of  stable envelope restrictions $S_{\chamb}$ by setting\footnote{This matrix differ by a transposition and a normalization from the one considered in \cite{BR}.}
\begin{equation}
    \label{stable envelope restriction}
   (S_{\chamb})_{ij}:=  (S_{\chamb})_{ij}:=\frac{\Stab_{\chamb}(f_i)|_{f_j}}{\Stab_{\chamb}(f_j)|_{f_j}}
\end{equation}

If the fixed points are ordered consistently with the partial order defined in \S\ref{sec: attracting loci}, namely if $f_i\leq f_j$ implies $i\leq j$, then $(S_{\chamb})_{ij}$ is an upper triangular matrix whose elements are sections of some bundle on the abelian variety
\[
 E_{\Tt\times Z}=E^{n}\times E^{m}\times E_{\hbar},
\]
and hence can be seen as functions in the variables $(a,z,\hbar)$ with prescribed quasi-periods.

\subsection{Duality of stable envelopes}
As explained in \cite[\S3.7]{aganagic2016elliptic}, the inverse of the matrix stable envelope can be characterized geometrically. In view of applications of this characterization to mirror symmetry of vertex functions, it is convenient to introduce a slightly different normalization for the stable envelopes. First, given an arbitrary bow variety $X$, we define the following line bundles on $\Ell_{\Tt\times Z}(X)$:
\begin{align*}
    \Sh{U}^{\opp}&:=\bigotimes_{k=1}^{m-1} \Sh{U}\left(\eta_k,\frac{z_k}{z_{k+1}}\hbar^{-\ell(\Zb_{k+1})+\ch(\Zb_k)}\right)
    \\
    \Sh{L}^{\opp}_X&:=\Theta(\alpha_{\opp})\otimes \Sh{U}^{\opp} 
\end{align*}
Here, $\alpha_{\opp}=\hbar \alpha^\vee$ is the dual of the class $\alpha$ defined in \eqref{eq: class alpha for general bow variety}. 
We leave as an exercise the check that Theorems \ref{prop: line bundle is attractive} and \ref{thm: stab exists} hold after replacing the line bundle $\Sh{L}_X$ with $\Sh{L}_X^{\opp}$. Hence there is a well defined stable envelope map 
\[
\Stab^{\opp}_{\chamb}(F)\in \Gamma( \Sh{L}^{\opp, \triangledown}_{A,F}\boxtimes \Sh{L}^{\opp}_X)_{mer}
\]
that is invariant under Hanany-Witten transition. In agreement with the general logic of elliptic stable envelopes, this map differs from $\Stab_{\chamb}(F)$ by a shift of the K\"ahler parameters. The advantage of this alternative definition is that, setting 
\begin{equation*}
    \label{stable envelope opp restriction}(S^{\opp}_{\chamb})_{ij}:=\frac{\Stab^{\opp}_{\chamb}(f_i)\Big|_{f_j}}{\Stab^{\opp}_{\chamb}(f_j)\Big|_{f_j}}
\end{equation*}
we obtain:
\begin{proposition}
\label{prop: duality stable envelopes}
    Let $\mathfrak{C}^{\opp}$ be the chamber opposite to $\mathfrak{C}$. Then 
    \[
    \left(S_{\chamb}(a,z,\hbar)\right)^{-1}=\left(S^{\opp}_{\chamb_{\opp}}(a,z^{-1},\hbar)\right)^T. 
    \]
    Here $(-)^T$ stands for matrix transposition.
\end{proposition}
\begin{proof}
    By Hanany-Witten invariance of both $S_{\chamb}$ and $S^{\opp}_{\chamb_{\opp}}$, it suffices to assume that the bow variety $X$ is separated. Then by \cite[Prop. 5.15]{BR} (and taking into account the different conventions), we have 
    \[
    \left(S_{\chamb}(a,z,\hbar)\right)^{-1}=\left(S_{\chamb_{\opp}}(a,z^{-1}\hbar^{r},\hbar)\right)^T. 
    \]
    Therefore, it suffices to check that $S^{\opp}_{\chamb_{\opp}}(a,z^{-1},\hbar)=S_{\chamb_{\opp}}(a,z^{-1}\hbar^{r},\hbar)$ for any chamber $\chamb$. This follows from uniqueness of stable envelopes, the definitions of $\Sh{L}_X$ and $\Sh{L}^{\opp}_X$, and the fifth statement of \cite[Lemma 4.1]{BR}.
\end{proof}


\subsection{Mirror symmetry for stable envelopes}

Consider an arbitrary bow variety $X$ with $m$ NS5 branes and $n$ D5 branes and its dual variety $X^!$. In this section, the choice of the standard chamber is always understood, and hence the symbol $\chamb$ is dropped from the notation. For instance, we will write $\Stab^{X}(f)$ as a shortcut for the stable envelope $\Stab^{X}_{\chamb}(f)$ of $X$ with standard chamber $\chamb=\lbrace a_1<\dots<a_n\rbrace $.

For a fixed point $f\in X^{\At}$, let the sign $\varepsilon_X(f)$ be defined as $(-1)$ to the power of
\[
\sum_{i=1}^n\sum_{j=1}^m\sum_{\substack{k>i\\ l>j}} b_{ij}b_{kl},
\]
where $b_{ij}$ are the entries of the BCT table of $f$ introduced in \S\ref{sec:fixedpoints}. Mirror symmetry for stable envelopes of bow varieties, which is the main result of \cite{BR}, states that:
\begin{theorem}
\label{thm: mirror symmetry stabs}
    Mirror symmetry for stable envelopes of bow varieties holds, i.e.
    \begin{equation}
    \label{3d mirror symmetry formula}
    \frac{\Stab^{X}(f)|_g}{\Stab^{X}(g)|_g}(a,z,\hbar)=\varepsilon_X(f)\varepsilon_X(g)\frac{\Stab^{X^!}(g^!)|_{f^!}}{\Stab^{X^!}(f^!)|_{f^!}}(z,a,\hbar^{-1})
    \end{equation}
    for any bow variety $X$ and fixed points $f,g\in X^{\At}$.
\end{theorem}

By the definition of \eqref{eq: ubund}, it follows that $\Stab$ depends only on the ratios $Q_i=z_i/z_{i+1}$ and $u_i=a_i/a_{i+1}$. Therefore, whenever convenient, we will think of the function $\Stab(f)|_g$ as a function in the variables $(u,Q,\hbar)$ and denote is by 
\[
\Stab(f)|_g(u, Q, \hbar).
\]
To compare mirror symmetry of stable envelopes to mirror symmetry of vertex functions, we need a particular normalization of stable envelopes. To this end, set 
\begin{equation}\label{eq: def of msstab}
    \MSstab_{gf}(u, Q,\hbar):=\left(\frac{\det (\alpha)|_{f}}{\det(N^-_{f})}\right)^{1/2}\frac{\Stab(f)|_{g}}{\Stab(g)|_{g}}(u, Q,\hbar)\left(\frac{\det (\alpha)|_{g}}{\det(N^-_{g})}\right)^{-1/2}
\end{equation}
and consider the change of variable map $\mirmap$ given by 
\[
q^{!} \mapsto q\qquad \hbar^{!} \mapsto \frac{1}{\hbar q} \qquad (q^{!})^{-\w(\Zb_{i}^{!})} Q_{i}^{!}\mapsto u_{i}\qquad u_{i}^{!} \mapsto  q^{\w(\Zb_{i})} Q_{i}.
\]

\begin{lemma}
   \label{lemma: mirror symmetry stabs rephrasing} 
   Theorem \eqref{thm: mirror symmetry stabs} is equivalent to 
   \[
   \mirmap \left( \MSstab^{X^{!}}_{f^{!}g^{!}}\right)= \MSstab^X_{gf}.
   \]
   \begin{proof}
       The proof follows from Theorem \ref{thm: mirror symmetry stabs}, formula \eqref{eq: def of msstab}, and the quasi-period of the stable envelopes under $Q$-shifts. To compute the latter, it is useful to notice that $\kappa(\Stab(g^{!})|_{f^!}/\Stab(f^{!})|_{f^!})$ has the same quasi-periods as
       \[
       \mirmap^{-1}\left(\Theta(\alpha^!|_{f^!}-N^{-}_{f^!})\otimes \Theta(\alpha^!|_{g^!}-N^{-}_{g^!})^{-1}\otimes \Sh{U}_{f^!}\otimes \Sh{U}_{g^{!}}^{-1}\right)
       \]
       (by definition) and also as
       \[
       \Theta(\alpha|_{g}-N^{-}_{g})\otimes \Theta(\alpha|_{f}-N^{-}_{f})\otimes \Sh{U}_{g}\otimes \Sh{U}_{f}^{-1}
       \]
       (by Theorem \ref{thm: mirror symmetry stabs}). Details are left to the reader.
   \end{proof}
\end{lemma}

\section{Quasimaps and vertex functions}\label{sec: quasimaps}

In this section, we will define and study vertex functions of bow varieties. We will calculate them explicitly, determine their invariance under Hanany-Witten transition, and study various analytic properties.

\subsection{Quasimaps to bow varieties}

Let $M$ be a smooth affine variety equipped with the action of a reductive group $G$. For a fixed character $\chi:G\to \Cs$, let $X=M\git^{\chi} G$ be the associated GIT quotient. The prototypical examples in this article are bow varieties. We denote the associated semistable (resp. unstable) locus by $M^{\chi\semis}$ (resp. $M^{\chi\un})$.

Fix a nonsingular complex projective curve $C$ and consider the mapping stack
\[
\text{Map}(C, M/G)
\]
The moduli space of quasimaps from $C$ to $M\git^{\chi}G$ is by definition the open substack
\[
\qm(C,M\git^{\chi}G)\subset \text{Map}(C, M/G)
\]
consisting of those maps $f: C\to M/G$ such that the preimage $f^{-1}(M^{\chi\un}/G)$ is zero dimensional. In other words, $f$ must land in the semistable locus $M^{\chi\semis}/G$ for all but finitely many points. Under some mild assumptions on $M$ and $G$, $\qm(C,M\git^{\chi}G)$ is a Deligne-Mumford stack \cite{qm}. 

We remark that the moduli space $\qm(C,M\git^{\chi}G)$ depends on the stack quotient $M/G$ and not only on the GIT quotient $M\git^{\chi}G$. In particular, it depends on the GIT presentation of the quotient. Since in this article we are only interested in the case $C=\bbP^1$, we set
\[
\qm(M\git^{\chi}G):=\qm(\bbP^1, M\git^{\chi}G).
\]

If $M\git^{\chi}G$ is a bow variety, there is a natural action of the torus $\Tq:=\Tt\times \Cs_q$ on $\qm(M\git^{\chi}G)$, induced by the action of the torus $\Tt$ on the target space and the maximal torus $\Cs_q\subset \text{Aut}(\bbP^1)$ on the source curve $\bbP^1$.

\subsection{Obstruction theory and virtual structure}

From now on, we assume that $X=M\git^{\chi} G$ is a bow variety, and we will use the notation from \S\ref{sec: bow}. The moduli space $\qm(X)$ fits in a natural diagram 
\begin{equation}
    \label{eq: diagram universal evaluation}
    \begin{tikzcd}
    \bbP^1\times \qm(X) \arrow[r, "u"] \arrow[d, "pr"] & \FX\\
    \qm(X)
\end{tikzcd}
\end{equation}
where the map $u$ is the universal evaluation. The perfect obstruction theory is then given by the complex 
\begin{equation}
    \label{eq: POT}
    \EE=\left(R^\bullet pr_* u^*\Tt_{\FX}\right)^\vee,
\end{equation}
where the $G$-equivariant perfect complex
\begin{equation}
    \label{eq: perfect complex stack quotient}
    \Tt_{\FX}=\left[\Sh{O}_{\MM}\otimes \Lie(G) \xrightarrow{d\alpha }T_{\MM}\xrightarrow{d\mu} \Sh{O}_{\MM}\otimes \NN \right]\Big|_{{M}}
\end{equation}
is the tangent complex of the stack $\FX=[M/G]$.  
Here, $\alpha: G\to \text{Aut}(\MM)$ is the canonical map.

Explicitly, a quasimap to $X$ consists of the data $(\Sh{W},s)$, where $\Sh{W}=(\Sh{W}_{\Xb})_{\Xb}$, $\Sh{W}_{\Xb}$ is a rank $d_{\Xb}$ bundle on $\bbP^1$ for each D3 brane $\Xb$, and $s$ is a global section of the bundle $\Sh{M}$ over $\bbP^1$ obtained by replacing the vector spaces $W_{\Xb}$ in \eqref{eq: MM} with the bundles $\Sh{W}_{\Xb}$. By definition, $s$ is required to satisfy the relations of the map $\mu$ from \S\ref{sec:def of bow variety} fiberwise. In other words, letting $\Sh{N}$ be the $\bbP^1$-bundle obtained by replacing the vector spaces $W_{\Xb}$ in \eqref{eq: NN} with the bundles $\Sh{W}_{\Xb}$, the map $\mu$ induces a vector bundle morphism $\Sh{M}\to \Sh{N}$ and $s$ is required to be in its kernel.

For any D3 brane $\Xb$ in the brane diagram we have a tautological bundle $\tb_{\Xb}$ on $\FX$. Pulled back along the inclusion $X\subset \FX$, these restrict to the tautological bundles defined in \S\ref{sec:def of bow variety}. Accordingly, the universal evaluation $u$ induces, via pullback, tautological bundles 
\[
\Xi_{\Xb}=u^*\tb_{\Xb}
\]
on $\qm(X)$ satisfying $\Xi_{\Xb}|_{\bbP^1\times \lbrace(\Sh{W},s)\rbrace}=\Sh{W}_{\Xb}$.

We further consider the tautological bundles $\UM$ and $\UN$ on $\bbP^1\times \qm(X)$ such that 
\[
\UM|_{\bbP^1\times \lbrace(\Sh{W},s)\rbrace}=\Sh{M}\qquad \UN|_{\bbP^1\times \lbrace(\Sh{W},s)\rbrace}=\Sh{N}
\]

From \eqref{eq: POT} it follows that the virtual tangent space at the point $(\Sh{W},s)$ is given by 
\begin{equation}\label{eq: qm tvir}
T_{\vir}|_{(\Sh{W},s)}=H^\bullet (\Sh{M}) -\hbar H^\bullet (\Sh{N}) -H^\bullet(\oplus_{\Xb}\End(\Sh{W}_{\Xb})).
\end{equation}

\subsection{Vertex functions for bow varieties}

Let $X=M\git^{\chi} G$ be a bow variety and fix $p \in (\mathbb{P}^{1})^{\Cs_q}$. We denote by $\qm_{\ns p}(X)$ the moduli space of quasimaps from $\mathbb{P}^{1}$ to $X$ nonsingular at $p$. By definition, this is defined via the pullback diagram 
\[
\begin{tikzcd}
    \qm_{\ns p}(X) \arrow[r, hookrightarrow] \arrow[d, "\ev_p"]&\qm(X)\arrow[d, "\ev_p"] \\
    X=M^{\chi\semis}/G\arrow[r, hookrightarrow]& M/G
\end{tikzcd}
\]
Consider now the commutative diagram
\[
\begin{tikzcd}
    \qm_{\ns p}(X) \arrow{r}{\ev_{p}} & X \\
    \qm_{\ns p}(X)^{\mathbb{C}^{\times}_{q}} \arrow[hookrightarrow]{u} \arrow{ur}
\end{tikzcd}
\]
where the horizontal map is given by evaluation at $p$. We will abuse notation and denote the diagonal arrow also by $\operatorname{ev}_{p}$. It is known to be proper \cite{Okounkov_lectures} and hence the pushforward $\operatorname{ev}_{p,*}$ is well-defined in localized $K$-theory. For bow varieties, we could also define the pushforward by localization with respect to the (larger) torus $\Tq$ using Proposition \ref{prop: fixedqm} below.

The degree of a quasimap $(\Sh{W},s)$ is defined to be $(\deg \Sh{W}_{\Xb_{k}})_{1\leq k\leq m-1} \in \mathbb{Z}^{m-1}$ where $m$ is the number of NS5 branes and $\Xb_k$ is one of the D3 branes between the NS5 branes $\Zb_{k}$ and $\Zb_{k+1}$. It does not depend on the choice of $\Xb_{k}$. This gives a decomposition
\[
\qm_{\ns p}(X) = \bigsqcup_{d\in \mathbb{Z}^{m-1}} \qm^{d}_{\ns p}(X),
\]

The perfect obstruction theory \eqref{eq: POT} gives rise to a virtual structure sheaf $\mathcal{O}_{\operatorname{vir}}^{d}$ on $\qm^{d}(X)$ and hence on the open substack $\qm^{d}_{\ns p}(X)$. To find tractable structures in $K$-theoretic curve counts, it is important to twist the virtual structure sheaf by a certain line bundle. We refer to \cite{Okounkov_lectures} for a detailed discussion of this and for several examples of ``rigidity" which require the symmetrization. 

Let $\pol$ be a polarization of $X$. We assume that $\pol$ can be expressed in terms of the tautological bundles $\tb_{\Xb}$ and we fix such a representative. This induces a $K$-theory class $\qmpol$ on $\bbP^{1} \times \qm_{\ns p}(X)$.

By definition,
\begin{equation}\label{symvss}
\vrs^{d}:=\mathcal{O}_{\operatorname{vir}}^{d} \otimes \left(\Can_{\operatorname{vir}} \frac{\det \qmpol|_{p}}{\det 
\qmpol|_{p'}} \right)^{1/2}
\end{equation}
where 
\[
\Can_{\operatorname{vir}}:=(\det T_{\vir})^{-1} \otimes  \det \left(\operatorname{ev}_{p}^{*} T X\right)
\]
is the (normalized) virtual canonical bundle on $\qm_{\ns p}^{d}(X)$ and $\{p,p'\}=(\bbP^{1})^{\Cs_{q}}$. 

We introduce formal variables $Q_{i}$ for $1 \leq i \leq m-1$. For $d \in \mathbb{Z}^{m-1}$, we will use the multidegree notation $Q^{d}:=\prod_{i=1}^{m-1} Q_{i}^{d_i}$. It will also be convenient for us later on to use variables $z_i$ for $1 \leq i \leq m$ such that $Q_{i}=\frac{z_i}{z_{i+1}}$.

\begin{definition}
The vertex function of $X$ is the formal power series
\[
\ver^{X}:=\sum_{d} \ev_{p,*}\left(\vrs^{d} \right) Q^{d} \in K_{\Tt\times \mathbb{C}^{\times}_{q}}(X)_{loc}[[Q]].
\]
where the sum runs over the cone of degrees $d$ such that $\qm_{\ns p}^{d}(X)$ is nonempty.
\end{definition}
Notice that the choice of normalizations in \eqref{symvss} ensures that $V^X(Q)=1+ \mathcal{O}(Q)$.

\begin{remark}
    More generally, one could consider vertex functions with descendant insertions, namely the power series defined by
    \[
    \ver^{X, \tau}:=\sum_{d} \ev_{p,*}\left(\vrs^{d}\otimes (\ev_{p'})^*(\tau) \right) Q^{d} \in K_{\Tt\times \mathbb{C}^{\times}_{q}}(X)_{loc}[[Q]],
    \]
    where 
    \[
    \begin{tikzcd}
         \FX & \qm_{\ns p}(X) \arrow[l, swap,  "\ev_{p'}"]\arrow{r}{\ev_{p}} & X 
    \end{tikzcd}
    \]
    and $\tau$ is a polynomial in tautological bundles of $\FX$. In this paper we study mirror symmetry only for the trivial insertion $\tau=1$, so we drop it from the notation. Understanding mirror symmetry for vertex functions with nontrivial insertion is, nevertheless, a very interesting topic. 
\end{remark}

\subsection{Dependence on polarization}

As shown in \cite[\S8.3]{Okounkov_lectures}, vertex functions for different polarizations are equal after a shift of $Q$ by a power $q$. For 3d mirror symmetry of vertex functions, we will make particular choices of polarizations. For now, we will leave the choice open.

\subsection{Localization formula}

Let $X$ be a separated or co-separated bow variety with brane diagram
\[
\underset{\Zb_{1}}{\fs}d_1 \underset{\Zb_{2}}{\fs} \ldots \underset{\Zb_{m}}{\fs} d_{m} \underset{\Ab_{1}}{\bs} \ldots \underset{\Ab_{n-1}}{\bs} d_{m+n-1} \underset{\Ab_{n}}{\bs} \quad \text{or} \quad \underset{\Ab_{1}}{\bs}d_1 \underset{\Ab_{2}}{\bs} \ldots \underset{\Ab_{n}}{\bs} d_{n} \underset{\Zb_{1}}{\fs} \ldots \underset{\Zb_{m-1}}{\fs} d_{m+n-1} \underset{\Zb_{m}}{\fs}
\] 
Following our running convention, we label the D3 branes $\Xb_i$ and, accordingly, the tautological vector bundles $\xi_i:=\xi_{\Xb_i}$ from left to right. Notice that $\rk(\xi_i)=d_i$.

Fix $f \in X^{\Tt}$. By Theorem \ref{thm: fixed points bow}, $f$ is determined by its butterfly diagram $B^{f}=(B_0^f, B_1^f, \col)$. Let $W_{\Xb}$ be the vector space with basis given by the vertices in $\col^{-1}(\Xb)$. Viewing the arrows of the butterfly as the components of linear maps, the butterfly diagram precisely encodes the quiver description of a distinguished representative of the fixed point $f$. There is a $\Tt$-weight associated to each vertex of the butterfly diagram, which we will denote by $w(a)$, such that
\[
\tb_{i}|_{f}=\sum_{\substack{a \in B^{f}_{0} \\ \col(a)=i}} w(a).
\]

\begin{definition}\label{rpp}
    A stable reverse plane partition over $B^{f}$ is a function $\pi: B^{f}_{0} \to \mathbb{N}$ such that 
    \begin{enumerate}
        \item $a \to b \in B^{f}_{1} \implies \pi(a) \leq \pi(b)$
        \item If either of the two branes adjacent to $\col(a)$ is a D5 brane, then $\pi(a)=0$.
    \end{enumerate}
    Let $\rpp(f)$ denote the set of stable reverse plane partitions over $B^{f}$.
\end{definition}

The degree of a stable reverse plane partition $\pi$ is the vector 
\[
\deg(\pi):=\left(\sum_{\substack{a\in B^{f}_{0} \\ \col(a)=\Zb_{i}^{+}}} \pi(a)\right)_{1 \leq i \leq m-1} \in \mathbb{Z}^{m-1}.
\]
In the previous line, we only take the index $i$ up to $m-1$ because $\col(a)>m-1 \implies \pi(a)=0$ by Definition \ref{rpp}. For $d\in \mathbb{Z}^{m-1}$, let $\rpp^{d}(f)\subset \rpp(f)$ be the subset of reverse plane partitions of degree $d$.

\begin{proposition}\label{prop: fixedqm}
    Let $\qm_{f}=\ev_{p}^{-1}(f)$. There are canonical bijections
    \[
\qm_{\ns p}^{d}(X)^{\Tq}=\bigsqcup_{f \in X^{\Tt}} \left(\qm_{f}^{d}\right)^{\Tq},
    \]
    and
    \[
\left(\qm_{f}^{d}\right)^{\Tq}= \rpp^{d}(f).
    \]
\end{proposition}
\begin{proof}
  A $\Cs_{q}$-fixed quasimap in $\qm^{d}_{\ns p}(X)$ must map the $\mathbb{A}^{1}$ chart around $p$ to a single point of $X$. If the quasimap is also $\Tt$-fixed, then the chart must map to a fixed point $f \in X^{\Tt}$. This proves the first statement.

It follows from the argument given in \cite[\S7.2]{Okounkov_lectures} that $\Cs_{q}$ fixed quasimaps in $\qm^{d}_{f}$ are in bijection with ``framed" flags of quiver sub-representations ending in the quiver representation for $f$. This quiver representation and in particular the vector spaces $W_{\Xb}$ are encoded by the butterfly diagram, as we recalled above.

If the quasimap is additionally $\Tt$-fixed, then all of the vector spaces appearing in this flag split into weight spaces for $\Tt$, and the inclusions in the flag are $\Tt$-equivariant. From the definition of the butterfly diagram, see \cite{rimanyi2020bow}, it follows that the nonzero $\Tt$-weight spaces of $\tb_{\Xb}|_{f}$ are all one-dimensional for any $\Xb$. Thus the $\Tt$-weight spaces of each $W_{\Xb}$ are also one-dimensional. So a quasimap in $(\qm^{d}_{f})^{\Tq}$ is uniquely determined up to isomorphism by specifying the positions in the flag at which each weight space first appears. These positions are encoded precisely by a reverse plane partition. Because the flag of quiver subrepresentations must be ``framed", see \cite{Okounkov_lectures}, the reverse plane partitions must be stable in the sense of Definition \ref{rpp}.

Conversely, any stable reverse plane partition over $B^{f}$ give rise to a framed flag of quiver subrepresentations ending in $f$ and thus to a $\Tq$-fixed quasimap.
  
\end{proof}

Recall that, for any all $d\in \ZZ$, the finite $q$-Pochhammer symbol $(x)_d$ is defined as $(x)_d:=\Phi(x)/\Phi(q^d x)$.
\begin{proposition}
\label{prop: vertex formula}
The restriction of the vertex function of $X$ to $f\in X^{\Tt}$ is 
\begin{multline*}
\ver^{X}_{f}(Q):= \ver^{X}(Q)|_{f}  \\
=\sum_{\pi \in \rpp(f)} \left(-q \hbar \right)^{\frac{N(\pi)}{2}} q^{\frac{\deg \qmpol_{\pi}}{2}} Q^{\deg(\pi)} \left(\prod_{\substack{a,b \in B^{f}_{0} \\ \col(b)-\col(a)=1}} \frac{\left(\hbar^{-1} \frac{w(b)}{w(a)}\right)_{\pi(b)-\pi(a)}}{\left(q \frac{w(b)}{w(a)}\right)_{\pi(b)-\pi(a)}} \right)\left(\prod_{\substack{a, b \in B^{f}_{0} \\ \col(a)=\col(b)}}\frac{\left( q \frac{w(b)}{w(a)}\right)_{\pi(b)-\pi(a)}}{\left( \hbar^{-1} \frac{w(b)}{w(a)}\right)_{\pi(b)-\pi(a)}}\right),
\end{multline*}
where
\[
N(\pi)=\sum_{\substack{a, b \in B^{f}_{0}\\ \col(b)-\col(a)=1}}(\pi(b)-\pi(a)).
\]
\end{proposition}
\begin{proof}
    This follows from the virtual localization theorem by combining Proposition \ref{prop: fixedqm}, \eqref{eq: alpha is pol}, and \eqref{eq: qm tvir}.
\end{proof}

From the definition of a stable reverse plane partition, it is immediate that the formula for $\ver^{X}_{f}$ can be written such that the products run only over vertices colored by D3 branes adjacent to an NS5 brane.

    For an arbitrary (not necessarily separated or co-separated), one can in principle apply localization. However, we do not know a direct characterization of torus fixed quasimaps. Instead, we prefer to study vertex functions of an arbitrary bow variety by first using Hanany-Witten transition to write it as a separated or co-separated bow variety and then applying the above formulas. This is justified by the following result.

\begin{proposition}[{\cite[Prop. 4.6.1]{qm}}]\label{prop: HW vertex}
    The moduli stack $\qm_{\ns p}(X)$ and its perfect obstruction theory depend only on the pair $(X,\FX)$.
\end{proposition}

As a result, we have the following.

\begin{proposition}
    Let $X$ be a bow variety and let $X'$ be the bow variety obtained by swapping a D5 brane and an NS5 brane using Hanany-Witten transition. Then for any choices of polarizations, there exists $\beta \in \mathbb{Z}^{m-1}$ such that 
    \[
\ver^{X}(Q)=\ver^{X'}(Q q^{\beta}).
    \]
\end{proposition}
\begin{proof}
    Since Hanany-Witten transition induces isomorphisms of both bow varieties and stack quotients, the previous proposition implies that it identifies the virtual structure sheaves. Vertex functions are defined using $\vrs$, which differs from $\mathcal{O}_{\text{vir}}$ by a twist depending on a choice of polarization. Depending on the choices of polarization for $X$ and $X'$, the twist may contribute differently, leading to the $q^{\beta}$ shift.
\end{proof}

For a certain choice of polarization, we will make this shift explicit below.

\begin{remark}
   The formula of Proposition \ref{prop: vertex formula} is almost identical to the formula for cotangent bundles of partial flag varieties viewed as quiver varieties, see \cite{KorZeit}. Given a type $A$ quiver variety $X$, this uniformity can only be seen by realizing $X$ as a bow variety, applying Hanany-Witten transition to obtain a separated or co-separated bow variety, and using Proposition \ref{prop: vertex formula}. Using the usual localization formula for quiver varieties, written for example in \cite{dinkinsD5Vertex}, leads to less uniform formulas.
\end{remark}

\subsection{Vertex functions and Hanany-Witten transition}

Recall the class $\alpha $ from $\eqref{eq: class alpha for general bow variety}$. For the rest of this subsection we fix the polarization $\alpha+\beta$ from $\eqref{eq: alpha is pol}$. Since the polarization $\qmpol$ appears in \eqref{symvss} in a ratio, only the topologically nontrivial part will contribute. Thus to compare the vertex functions before and after Hanany-Witten transition, we need to understand how $\alpha$ changes.

Consider the following Hanany-Witten transition from left to right:
\[
\cdots \ttt{$d_1$\fs $d_2$\bs $d_3$}\cdots  \longrightarrow  \cdots \ttt{$d_1$\bs $d_2'$\fs $d_3$} \cdots
\] 

\begin{proposition}[\cite{BR}]
Let $X^{\text{before}}$ and $X^{\text{after}}$ be the two bow varieties above. Then 
    \[
\alpha^{\text{after}}-\alpha^{\text{before}}=\Hom(\tb_{1}^{\text{before}},\mathbb{C}_{\Ab})-\hbar \Hom(\mathbb{C}_{\Ab},\tb_{2}^{\text{after}}) \in K_{\Tt}(X^{\text{before}})
    \]
\end{proposition}

It is shown in \cite{BR} that the $K$-theory class of $\tb_{2}^{\text{after}}$ is equal to that of $\tb_{1}^{\text{before}}$ up to the addition of a topologically trivial virtual bundle. As a consequence, we have the following.
\begin{proposition}\label{detalpha}
The $K$-theory class of
    \[
\det(\alpha^{\text{after}}) \otimes \det(\alpha^{\text{before}})^{-1} \otimes \det(\tb_{1}^{\text{before}})^{-2}
    \]
    is equal to that of a topologically trivial line bundle.
\end{proposition}

\begin{proposition}\label{prop: vertex and HW}
Assume that \eqref{eq: alpha is pol} is the chosen polarization for all vertex functions. If Hanany-Witten transition moves a D5 brane $\Ab$ from the right to the left of an NS5 brane $\Zb_{k+1}$ then
    \[
\ver^{X^{\text{after}}}(Q)=\ver^{X^{\text{before}}}(Q)\big|_{Q_{k}=Q_{k} q}.
    \]
    \end{proposition}
    \begin{proof}
        This is immediate from our choice of polarization, Proposition \ref{detalpha}, and \eqref{symvss}.
    \end{proof}

Applying the previous propsition repeatedly, we obtain the following.

\begin{corollary}\label{corcosepV}
    Let $\D$ be a separated brane diagram with $n$ D5 branes. Let $\D'$ be the co-separated brane diagram Hanany-Witten equivalent to $\D$. Let $X$ and $X'$ be the corresponding bow varieties. Assume that \eqref{eq: alpha is pol} is the chosen polarization for all vertex functions. Then
    \[
V^{X'}=V^{X}|_{Q_{k}=Q_{k} q^{n}}.
    \]
    where the shift is applied to all the variables $Q_{k}$.
\end{corollary}

\subsection{Normalization for mirror symmetry}\label{sec: mirsym vertex normalization}
For 3d mirror symmetry, we need to make a specific choice of polarization and normalization of vertex functions. Let $X$ be a bow variety corresponding to a brane diagram $\D$. There are two sides of mirror symmetry: the side of $X$ and the side of $X^{!}$. When mirror symmetry is considered, $\ver^{X}(Q)$ will denote the vertex function of $X$ for the polarization $\hbar(\alpha+\beta)^{\vee}$, while $\ver^{X^{!}}(Q^{!})$ will denote the vertex function of $X^{!}$ for the polarization $\alpha^{!}+\beta^{!}$.

If we change which we are considering the ``dual" side, then we must also change polarization.

With these choices understood, we let
\begin{equation}\label{eq: def of MSver}
\begin{split}
\MSver^{X}(Q)&:=\ver^{X}(Q)\big|_{Q_{i}=Q_{i} (-\hbar^{1/2})^{\w(\Zb_{i})-\w(\Zb_{i+1})}} \\
\MSver^{X^{!}}(Q^{!})&:=\ver^{X^{!}}(Q^{!})\big|_{Q^{!}_{i}=Q^{!}_{i} (-(\hbar^{!})^{1/2})^{\w(\Zb^{!}_{i})-\w(\Zb^{!}_{i+1})}}.
\end{split}
\end{equation}

\subsection{Weight one D5 branes: holomorphicity of vertex functions}\label{sub: vholo}
From Proposition \ref{prop: vertex formula}, one can see that vertex functions may have poles in the equivariant parameters $a_1,\ldots, a_n$ of the torus $\At$. Nevertheless, sometimes a vertex function will be pole-free in a neighborhood of a certain point in $a$. We now describe one such situation.

Assume that $\D$ is a separated or co-separated brane diagram such that $\w(\Ab_{i})=1$ for all $i$. Let $X$ be the corresponding bow variety. Such bow varieties are isomorphic to cotangent bundles of (type $A$) partial flag varieties. A torus fixed point $f\in X^{\At}$ is determined by the data of a function $\sigma_{f}:\{1,\ldots,n\}\to \{1,\ldots,m\}$ such that tie attached to $\Ab_{i}$ is also attached to $\Zb_{\sigma_{f}(i)}$. 

Consider the partial ordering of equivariant parameters by the rule $a_{i}< a_{j}$ whenever $i \in \sigma_{f}^{-1}(k)$ and $j \in \sigma_{f}^{-1}(l)$ for some $k$ and $l$ such that $k <l$. We denote this ordering by $<_{f}$. Let $\chamb_{f}$ be any generic chamber inducing this ordering\footnote{Here, generic means that each $\frac{a_i}{a_j}$ is either attracting or repelling. This is a stronger assumption than the usual condition on chambers.}. In particular, if $i<_{f} j$, then $\frac{a_i}{a_j}$ is repelling with respect to $\chamb_{f}$. Unless $X$ is the cotangent bundle of a complete flag variety, there are many choices of such a $\chamb_{f}$, but several of our results in this paper are independent of the choice.

Let $\lim_{a \to 0_{\chamb_{f}}}$ denote the limit of sending the repelling directions for $\chamb_{f}$ to $0$. In other words, if $\sigma$ is a cocharacter in $\chamb_{f}$, then $\lim_{a \to 0_{\chamb_{f}}}F(a)=\lim_{t \to 0} (F \circ \sigma)(t)$.
Let $0_{\chamb_{f}}$ denote the point given by $w=0$ whenever $w$ is a repelling weight for $\chamb_{f}$.

\begin{proposition}\label{Vholomorphic}
The vertex function $\ver^{X}_{f}$ defines a holomorphic function of $a$ in a neighborhood of $0_{\chamb_{f}}$ for any choice of $\chamb_{f}$.
\end{proposition}
\begin{proof}
    The point $f$ is a minimal fixed point with respect to the order induced by $\chamb_{f}$. So the result follows as a trivial case of Theorem 5 from \cite{aganagic2016elliptic}.
\end{proof}

\subsection{Weight one D5 branes: refined formula, limits, and convergence}\label{sec: flag vertex results}

Suppose that $\D$ is a separated brane diagram, with bow variety $X$, such that $\w(\Ab_{i})=1$. In this case, we can rewrite Proposition \ref{prop: vertex formula} more efficiently.

Let $f \in X^{\At}$ and let $\sigma_{f}$ be as in \S\ref{sub: vholo}. A stable reverse plane partition over the butterfly diagram $B^{f}$ is the disjoint union $\pi=\bigsqcup_{i=1}^{n} \pi^{i}$ where $\pi^{i}$ is a stable reverse plane partition over the butterfly corresponding to the D5 brane $\Ab_{i}$. Since these butterflies are chains, the data is simply $\pi^{i}=(\pi^{i}_{\sigma_{f}(i)} \geq \pi^{i}_{\sigma_{f}(i)+1} \geq \ldots \geq \pi^{i}_{m} = 0)$.

Define $\theta^{i}_{k}:=\pi^{i}_{k}-\pi^{i}_{k+1}$. Notice that $\theta^{i}_{k}=0$ if $k \geq m$. Then summing over reverse plane partitions $\pi$ is equivalent to summing over all nonnegative $\theta^{i}_{j}$ for $1 \leq i \leq n$ and $\sigma_{f}(i) \leq j \leq m-1$. Let $M_{f}$ denote the set of all such $\theta^{i}_{j}$.

\begin{proposition}\label{prop:flagvertex}
    The vertex function is
    \begin{align*}
        \ver^{X}_{f}(Q)&=\sum_{\theta \in M_{f}}  \left( \prod_{i=1}^{n} \prod_{k=\sigma_{f}(i)}^{m-1} \left( \prod_{l=\sigma_{f}(i)}^{k} (-q \hbar^{1/2})^{-\w(\Zb_{l})-\w(\Zb_{l+1})} Q_{l}\right)^{\theta^{i}_{k}} \right)
     \left(    \prod_{i=1}^{n} \prod_{k =\sigma_{f}(i)}^{m-1}\frac{\left( \hbar^{-1}\right)_{\theta^{i}_{k}}}{\left( q\right)_{\theta^{i}_{k}}} (q \hbar)^{\theta^{i}_{k}}\right)
        \\ &\left( \prod_{\substack{1 \leq i,j \leq n\\ \sigma_{f}(i)<\sigma_{f}(j)}}  \prod_{k = \sigma_{f}(j)}^{m}\frac{\left(q^{\sum_{l \geq k} (\theta^{i}_{l}-\theta^{j}_{l})} \hbar^{-1} \frac{a_{i}}{a_{j}}\right)_{\theta^{i}_{k-1}}}{\left(q^{\sum_{l \geq k} (\theta^{i}_{l}-\theta^{j}_{l})} q \frac{a_{i}}{a_{j}}\right)_{\theta^{i}_{k-1}}} 
      \frac{\left(q^{-\theta^{j}_{k}+\sum_{l > k}( \theta^{i}_{l}-\theta^{j}_{l})} q \hbar \frac{a_{i}}{a_{j}}\right)_{\theta^{i}_{k}}}{\left(q^{-\theta^{j}_{k}+\sum_{l > k} (\theta^{i}_{l}-\theta^{j}_{l})} \frac{a_{i}}{a_{j}}\right)_{\theta^{i}_{k}}} (q\hbar)^{\theta^{i}_{k-1}} \right) \\
        & \left( \prod_{\substack{1 \leq i<j \leq n \\ \sigma_{f}(i)=\sigma_{f}(j)}} \prod_{k = \sigma_{f}(j)}^{m-1}\frac{\left(q^{\sum_{l \geq k} (\theta^{i}_{l}-\theta^{j}_{l})} \hbar^{-1} \frac{a_{i}}{a_{j}}\right)_{\theta^{j}_{k}}}{\left(q^{\sum_{l \geq k} (\theta^{i}_{l}-\theta^{j}_{l})} q \frac{a_{i}}{a_{j}}\right)_{\theta^{j}_{k}}} 
        \frac{\left(q^{-\theta^{j}_{k}+\sum_{l \geq k} (\theta^{i}_{l}-\theta^{j}_{l})} q \hbar \frac{a_{i}}{a_{j}}\right)_{\theta^{i}_{k}}}{\left(q^{-\theta^{j}_{k}+\sum_{l \geq k} (\theta^{i}_{l}-\theta^{j}_{l})} \frac{a_{i}}{a_{j}}\right)_{\theta^{i}_{k}}} (q \hbar )^{\theta^{j}_{k}}\right) .
    \end{align*}
\end{proposition}
\begin{proof}
Fix two D5 branes $\Ab_{i}$ and $\Ab_{j}$ such that $\sigma_{f}(i)<\sigma_{f}(j)$ and consider all the terms in Proposition \ref{prop: vertex formula} involving the corresponding equivariant parameters $a_i$ and $a_j$. The first product contributes the terms
\[
\left(\prod_{k=\sigma_{f}(j)}^{m} \underbrace{\frac{\left(\frac{a_{j}}{a_{i}} \right)_{\pi^{j}_{k}-\pi^{i}_{k-1}}}{\left(q \hbar \frac{a_{j}}{a_{i}} \right)_{\pi^{j}_{k}-\pi^{i}_{k-1}}}}_{E_{k}:=} \right) \left(\prod_{k=\sigma_{f}(j)+1}^{m} \underbrace{\frac{\left(\frac{a_{i}}{a_{j}} \right)_{\pi^{i}_{k}-\pi^{j}_{k-1}}}{\left(q \hbar \frac{a_{i}}{a_{j}} \right)_{\pi^{i}_{k}-\pi^{j}_{k-1}}}}_{F_{k}:=} \right),
\]
while the second product contributes
\[
\left(\prod_{k=\sigma_{f}(j)}^{m} \underbrace{\frac{\left( q \frac{a_{i}}{a_{j}} \right)_{\pi^{i}_{k}-\pi^{j}_{k}}}{\left( \hbar^{-1} \frac{a_{i}}{a_{j}} \right)_{\pi^{i}_{k}-\pi^{j}_{k}}}}_{G_{k}:=} \underbrace{\frac{\left( q \frac{a_{j}}{a_{i}} \right)_{\pi^{j}_{k}-\pi^{i}_{k}}}{\left( \hbar^{-1} \frac{a_{j}}{a_{i}} \right)_{\pi^{j}_{k}-\pi^{i}_{k}}}}_{H_{k}:=} \right) .
\]

Then
\begin{align*}
E_{k} G_{k}&=\frac{\left(\frac{a_{j}}{a_{i}} \right)_{\pi^{j}_{k}-\pi^{i}_{k-1}}}{\left(q \hbar \frac{a_{j}}{a_{i}} \right)_{\pi^{j}_{k}-\pi^{i}_{k-1}}} \frac{\left( q \frac{a_{i}}{a_{j}} \right)_{\pi^{i}_{k}-\pi^{j}_{k}}}{\left( \hbar^{-1} \frac{a_{i}}{a_{j}} \right)_{\pi^{i}_{k}-\pi^{j}_{k}}} \\
&= (q\hbar)^{\pi^{i}_{k-1}-\pi^{j}_{k}}\frac{\left(q^{\pi^{i}_{k}-\pi^{j}_{k}}\hbar^{-1}\frac{a_i}{a_j} \right)_{\pi^{i}_{k-1}-\pi^{i}_{k}}}{\left(q^{\pi^{i}_{k}-\pi^{j}_{k}}q\frac{a_i}{a_j} \right)_{\pi^{i}_{k-1}-\pi^{i}_{k}} } \\
&=
(q\hbar)^{\theta^{i}_{k-1}+\sum_{l \geq k}(\theta^{i}_{l}-\theta^{j}_{l})}\frac{\left( q^{\sum_{l \geq k}(\theta^{i}_{l}-\theta^{j}_{l})} \hbar^{-1} \frac{a_{i}}{a_{j}} \right)_{\theta^{i}_{k-1}}}{\left( q^{\sum_{l \geq k}(\theta^{i}_{l}-\theta^{j}_{l})} \hbar^{-1} \frac{a_{i}}{a_{j}} \right)_{\theta^{i}_{k-1}}}
\end{align*}
for $\sigma_{f}(j)\leq k \leq m$. Similarly,
\[
F_{k+1} H_{k}=(q \hbar)^{-\sum_{l \geq k}(\theta^{i}_{l}-\theta^{j}_{l})}\frac{\left(q^{-\theta^{j}_{k}+\sum_{l > k}(\theta^{i}_{l}-\theta^{j}_{l})} q \hbar \frac{a_{i}}{a_{j}} \right)_{\theta^{j}_{k}}}{\left(q^{-\theta^{j}_{k}+\sum_{l > k}(\theta^{i}_{l}-\theta^{j}_{l})}  \frac{a_{i}}{a_{j}} \right)_{\theta^{j}_{k}}}
\]
for $\sigma_{f}(j)\leq k \leq m$. Together, $E_{k} G_{k} F_{k+1} H_{k}$ precisely accounts for the terms in the second row of the formula in the present proposition.

A similar computation involving two different D5 branes $\Ab_{i}$ and $\Ab_{j}$ such that $\sigma_{f}(i)=\sigma_{f}(j)$ provides the terms in the third row of the proposition.

Finally, the terms in the first row are provided by the K\"ahler terms and the contributions between $\Ab_{i}$ and $\Ab_{j}$ where $i=j$.
\end{proof}

\begin{remark}\label{rmk: bispectral}
    When, additionally, $\w(\Zb_{j})=1$ for all $j$ and $\sigma_{f}(i)=i$ for all $i$, Proposition \ref{prop:flagvertex} is exactly the formula for the Macdonald function from (1.11) of \cite{NSmac} after the change of variables $\hbar=t^{-1}$ and $\prod_{l=\sigma_{f}(i)}^{k} (-q \hbar^{1/2})^{-\w(\Zb_{l})-\w(\Zb_{l+1})} Q_{l}=x_{l+1}/x_{l}$.
\end{remark}

From the previous formula, it is easy to take limits in the equivariant parameters.

\begin{proposition}
For any $f \in X^{\At}$ and for any choice of $\chamb_{f}$, we have
    \begin{multline*}
\flaglim^{X}(Q):=\lim_{a \to 0_{\chamb_{f}}} \ver^{X}_{f}(Q) \\
=\prod_{i=1}^{m-1} \prod_{j=1}^{\w(\Zb_{i})} \prod_{k=i}^{m-1}\frac{\Phi\left( \hbar^{-1} (q\hbar)^{j}\prod_{i \leq l \leq k} (q\hbar)^{\w(\Zb_{l+1})} (-q\hbar^{1/2})^{-\w(\Zb_{l})-\w(\Zb_{l+1})} Q_{l} \right)}{\Phi\left( (q\hbar)^{j}\prod_{i \leq l \leq k} (q\hbar)^{\w(\Zb_{l+1})} (-q\hbar^{1/2})^{-\w(\Zb_{l})-\w(\Zb_{l+1})} Q_{l} \right)}.
    \end{multline*}
    In particular, the limit is independent of both $f$ and $\chamb_{f}$.
\end{proposition}
\begin{proof}
We make the computation for the choice of $\chamb_{f}$ such that $\frac{a_i}{a_j}$ is repelling if $\sigma_{f}(i)<\sigma_{f}(j)$ or if $\sigma_{f}(i)=\sigma_{f}(j)$ and $i<j$. For other choices of $\chamb_{f}$, the computation is similar. We compute
     \begin{align*}
       \lim_{a \to 0_{\chamb_f}} \ver^{X}_{f}(Q)&=\sum_{\theta \in M_{f}}  \left( \prod_{i=1}^{n} \prod_{k=\sigma_{f}(i)}^{m-1} \left( \prod_{l=\sigma_{f}(i)}^{k} (-q \hbar^{1/2})^{-\w(\Zb_{l})-\w(\Zb_{l+1})} Q_{l}\right)^{\theta^{i}_{k}} \right)
     \\
        & \left(    \prod_{i=1}^{n} \prod_{k =\sigma_{f}(i)}^{m-1}\frac{\left( \hbar^{-1}\right)_{\theta^{i}_{k}}}{\left( q\right)_{\theta^{i}_{k}}} (q \hbar)^{\theta^{i}_{k}}\right)\left( \prod_{\substack{1 \leq i,j \leq n\\ \sigma_{f}(i)<\sigma_{f}(j)}}  \prod_{k = \sigma_{f}(j)}^{m}
     (q\hbar)^{\theta^{i}_{k-1}} \right) \left( \prod_{\substack{1 \leq i<j \leq n \\ \sigma_{f}(i)=\sigma_{f}(j)}} \prod_{k = \sigma_{f}(j)}^{m-1} (q \hbar )^{\theta^{j}_{k}}\right). 
    \end{align*}

    Simplifying, we get
      \begin{align*}
       \lim_{a \to 0_{\chamb_f}} \ver^{X}_{f}(Q)&=\sum_{\theta \in M_{f}}  \left( \prod_{i=1}^{n} \prod_{k=\sigma_{f}(i)}^{m-1} \left( \prod_{l=\sigma_{f}(i)}^{k} (-q \hbar^{1/2})^{-\w(\Zb_{l})-\w(\Zb_{l+1})} Q_{l}\right)^{\theta^{i}_{k}} \right)
     \\
        & \left(    \prod_{i=1}^{n} \prod_{k =\sigma_{f}(i)}^{m-1}\frac{\left( \hbar^{-1}\right)_{\theta^{i}_{k}}}{\left( q\right)_{\theta^{i}_{k}}} (q \hbar)^{\theta^{i}_{k}}\right)
        \left( \prod_{i=1}^{n} \prod_{k=\sigma_{f}(i)}^{m-1} (q \hbar )^{(c_{i}+c_{i,k})\theta^{i}_{k}}\right) \\
        &=\sum_{\theta \in M_{f}}  \prod_{i=1}^{n} \prod_{k=\sigma_{f}(i)}^{m-1}   \frac{\left( \hbar^{-1}\right)_{\theta^{i}_{k}}}{\left( q\right)_{\theta^{i}_{k}}} (q \hbar)^{\theta^{i}_{k}(c_{i}+c_{i,k}+1)} \left( \prod_{l=\sigma_{f}(i)}^{k} (-q \hbar^{1/2})^{-\w(\Zb_{l})-\w(\Zb_{l+1})} Q_{l}\right)^{\theta^{i}_{k}}  \\
        &=\prod_{i=1}^{n} \prod_{k=\sigma_{f}(i)}^{m-1} \sum_{\theta^{i}_{k}=0}^{\infty} \frac{\left( \hbar^{-1}\right)_{\theta^{i}_{k}}}{\left( q\right)_{\theta^{i}_{k}}}\left(  (q \hbar)^{c_{i}+c_{i,k}+1} \prod_{l=\sigma_{f}(i)}^{k} (-q \hbar^{1/2})^{-\w(\Zb_{l})-\w(\Zb_{l+1})} Q_{l}\right)^{\theta^{i}_{k}}   \\
        &=\prod_{i=1}^{n} \prod_{k=\sigma_{f}(i)}^{m-1}  \frac{\Phi\left(\hbar^{-1} (q\hbar)^{c_i+c_{i,k}+1} \prod_{\sigma_{f}(i) \leq l \leq k}(-q\hbar^{1/2})^{-\w(\Zb_{l})-\w(\Zb_{l+1})} Q_{l} \right)}{\Phi\left( (q\hbar)^{c_i+c_{i,k}+1} \prod_{\sigma_{f}(i) \leq l \leq k}(-q\hbar^{1/2})^{-\w(\Zb_{l})-\w(\Zb_{l+1})} Q_{l} \right)}
    \end{align*}
    where $c_{i}=|\{j \, \mid \, j<i, \sigma_{f}(i)=\sigma_{f}(j)\}|$ and $c_{i,k}=|\{j \, \mid \, k+1 \geq \sigma_{f}(j)>\sigma_{f}(i)  \}|$. In the last equality, we used the $q$-binomial theorem. This equals
    \[
\prod_{i=1}^{m-1} \prod_{j=1}^{\w(\Zb_{i})} \prod_{k=i}^{m-1}\frac{\Phi\left( \hbar^{-1} (q\hbar)^{j}\prod_{i \leq l \leq k} (q\hbar)^{\w(\Zb_{l+1})}(-q\hbar^{1/2})^{-\w(\Zb_{l})-\w(\Zb_{l+1})} Q_{l} \right)}{\Phi\left( (q\hbar)^{j}\prod_{i \leq l \leq k}  (q\hbar)^{\w(\Zb_{l+1})} (-q\hbar^{1/2})^{-\w(\Zb_{l})-\w(\Zb_{l+1})} Q_{l} \right)},
    \]
    which is clearly independent of $f$. 
\end{proof}

By a change of variable, we obtain the following. 

\begin{proposition}\label{prop: msflaglim}
For any $f \in X^{\At}$ and for any choice of $\chamb_{f}$, we have
    \[
\MSflaglim^{X}(Q):=\lim_{a \to 0_{\chamb_f}} \MSver^{X}(Q^{-1})=\prod_{i=1}^{m-1} \prod_{j=1}^{\w(\Zb_{i})} \prod_{k=i}^{m-1}\frac{\Phi\left(\hbar^{-1} (q\hbar)^{j}\prod_{i \leq l \leq k} q^{-\w(\Zb_{l})}  Q_{l}^{-1}\right)}{\Phi\left( (q\hbar)^{j}\prod_{i \leq l \leq k} q^{-\w(\Zb_{l})} Q_{l}^{-1} \right)}.
    \]
    
\end{proposition}

\subsection{Weight one D5 and NS5 branes: analytic continuation}

Now assume furthermore that $X$ has $m$ D5 branes and $m$ NS5 branes, all of weight one. The bow variety $X$ is isomorphic to the cotangent bundle of the complete flag variety in $\mathbb{C}^{m}$. By Remark \ref{rmk: bispectral}, we can apply the results of \cite{NSmac} to $\ver^{X}$. We need the following.

\begin{proposition}[\cite{NSmac} Theorem 1.6]\label{prop:fullflagpoles}
    For any $f \in X^{\At}$, $\ver^{X}_{f}(Q)$ can be analytically continued to a meromorphic function of $Q$ with simple poles at the divisors $\hbar q^{i-j} \prod_{i \leq l \leq j} Q_{l}=q^{-d}$ for $1 \leq i \leq m-1$ and $i \leq j \leq m-1$ and $d\geq 0$.
\end{proposition}
Equivalently, the function $\flaglim^{X}(Q)^{-1} \ver^{X}_{f}(Q)$ is an entire function of $Q$, and the same is true for $\MSflaglim^{X}(Q)^{-1} \MSver^{X}_{f}(Q)$.

\subsection{Permutation of D5 branes}

Let $\D$ be a separated or co-separated brane diagram with $n$ D5 branes. For $\sigma  \in S_{n}$, let $\D':=\sigma(\D)$ be the brane diagram obtained by permuting the D5 branes according to $\sigma$. By this, we mean that $\D'$ has D5 branes $\Ab'_{1},\ldots,\Ab'_{n}$ such that $\w(\Ab_{i}')=\w(\Ab_{\sigma^{-1}(i)})$. For a simple reflection, see the diagram below.

\begin{equation*}
\begin{tikzpicture}[baseline=(current  bounding  box.center), scale=.35]
\draw[thick] (0,1)--(1,1) node [above] {$d_1$} -- (2,1);
\draw[thick,blue] (3,0) node [below] {$\Ab_i$}--(2,2);
\draw[thick] (3,1) -- (4,1) node [above] {$d_2$}--(5,1);
\draw[thick,blue] (6,0)  node [below] {$\Ab_{i+1}$} -- (5,2);
\draw[thick] (6,1)--(7,1) node [above] {$d_3$}--(8,1);
\draw[ultra thick, <->] (10,1)--(11.5,1)  -- (13,1);
\draw[thick] (15,1)--(16,1) node [above] {$d_1$} -- (17,1);
\draw[thick,blue] (18,0) node [below] {$\Ab_{i+1}$} --(17,2);
\draw[thick] (18,1)--(19,1) node [above] {$d'_2$}--(20,1);
\draw[thick,blue] (21,0) node [below] {$\Ab_i$} --(20,2);
\draw[thick] (21,1)--(22,1) node [above] {$d_3$}--(23,1);
\end{tikzpicture}
\qquad\quad \text{for } d_2+d'_2=d_1+d_3.
\end{equation*}

This induces an action of the symmetric group $S_{n}$ on the set of all bow varieties with $n$ D5 branes. This action can be lifted to an $S_{n}$ action on the set of all torus fixed points of such bow varieties by considering the tie diagram description of fixed points as in Figure \ref{fig: sym group action fixed points}.

\begin{figure}
\centering
\begin{tikzpicture}[scale=.3]
\draw [thick] (0,1) --(8,1);  
\draw [thick, blue] (3,0) -- (2,2);
\draw[thick, blue] (6,0)--(5,2);
\draw [dashed](2,2) to [out=120,in=0] (0,3) ;
\draw [dashed](2,2) to [out=120,in=0] (0,3.25) node [left] {$B$} ;
\draw [dashed](2,2) to [out=120,in=0] (0,3.4) ;
\draw [dashed](6,0) to [out=-60,in=180] (8,-1)  ;
\draw [dashed](6,0) to [out=-60,in=180] (8,-1.25)  node [right] {$C$} ;
\draw [dashed](6,0) to [out=-60,in=180] (8,-1.4);
\draw [dashed](3,0) to [out=-60,in=180] (8,-2.6);
\draw [dashed](3,0) to [out=-60,in=180] (8,-2.8) node [right]{$D$};
\draw [dashed](3,0) to [out=-60,in=180] (8,-3) ;
\draw [dashed](5,2) to [out=120,in=0] (0,4.6) ;
\draw [dashed](5,2) to [out=120,in=0] (0,4.8) node [left]{$A$};
\draw [dashed](5,2) to [out=120,in=0] (0,5);
\draw[ultra thick, <->] (10.5,1)--(12.5,1) node[above]{$\sigma_{i,i+1}$} -- (14.5,1);
\draw[ultra thick, <->] (10.5,1)--(12.5,1) node[below]{action} -- (14.5,1);
\begin{scope}[xshift=17cm]
\draw [thick] (0,1) --(8,1);  
\draw [thick, blue] (3,0) -- (2,2);
\draw[thick, blue] (6,0)--(5,2);
\draw [dashed](5,2) to [out=120,in=0] (0,3) ;
\draw [dashed](5,2) to [out=120,in=0] (0,3.25) node [left] {$B$} ;
\draw [dashed](5,2) to [out=120,in=0] (0,3.4) ;
\draw [dashed](3,0) to [out=-60,in=180] (8,-1)  ;
\draw [dashed](3,0) to [out=-60,in=180] (8,-1.25)  node [right] {$C$} ;
\draw [dashed](3,0) to [out=-60,in=180] (8,-1.4);
\draw [dashed](6,0) to [out=-60,in=180] (8,-2.6);
\draw [dashed](6,0) to [out=-60,in=180] (8,-2.8) node [right]{$D$};
\draw [dashed](6,0) to [out=-60,in=180] (8,-3) ;
\draw [dashed](2,2) to [out=120,in=0] (0,4.6) ;
\draw [dashed](2,2) to [out=120,in=0] (0,4.8) node [left]{$A$};
\draw [dashed](2,2) to [out=120,in=0] (0,5);
\end{scope}
\end{tikzpicture}.
\caption{}
\label{fig: sym group action fixed points}
\end{figure}

For $\sigma \in S_{n}$, we define the action on the set of equivariant parameters by
\[
\sigma \cdot a_{i} =a_{\sigma^{-1}(i)}.
\]

\begin{proposition}\label{prop: permute vertex}
With the notation as above, we have
    \[
\ver^{X}_{f}(a,Q)=\ver_{\sigma \cdot f}^{\sigma \cdot X}(\sigma \cdot a,Q).
\]
\end{proposition}
\begin{proof}
    For definiteness, suppose that $\D$ is separated. While the butterfly diagrams for $f$ and $\sigma\cdot f$ differ, the ``left halves" of the two butterfly diagrams, i.e. the maximal subquivers of $B^{f}$ and $B^{\sigma \cdot f}$ whose vertices $b$ are such that $\col(b)$ is adjacent to two NS5 branes, are canonically isomorphic. Moreover, under this isomorphism, the weights of the vertices are compatible under $\sigma$. Now the statement follows from Proposition \ref{prop: vertex formula}. 
\end{proof}

\begin{remark}
    It is shown in \cite{rimanyi2020bow} that a bow variety is Hanany-Witten equivalent to a quiver variety if and only if the charges of D5 branes are weakly decreasing. By combining this with Propositions \ref{prop: HW vertex} and \ref{prop: permute vertex}, we see that the vertex function of any bow variety is equal to the vertex function of a quiver variety, up to permutation of the equivariant variables. 
\end{remark}

\section{D5 resolutions}\label{sec: d5 resolutions}

In this section, we study D5 resolutions of bow varieties. These are certain closed embeddings of bow varieties, constructed locally on the brane diagram, which reduce the weight of a chosen D5 brane. They are the geometric manifestation of the fusion procedure for $R$-matrices. In addition, they are responsible for the so-called resonance of Bethe equations studied in \cite{KorZeit2}. 

A crucial property proven in \cite{BR} is that stable envelopes are compatible with D5 resolutions. We review the constructions and results of \cite{BR}. Then we prove in Theorem \ref{thm: d5 vertex} that vertex functions are also compatible with D5 resolutions.

\subsection{D5 resolutions of bow varieties}

Let $\D$ be a separated or co-separated brane diagram.
Let $\wt \D$ be the brane diagram obtained by replacing a single D5 brane $\Ab$ of weight $\w=\w(\Ab)\geq 2$ in $\D$ by a pair of consecutive D5 branes $\Ab'$ and $\Ab''$ of weights $\w'=\w(\Ab')\geq 1$ and $\w''=\w(\Ab'')\geq 1$ such that $\w=\w'+\w''$. Pictorially, the diagram $\wt \D$ is obtained via the local surgery
\begin{center}
\begin{tikzpicture}[scale=.7]
\draw[thick] (-1,0) -- (1,0);
\draw[thick, blue] (0.2,-0.5) -- (-0.2,0.5);
\node[label={\small $\w'+\w''$}] at (-0.2,0.5) {};

\draw[stealth-stealth, thick] (2,0) -- (3,0);

\draw[thick] (4,0) -- (7,0);
\draw[thick, blue] (5.2,-0.5) -- (4.8,0.5);
\draw[thick, blue] (6.2,-0.5) -- (5.8,0.5);
\node[label={\small $\w'$}] at (4.8,0.5) {};
\node[label={\small $\w''$}] at (5.8,0.5) {};

\end{tikzpicture}
\end{center}

We call $\wt \D$ a D5 resolution of the brane diagram $\D$ and the branes $\Ab'$ and $\Ab''$ resolving branes. Notice that if $\D$ is separated (resp. co-separated), then $\wt \D$ is also separated (resp. co-separated).

Let now $\wt X$ and $X$ be the bow varieties associated to $\wt \D$ and $\D$. They are open subsets of the stacks $ \wt{\FX}=\wt{M}/\wt{G}$ and $\FX=M/G$, respectively. We say that $\wt X$ is a D5 resolution of the brane $\Ab$ in the bow variety $X$. 

Let $\Tt=\At\times \Cs_{\hbar}$ (resp. $\wt \Tt= \wt \At \times \Cs_{\hbar}$) be the torus acting on $X$ (resp. $\wt X$). We define a homomorphism $\varphi: \Tt\to \wt \Tt$ to be the identity on most components, except 
\begin{equation}
\label{group homomorphism A-resolution}
    \Cs_{\Ab} \times \Cs_{\hbar}  \to  \Cs_{\Ab'} \times  \Cs_{\Ab''} \times \Cs_{\hbar} \qquad (a,\hbar)\mapsto \begin{cases}
        (a\hbar^{-\w''}, a, \hbar ) & \text{if $\D$ is separated}\\
        (a, a\hbar^{\w'}, \hbar ) & \text{if $\D$ is co-separated.}
    \end{cases}
\end{equation}
The following theorem sums up the main results in {\cite[\S5]{BR}}.
\begin{theorem}
\label{theorem: embedding resolution D5 branes}
    Let $\D$ be a separated bow diagram and let $\wt\D$ be a D5 resolution of some brane $\Ab$. Associated to these data, there exists a regular closed embedding $j: \FX\hookrightarrow \wt{\FX}$ induced by a pair of morphisms 
    \[
    M\to \wt{M}\qquad G\to \wt{G}.
    \]
    Moreover, the following statements hold:
    \begin{itemize}
        \item The map $j$ is equivariant along $\varphi:\Tt\to \wt \Tt$.
        \item The map $j$ preserves the semistable loci, and hence induces a regular embedding $j: X\hookrightarrow \wt{X}$, which with slight abuse of notation we denote in the same way.
        \item The morphisms $j$ are pulled back from a regular embedding $HS\hookrightarrow\wt{HS}$. In other words, there exist diagrams all of whose squares are Cartesian.
        \[
        \begin{tikzcd}
    X\arrow[r, hookrightarrow, "j"]\arrow[d, "\pi"] & \widetilde X \arrow[d, "\widetilde \pi "]\\
    X_0\arrow[r, hookrightarrow]\arrow[d, "\rho"] & \widetilde X_0 \arrow[d, "\tilde 
    \rho"]\\
    HS\arrow[r, hookrightarrow] & \widetilde{HS}
\end{tikzcd}
\qquad 
\begin{tikzcd}
    \FX \arrow[r, hookrightarrow, "j"]\arrow[d, "\pi"] & \wt \FX \arrow[d, "\widetilde \pi "]\\
    X_0\arrow[r, hookrightarrow]\arrow[d, "\rho"] & \widetilde X_0 \arrow[d, "\tilde 
    \rho"]\\
    HS\arrow[r, hookrightarrow] & \widetilde{HS}
\end{tikzcd}
        \]
    \end{itemize}
\end{theorem}

\subsection{Fixed points and D5 resolutions}

Let $j: X\hookrightarrow \wt X$ be a D5 resolution of the brane $\Ab$ in the bow variety $X$. Set $\w=\w(\Ab)$. We consider both varieties as $\Tt=\At\times \Cs_{\hbar}$ varieties via the map $\varphi: \Tt\to \tilde{\Tt}$. 

For a given fixed point $f\in X^{\At}$, let $F\subset \wt X^{\At}$ be the unique fixed component containing $j(f)$.

\begin{lemma}[{\cite{BR}}]
    If $X$ is separated, then the fixed component $F$ is isomorphic to the bow variety 
    \ttt{\fs 1\fs 2\fs \dots \fs$\w-1$\fs $\w$\bs$\w''$\bs}. If, instead, $X$ is co-separated, then $F$ is isomorphic to the bow variety \ttt{\bs $\w'$\bs $\w$\fs $\w-1$\fs \dots  \fs2\fs1\fs}.
\end{lemma}

Consider the residual action of $\wt\At$ on $F$. We call the fixed points $\tilde f\in F^{\wt \At}$ resolutions of $f$. We remark that none of them coincides with $j(f)$. However, there is a distinguished representative. If $X$ is separated, we denote by $\tilde f_\sharp\in F^{\wt A}$ the unique fixed point whose tie diagram has no crossings. If 
$X$ is co-separated, we denote by $\tilde f_\sharp\in F^{\wt A}$ the unique fixed point whose tie diagram has $\w'\w''$ crossings, as illustrated in Figure \ref{fig: diagrams distinguished resolving fixed points D5 case}.

 \begin{figure}
 \centering
\begin{tikzpicture}[scale=.3]
\draw [thick] (0,1)node[left]{$f=$}  --(7,1);
\draw [thick, blue] (4,0) node [right]{$a$} -- (3,2);
\draw [dashed](3,2) to [out=120,in=0] (0,4.6) ;
\draw [dashed](3,2) to [out=120,in=0] (0,4.8) ;
\draw [dashed](3,2) to [out=120,in=0] (0,5) node [left]{$\w'+\w'' \Big\lbrace$};
\draw [dashed](3,2) to [out=120,in=0] (0,5.2);
\draw [dashed](3,2) to [out=120,in=0] (0,5.4);
\draw[ultra thick, <->] (9.5,1)--(12.5,1) node[above]{separated} node[below]{D5 res.} -- (15.5,1);
\begin{scope}[xshift=22cm]
\draw [thick] (0,1) node[left]{$\tilde{f}_\sharp=$} --(8,1);  
\draw [thick, blue] (3,0) node [right]{$a'$} -- (2,2);
\draw[thick, blue] (6,0) node [right]{$a''$} --(5,2);
\draw [dashed](2,2) to [out=120,in=0] (0,3) ;
\draw [dashed](2,2) to [out=120,in=0] (0,3.2) node [left] {$\w' \lbrace$} ;
\draw [dashed](2,2) to [out=120,in=0] (0,3.4) ;
\draw [dashed](5,2) to [out=120,in=0] (0,5.6) node [left]{$\w'' \lbrace$};
\draw [dashed](5,2) to [out=120,in=0] (0,5.8);
\end{scope}

\draw [thick] (0,-4) node[left]{$f=$} --(7,-4);
\draw [thick, blue] (4,-5) node [left]{$a$} -- (3,-3);
\draw [dashed](4,0-5) to [out=-60,in=180] (7,-7.6) ;
\draw [dashed](4,-5) to [out=-60,in=180] (7,-7.8) ;
\draw [dashed](4,-5) to [out=-60,in=180] (7,-8) node [right]{$\Big\rbrace \w'+\w''$};
\draw [dashed](4,-5) to [out=-60,in=180] (7,-8.2);
\draw [dashed](4,-5) to [out=-60,in=180] (7,-8.4);
\draw[ultra thick, <->] (9.5,-4)--(12.5,-4) node[above]{co-separated} node[below]{D5 res.} -- (15.5,-4);

\begin{scope}[xshift=22cm]
\draw [thick] (0,-4) node[left]{$\tilde{f}_\sharp=$} --(8,-4);  
\draw [thick, blue] (3,-5) node [left]{$a'$} -- (2,-3);
\draw[thick, blue] (6,-5) node [left]{$a''$} --(5,-3);
\draw [dashed](3,-5) to [out=-60,in=180] (9,-7) ;
\draw [dashed](3,-5) to [out=-60,in=180] (9,-7.2) node [right] {$\rbrace \w'$} ;
\draw [dashed](3,-5) to [out=-60,in=180] (9,-7.4) ;
\draw [dashed](6,-5) to [out=-60,in=180] (9,-9.6) node [right]{$\rbrace \w''$};
\draw [dashed](6,-5) to [out=-60,in=180] (9,-9.8);
\end{scope}
\end{tikzpicture}
\caption{}
\label{fig: diagrams distinguished resolving fixed points D5 case}
 \end{figure}
 
\begin{remark}
    With respect to the standard chamber, the fixed point $\tilde f_{\sharp}$ is maximal among the D5 resolutions of $f$.
\end{remark}

Given a bow variety $X$, we denote by $K_{\Tt}(X)^{\taut}\subset K_{\Tt}(X)$ the sub-ring generated by the tautological bundles from \S\ref{sec:def of bow variety}.
\begin{lemma}[\cite{BR}]
\label{lemma: restrction Chern roots in A resolution}
    Let $\mathcal \tb$ be a tautological bundle of $\widetilde X$. Then $j^* \tb\big|_f= \tb\big|_{\tilde f_\sharp}$ as representations of $\Tt$. As a consequence, the diagram
    \[
    \begin{tikzcd}
        K_{ \widetilde \Tt}(\widetilde X)^{\taut}\arrow[d, "j^*\circ \varphi^*"] \arrow[r] &K_{ \widetilde \Tt}(\tilde f_\sharp)\arrow[d, "\varphi^*"]\\
        K_{\Tt}(X)^{\taut} \arrow[r] &K_{\Tt}(f)\\
    \end{tikzcd}
    \]
    associated with fixed point restriction at $\tilde f_\sharp$ and $f$ and the change of group map $\varphi:  \Tt\to \widetilde \Tt$ is commutative.
\end{lemma}
\begin{proof}
    The statement follows from the combinatorics of the fixed point restrictions developed in \cite[\S4.4]{rimanyi2020bow}. Alternatively, it can be checked that $j(f)$ lies in the attracting set of $\tilde f_{\sharp}$ for the action of $\Cs_{\hbar}$ on $\wt X$ induced by the composition $\Cs_{\hbar}\hookrightarrow\Tt\xrightarrow{\varphi}\wt\Tt$. This implies the statement.
\end{proof}

\subsection{D5 resolutions of vertex functions}

In this section, we relate the moduli spaces of quasimaps to $\wt X$ and $X$. Fix an arbitrary separated or co-separated bow variety $X$ and consider the chain of maps
\[
\qm(X)\xrightarrow{\ev_p}  \FX \xrightarrow{\pi} X_0\xrightarrow{\rho} HS.
\]
Restricting to non-singular quasimaps, we get morphisms 
\[
\qm_{\ns p}(X)\xrightarrow{\ev_p}  X \xrightarrow{\pi} X_0\xrightarrow{\rho} HS.
\]
The embedding $j:\FX\hookrightarrow \wt \FX$ induces a morphism
\begin{equation}
    \label{eq: D5 map quasimap spaces}
    J: \qm_{}(X)\to \qm_{}(\wt X).
\end{equation}
Tautologically, $J$ preserves the connected components so it is given by a tower of morphisms 
\[
J_d: \qm^{d}(X)\to \qm^{d}(\wt X).
\]
Furthermore, the morphism $J$ fits in the following commutative diagram
\begin{equation}
\label{eq: diagram D5 quasimaps}
    \begin{tikzcd}
        \qm(X)\arrow[r, "{\ev_p}"]\arrow[d, swap, "J"]  &  \FX\arrow[d, swap, "j", hookrightarrow] \arrow[r] & X_0\arrow[r]\arrow[d, hookrightarrow]  & {HS}\arrow[d, hookrightarrow, "j_0"]\\
        \qm(\wt X)\arrow[r, "{\ev_p}"]  &  \wt\FX\arrow[r] & \wt X_0\arrow[r]  & \wt{HS}
    \end{tikzcd}
\end{equation}
Recall that Theorem \ref{theorem: embedding resolution D5 branes} states that the two rightmost squares in this diagram are Cartesian.
\begin{lemma}
\label{lemma: Quasimap cartesian}
    The leftmost square in diagram \eqref{eq: diagram D5 quasimaps} above is also Cartesian. 
\end{lemma}
\begin{proof}
Since the central square in the diagram above is Cartesian, a morphism $\bbP^1\to \wt \FX$ factors through $\FX\hookrightarrow \wt \FX$ if and only if the composition $\bbP^1\to \wt \FX\to \wt X_0$ factors through $X_0\hookrightarrow \wt X_0$. But the latter varieties are affine so this holds if and only if the evaluation of $\bbP^1\to \wt \FX\to \wt X_0$ at $p$ belongs to $X_0$. Therefore, the statement follows from the definition of the moduli space of quasimaps.

\end{proof}

By considering the outer frame of the diagram above we obtain the Cartesian diagram 
\begin{equation*}
    \begin{tikzcd}
        \qm(X)\arrow[d]\arrow[r, swap, "J"] & \qm(\wt X)\arrow[d]\\
        {HS}\arrow[r,  hookrightarrow, "j_0"]  & \wt{HS}
    \end{tikzcd}
\end{equation*}
where the vertical maps do not depend on the choice of $p$. The morphism $j_0$ is regular, so it induces a refined pullback 
\[
j_0^!: K_{\wt \Tt}(\qm(\wt X))\to  K_{\Tt}(\qm( X)).
\]

\begin{lemma}
\label{lemma: pullback along j_0}
    We have 
    \[
    j_0^! \nvrs^{\qm(\wt X)} = \nvrs^{\qm(X)}.
    \]
\end{lemma}

\begin{proof}
    The lemma will follow from a standard comparison of the obstruction theories defined in \eqref{eq: POT}. Consider the Cartesian square
    \[
    \begin{tikzcd}
       \FX\arrow[d, " \rho\circ \pi"]\arrow[r, hookrightarrow, "j"] &\wt \FX \arrow[d, "\wt \rho\circ \wt \pi "]\\
        HS\arrow[r, hookrightarrow] & \wt{HS}
    \end{tikzcd}
    \]
    and the following distinguished triangle on $\FX$:
    \[
    \Tt_{\FX}\to  j^*\Tt_{\tilde\FX}\to \Tt_{\wt \FX/\FX}.
    \]
    By Theorem \ref{theorem: embedding resolution D5 branes}, the embedding $\FX\to \wt\FX$ is regular and $\Tt_{\wt\FX/\FX}$ is quasi-isomorphic to the normal bundle $(\rho \circ \pi)^* N_{\wt{HS}/HS}$. Recall now the maps in \eqref{eq: diagram universal evaluation}. Applying $u^{*}$ followed by $ R^\bullet pr_*$ we get a triangle
    \begin{equation*}
        \begin{tikzcd}
            R^\bullet pr_* u^* \Tt_{\FX}\arrow[r] &  R^\bullet pr_* u^* j^*\Tt_{\wt \FX}   \arrow[r] & R^\bullet pr_* u^* (\rho \circ \pi)^* N_{\wt{HS}/HS}.
        \end{tikzcd}
    \end{equation*}
    Since $ HS$ is affine, we have a factorization
    \[
    \begin{tikzcd}
    \bbP^1\times \qm(X) \arrow[r, "u"] \arrow[d, "pr"] & \FX \arrow[r, "\rho \circ \pi"] & HS\\
    \qm(X)\arrow[urr, swap, "q"]
    \end{tikzcd}
    \]
    So we can rewrite the triangle above as follows:
    \begin{equation*}
        \begin{tikzcd}
            R^\bullet pr_* u^* \Tt_{\FX}\arrow[r] &  R^\bullet pr_* u^* j^* \Tt_{\wt \FX}   \arrow[r] & R^\bullet pr_*  pr^* \left(q^*N_{\wt{HS}/HS}\right)
        \end{tikzcd}
    \end{equation*}
    Applying the projection formula and proper base change
    , it follows that $R^{i} pr_*  pr^* \left(q^*N_{\wt{HS}/HS}\right)=0$ for all $i>0$ and $q^*N_{\wt{HS}/HS}\to  pr_*  pr^* \left(q^*N_{\wt{HS}/HS}\right)$ is an isomorphism. 
    Moreover, from the diagram
    \[
    \begin{tikzcd}
        \bbP^1\times \qm(X)\arrow[d, "u"]\arrow[r, "id\times J"] & \bbP^1\times  \qm(\wt X)\arrow[d, "\tilde{u}"]\\
    \FX\arrow[r,  hookrightarrow, "j"]  & \wt \FX
    \end{tikzcd}
    \]
we deduce that 
    \begin{equation}
        \label{eq: intermediate base change}
        R^\bullet pr_* u^* j^*\Tt_{\wt \FX}=R^\bullet pr_* (id \times J)^*\tilde u^* \Tt_{\wt \FX}
    \end{equation}
    and by derived flat base change
    we conclude that \eqref{eq: intermediate base change} is quasi-isomorphic to 
    \[
    J^* R^\bullet \wt{pr}_* \tilde u^*\Tt_{\wt \FX}.
    \] 
    Therefore, we obtain a morphism of triangles 
    \begin{equation}
    \label{eq: compatibility POT}
        \begin{tikzcd}
            R^\bullet pr_* u^* \Tt_{\FX}\arrow[r] &  J^* R^\bullet \wt{pr}_* \tilde u^*\Tt_{\wt \FX}  \arrow[r] & q^*N_{\wt{HS}/HS}\\
            \Tt_{\qm(X)}\arrow[u]\arrow[r] & J^*\Tt_{\qm(\wt X)}\arrow[u]\arrow[r] &  \Tt_{\qm(\wt X)/ \qm( X)}\arrow[u]
        \end{tikzcd}
    \end{equation}
    Commutativity of the squares follows from functoriality of the tangent complex. As a result, the obstruction theories for $\qm(X)$ and $\qm(\wt X)$ are compatible over $j_0$. The statement then follows from \cite[Prop. 4]{lee_quantumktheory} (cf. \cite[Prop. 5.10]{BF97} for the analogous result in cohomology).
\end{proof}
\begin{remark}
    The statement of the previous lemma can be seen as the quasimap version of the comparison statement for the virtual fundamental classes (or structure sheaves) of the moduli space of genus zero stable maps to $\bbP^n$ and some divisor $X\subset \bbP^n$, see \cite[\S8.5]{manolache_working} for an overview. Notice also that the proof fails in higher genus because the higher derived pushforwards may not vanish.
\end{remark}

We now restrict ourselves to the substack of non-singular maps. Then we have a diagram 
\begin{equation}
\label{eq: diagram D5 quasimaps ns}
    \begin{tikzcd}
        \qm_{\ns p}(X)\arrow[r, "{\ev_p}"]\arrow[d, swap, "J"]  &  X\arrow[d, swap, "j", hookrightarrow] \arrow[r]  & {HS}\arrow[d, hookrightarrow, "j_0"]\\
        \qm_{\ns p}(\wt X)\arrow[r, "{\wt{\ev}_p}"]  &  \wt X\arrow[r]  & \wt{HS}
    \end{tikzcd}
\end{equation}
whose squares are are Cartesian. Hence, we get two a priori different virtual pullbacks 
\begin{align*}
    j_0^!&: K_{\wt \Tt}(\qm_{\ns p}(\wt X))\to  K_{\Tt}(\qm_{\ns p}( X))\\
    j^!&: K_{\wt \Tt}(\qm_{\ns p}(\wt X))\to  K_{\Tt}(\qm_{\ns p}( X)).
\end{align*}
However, the morphisms $ HS\hookrightarrow \wt{HS}$ and $\wt X\to \wt{HS}$ are transversal, so we get
\begin{lemma}
    These two maps are the same, i.e. $j_0^!=j^!$.
\end{lemma}
Combining this result with Lemma \ref{lemma: pullback along j_0} we deduce the following:
\begin{corollary}
\label{cor: D5 vrs}
    We have 
    \[
    j^! \nvrs^{\qm_{\ns p}(\wt X)} = \nvrs^{\qm_{\ns p}(X)}.
    \]
\end{corollary}


We can now state the main result of this section.

\begin{theorem}\label{thm: d5 vertex}
    We have\footnote{We use the symbol $j^*$ to denote the K-theoretic pullback $K_{\Tt}(\wt X)\to K_{\Tt}(X)$ along the regular embedding $j$. This should not be confused with $j^!$, which also depends on $j$ but relates the K-theory of the quasimap spaces.}
    \[
     j^* \varphi^* V^{\wt X}= V^{X}.
    \]
\end{theorem}
\begin{proof}
To remove clutter, we drop the specialization map $\varphi^*$ from the notation. Equivalently, we view the vertex function $V^{\wt X}$ as a class in $\Tt$-equivariant $K$-theory of $\wt X$ rather than in $\wt \Tt$ equivariant $K$-theory. The top triangle in diagram \eqref{eq: compatibility POT} implies that\footnote{Notice that the last equality follows from our normalization of the virtual canonical bunldes, which we twist by the anticanonical bundle of the target. Without the latter, the two sides would differ by $q^* \Can^{\wt{HS}/HS}=(\ev_p)^*\Can^{\wt{X}/X}$.}
\begin{align*}
         &\frac{\det \qmpol_{\wt X}|_{p}}{\det 
         \qmpol_{\wt X}|_{p'}}=\frac{\det \qmpol_X|_{p}}{\det\qmpol_X|_{p'}} 
         \\
         &J^*\Can_{\vir}^{\qm(\wt X)}= \Can_{\vir}^{\qm(X)}.
\end{align*}
Therefore, the same logic of the proof of Corollary \ref{cor: D5 vrs} applied to the symmetrized virtual structure sheaves (cf. \eqref{symvss}) implies that 
\[
    j^! \vrs^{\qm_{\ns p}(\wt X)} = \vrs^{\qm_{\ns p}(X)}.
\]
But virtual pullbacks commute with pushforward in Cartesian squares, so applying this to the leftmost square in diagram \eqref{eq: diagram D5 quasimaps ns} we get
\begin{align*}
    j^*  V^{\wt X} &=j^* (\wt{\ev}_p)_* \left(\vrs^{\qm_{\ns p}(\wt X)} \right)
    \\
    &= (\ev_p)_*j^! \left(\vrs^{\qm_{\ns p}(\wt X)} \right)
    \\
    &= (\ev_p)_*\left( \vrs^{\qm_{\ns p}(X)} \right)
    \\
    &= V^{ X},
\end{align*}
as claimed.
\end{proof}

Combining the previous statement with Lemma \ref{lemma: restrction Chern roots in A resolution}, we deduce the last result of this section.
\begin{corollary}\label{cor:D5resvertex}
    Let $f \in X^{\At}$ and $\tilde{f}_{\sharp} \in \tilde{X}^{\tilde{\At}}$ be its distinguished resolution. Then
    \[
    \ver^{X}_{f}(Q)=\varphi^{*}\ver^{\tilde{X}}_{\tilde{f}_{\sharp}}(Q).
    \]
\end{corollary}
\begin{remark}
For the special case when $X$ is a cobalanced bow variety, i.e. a Nakajima quiver variety, the maximal D5 resolution $\wt{X}$ is the cotangent bundle to a partial flag variety, which is also a Nakajima quiver variety. In this case, Corollary \ref{cor:D5resvertex} was proven earlier by the second author in \cite{dinkinsD5Vertex} using combinatorial properties of the localization formula. Similar techniques could also be applied here, in which case the combinatorics is much simpler than in \cite{dinkinsD5Vertex} for two reasons. First, the localization formula for separated (or co-separated) bow varieties is more uniform than the localization formula for quiver varieties. Second, in this paper we are able to study the situation where a single D5 brane is resolved, whereas the situation in \cite{dinkinsD5Vertex} must consider the composition of many D5 resolutions to avoid leaving the setting of quiver varieties. This is one of the many places where studying bow varieties is both easier and more general than quiver varieties.
\end{remark}

\section{NS5 resolutions}\label{sec: ns5 resolutions}

In this section, we study the mirror dual of the constructions of \S\ref{sec: d5 resolutions}. To that end, we review the NS5 resolutions constructed in \cite{BR}. NS5 resolutions provide a way of reducing the weight of a chosen NS5 brane. A bow variety and an NS5 resolution of it fit into a certain diagram which is constructed locally on the brane diagram. The geometry of NS5 resolutions is more complicated than that of D5 resolutions.

\subsection{NS5 resolutions of bow varieties}

Let $\D$ be a separated or co-separated brane diagram.
Let $\overline \D$ be the brane diagram obtained by replacing a single NS5 brane $\Zb$ of weight $\w=\w(\Zb)\geq 2$ in $\D$ by a pair of consecutive NS5 branes $\Zb'$ and $\Zb''$ of weights $\w'=\w(\Zb')\geq 1$ and $\w''=\w(\Zb'')\geq 1$ such that $\w=\w'+\w''$.
We call $\overline \D$ an NS5 resolution of the brane diagram $\D$ and the branes $\Zb'$ and $\Zb''$ resolving branes. Notice that if $\D$ is separated (resp. co-separated), then $\overline \D$ is also separated (resp. co-separated).

Let now $\overline X$ and $X$ be the bow varieties associated with $\overline \D$ and $\D$, respectively. We say that $\overline X$ is an NS5 resolution of the bow variety $X$. The same torus $\Tt$ acts on both of them. In this section, we review the construction of a distinguished Lagrangian correspondence 
\begin{equation}
    \label{eq: NS5 correspondence}
    \overline X\leftarrow L\rightarrow X.
\end{equation}
As mentioned in the introduction, in the event that the bow variety $X$ is the cotangent bundle of a partial flag variety $X=T^*\Fl$, then $\ol X$ is the cotangent bundle of a ``finer'' flag variety with an additional step, and the Lagrangian correspondence above is the canonical correspondence 
\begin{equation*}
    T^*\Fl\leftarrow L\to T^*\ol \Fl
\end{equation*}
induced by the pullback diagram
\[
\begin{tikzcd}
    L \arrow[r]\arrow[d] & T^* \Fl\arrow[d]\\
    \ol \Fl \arrow[r, "f"] & \Fl
\end{tikzcd}
\]
where $f$ is just the map forgetting the extra step in the flag.
More generally, the space $L$ and the morphisms above can be defined by modifying
the quiver data giving rise to $X$. For simplicity, assume that $X$ is separated. Then to pass from $X$ to $\overline X$ one replaces the two way part of the diagram of $X$ attached to the brane $\Zb$ as follows:
\[
\begin{tikzcd}
    W_{\Zb_-}  \arrow[r, bend right, swap, "D"]& W_{\Zb_+} \arrow[l, bend right, swap, "C",  "\circ" marking]
\end{tikzcd}
\quad \Longrightarrow \quad 
\begin{tikzcd}
    W_{\Zb_-}=W_{\Zb'_-}  \arrow[r, bend right, swap, "D'"]&  W_{\Zb'_+}=W_{\Zb''_-} \arrow[r, bend right, swap, "D''"]\arrow[l, bend right, swap, "C'",  "\circ" marking] & W_{\Zb''_+} =W_{\Zb_+} \arrow[l, bend right, swap, "C''",  "\circ" marking]
\end{tikzcd}.
\]
The circles in the arrows indicate the rescaling action by $\hbar$. Similarly, the space $L$ is defined in terms of $X$ as follows
\begin{equation*}
\begin{tikzcd}
    W_{\Zb_-}  \arrow[r, bend right, swap, "D"]& W_{\Zb_+} \arrow[l, bend right, swap, "C",  "\circ" marking]
\end{tikzcd}
\quad \Longrightarrow \quad 
\begin{tikzcd}
    W_{\Zb_-}=W_{\Zb'_-}  \arrow[rr, bend right, swap, "D"]&  W_{\Zb'_+}=W_{\Zb''_-} \arrow[l, bend right, swap, "C'"] & W_{\Zb''_+} =W_{\Zb_+}   \arrow[l, bend right, swap, "C''",  "\circ" marking]
\end{tikzcd}.
\end{equation*}

The maps in diagram \eqref{eq: NS5 correspondence} are then easily described in terms of these diagrams. For instance, the map $L\to X$ is induced by the assignment $C=C'\circ C''$. The case when $X$ is co-separated is analogous.\footnote{In the co-separated case the roles of the maps $C$ and $D$ are swapped to take into account the stability condition.} We refer to \cite[\S7]{BR} for further details.

\begin{proposition}[\cite{BR}]
\label{proposition NS5 resolution bow varieites}
     The morphisms 
     \begin{equation*}
         \begin{tikzcd}
       \overline X & L \arrow[l, swap, "j"] \arrow[r, "p"] & X.
       \end{tikzcd}
     \end{equation*}
     are well defined and $\Tt$-equivariant. Moreover, they enjoy the following properties:
     \begin{itemize}
         \item The map $j$ is a closed immersion and $p$ is proper. 
         \item If $X$ is separated (resp. co-separated) then the fibers of $p$ are isomorphic to the Grassmannian $\Gr(\w', \w'+\w'')$ (resp. $\Gr(\w'', \w'+\w'')$). 
         \item  There exists a commutative diagram 
        \[
        \begin{tikzcd}
            \overline X \arrow[d, swap,  "\bar\pi"]& L\arrow[d] \arrow[l, swap, "j"] \arrow[r, "p"] & X \arrow[dl, "\pi"] \\
            \overline X_0 &  X_0 \arrow[l] & 
        \end{tikzcd}
        \]
        Moreover, the left square is Cartesian.
     \end{itemize}
\end{proposition}

\subsection{The equivariant geometry of NS5 resolutions}\label{sec: equivariant geometry of ns5}

In Proposition \ref{proposition NS5 resolution bow varieites}, we identified the fibers of the map $p:L \to X$ with Grassmannians. We begin this section by refining our analysis of these fibers. For a given separated bow variety $X$ and NS5 resolution $\overline X$ such that $\w=\w'+\w''$, we set $Y:=\ttt{\fs$\w'$\fs$\w'+ \w''$\bs\dots \bs 3\bs2\bs1\bs}$. If instead $X$ is co-separated, we set $Y:=\ttt{\bs1\bs2\bs3\bs\dots\bs$\w'+\w''${\fs}$\w''${\fs}}$.

\begin{lemma} $ $
\label{lemma points in lagrangian variety Z resolution}
\begin{enumerate}
    \item Let $f\in X^{\At}$. The $\Cs_{\hbar}$-fixed locus $Y^{\hbar}$ of the bow variety $Y$ fits in the following pullback diagram:
    \begin{equation*}
        \begin{tikzcd}
            Y^{\hbar}\arrow[r, hookrightarrow]\arrow[d] & L\arrow[d, "p"]\\
            \lbrace f\rbrace \arrow[r, hookrightarrow] &X
        \end{tikzcd}
    \end{equation*}
    \item Any fixed point $f\in X^{\At}$ admits exactly $\binom{\w'+\w''}{\w''}$ resolutions, i.e. fixed points $\ol f \in \overline X^{\At}\cap L$ such that $p(\ol{f})=f$.
    \item For any fixed point $f\in X^{\At}$, we have $T\overline X|_{Y^{\hbar}}-TL|_{Y^{\hbar}}=TY|_{Y^{\hbar}}-TY^{\hbar}$ in $K_{\Tt}(Y^{\hbar})$.
\end{enumerate}
\end{lemma}

\begin{remark}
\label{remark torus action on fibers of NS5 resolutions}
    The action of $\Tt$ on $Y^{\hbar}$ (and hence on $Y$ since $Y=T^*Y^{\hbar}$) induced by part (1) of Lemma \ref{lemma points in lagrangian variety Z resolution} is a twist of the standard action on $Y$ seen as a bow variety. Indeed, if $X$ is separated (resp. co-separated), then $\Tt$ acts on the $i$-th D5 brane of $Y$ with weight equal to $a_i \hbar^{-\gamma_i}$ (resp. $a_i h^{\gamma_i}(f)$), where $\gamma_i(f)$ is equal to the number of ties in the tie diagram of $f\in X^{\At}$ that are connected to the NS5 branes left of the resolved NS5 brane $\Zb$. This again follows from the explicit description of $f\in X^{\At}$ given in \cite[\S4.3]{rimanyi2020bow}.
\end{remark}

\begin{remark}
\label{remark resolution of fixed points NS5 resolutions}
By part (2) of Lemma \ref{lemma points in lagrangian variety Z resolution}, the resolutions of $f\in X^A$ are in one-to-one correspondence with fixed points in $Y$. In terms of tie diagrams, the correspondence sends the tie diagram of $\ol f\in \ol X$ to the tie diagram of $Y=\ttt{\fs$\w'$\fs$\w'+ \w''$\bs\dots \bs 3\bs2\bs1\bs}$ (resp. $Y=\ttt{\bs1\bs2\bs3\bs\dots\bs$\w'+\w''${\fs}$\w''${\fs}}$) obtained by erasing all the branes and ties not connected to $\Zb'$ or $\Zb''$. By slightly abusing notation, we will still denote the resulting fixed point in  $Y$ by $\ol f$.
\end{remark}

\begin{corollary}
    Let $f\in X^{\At}$ and $\ol f\in \ol X^{\At}$ be a resolution of $f$. Then  
    \[
    T_{\ol f}\ol X-T_fX= T_{\ol f} Y
    \]
    as virtual $\Tt$-representations.
\end{corollary}
\label{cor: different tangents NS5}
\begin{proof}
    We compute 
    \begin{align*}
         T_{\ol f} \ol X-T_fX 
         &= T_{\ol f} \ol X - T_{\ol f}L +T_{\ol f }L -T_fX 
         \\
         &= T_{\ol f } Y  -T_{\ol f }Y^{\hbar} + T_{\ol f }L -T_fX 
         \\
         &=T_{\ol f } Y  -T_{\ol f }Y^{\hbar} +  T_{\ol f }Y^{\hbar}
         \\
         &=T_{\ol f } Y.
    \end{align*}
    In the second step we used the third point Lemma \ref{lemma points in lagrangian variety Z resolution} and in the third step the first point.
\end{proof}

Recall the definition of full attracting set from Remark \ref{rem: full att set}.

\begin{lemma}
\label{lemma supports NS5 resolutions}
Let $f\in X^{\At}$ and let $\ol f\in \overline X^{\At}$ be any resolution of the fixed point $f$. Then $\ol f'\not\in \Attfull{C}^{\overline X}(\bar f)$ whenever $\ol f'\in (\ol X\setminus L)^{\At}$.
\end{lemma}
\begin{proof}
    Assume the converse. Then there exists a equivariant curve $\mathbb{P}^1\to \ol X$ connecting $\ol f$ to $\ol f'$. But $\ol f\in L$ and $\ol f'\in \ol X\setminus L$, so the Cartesian diagram in Proposition \ref{proposition NS5 resolution bow varieites} implies that the composition $\mathbb{P}^1\to \ol X\to \ol X_0$ is non-constant, which cannot be because $\ol X_0$ is affine.
\end{proof}
\begin{corollary}\label{cor: NS5 stab restrictions}
Let $f\in X^{\At}$ and let $\ol f\in \ol X^{\At}$ be any resolution of $f$. Then $\Stab(\ol f)|_{\ol g}=0$ for all $\ol g\in (\ol X\setminus L)^A$.
\end{corollary}

\subsection{NS5 resolutions of stable envelopes}
\label{sub: NS5 resolutions of stable envelopes}

Fix a separated or co-separated brane diagram $\D$ and let $X=X(\D)$ be the associated bow variety. By Lemma \ref{lemma points in lagrangian variety Z resolution}, an arbitrary fixed point $f\in X^{\At}$ admits exactly $\binom{\w'+\w''}{\w'}$ possible resolutions $\ol f\in \overline X^{\At}$. Among these, we denote by $\ol f_\sharp$ the smallest one with respect to the order determined by the chamber $\chamb$. In this section, we relate the stable envelopes $\Stab^{X}_{\chamb}(f)$ and $\Stab^{\ol X}_{\chamb}(\ol f_\sharp)$ via the correspondence \eqref{eq: NS5 correspondence}.

Consider the morphism between the K\"ahler tori $\psi: Z\to \overline Z$ that is the identity on most components except
\[
    \Cs_{\Zb} \times \Cs_{\hbar}  \to  \Cs_{\Zb'} \times  \Cs_{\Zb''} \times \Cs_{\hbar} \qquad (z,\hbar)\mapsto \begin{cases}
        (z, z\hbar^{-\w'}, \hbar ) & \text{if $\D$ is separated}\\
        (z\hbar^{\w''}, z, \hbar ) & \text{if $\D$ is co-separated.}
    \end{cases}
\]
Following our convention, we denote the induced embedding $\psi: \Ell_{Z}\to \Ell_{\overline Z}$ between equivariant elliptic cohomology of a point in the same way. The next theorem relates the stable envelopes of $X$ and $\ol X$.

\begin{theorem}[{{\cite[Theorem 7.9]{BR}}}]
\label{main theorem proof NS5 resolution for stabs}

Let $f\in X^{\At}$ and let $\ol f_\sharp\in p^{-1}(f)^{\At}$ be the minimal fixed point with respect to the attracting order determined by $\chamb$. Then the following formula holds
\[
p_{\oast} \left(j^{\oast}\psi^{\oast} \Stab^{\ol X}_{\chamb}(\ol f_\sharp)\right)=\Stab_{\chamb}(f).
\]
\end{theorem}
 Here, the subscripts and superscripts refer to pushforward and pullback in elliptic cohomology as in \S\ref{ell functoriality}. See also \cite[\S4]{BR} for details.

\begin{remark}
   Note that the image of $\psi$ is the divisor $\{z'/z''\hbar^{ \w'}=1\}$ (or $\{z'/z''\hbar^{ -\w''}=1\}$ in the co-separated case). On the other hand, the stable envelopes $\Stab_{\chamb}(\ol f)$ are meromorphic sections with poles on $\{z_i/z_j\hbar^{\alpha}=1\}$ for certain $\alpha$. The fact that the class $\Stab^{\ol X}_{\chamb}(\ol f_\sharp)$ does not pick a pole upon pullback by $\psi^{\oast}$ is part of the theorem.
For details, see \cite[Lemma 7.11]{BR}. 
\end{remark}

\begin{remark}
     If $\mathfrak{C}$ is the standard chamber $\mathfrak{C}=\lbrace a_1<a_2<\dots<a_n\rbrace$, then the fixed point $\ol f_\sharp$ appearing in the formulas above
    can be described as in Figure \ref{fig:NS5 resolution fixed points}.
\begin{figure}
\centering
\begin{tikzpicture}[scale=.3]
\draw [thick] (0,1) node[left]{$f=$} --(7,1);  
\draw [thick, red] (3,0) node [left]{$z$} -- (4,2);
\draw [dashed](4,2) to [out=60,in=180] (7,4.6) ;
\draw [dashed](4,2) to [out=60,in=180] (7,4.8) ;
\draw [dashed](4,2) to [out=60,in=180] (7,5) node [right]{$ \Big\rbrace \w'+\w'' $};
\draw [dashed](4,2) to [out=60,in=180] (7,5.2);
\draw [dashed](4,2) to [out=60,in=180] (7,5.4);
\draw[ultra thick, <->] (9.5,1)--(12.5,1) node[above]{separated} node[below]{NS5 res.} -- (15.5,1);
\begin{scope}[xshift=22cm]
\draw [thick] (0,1) node[left]{$\ol f_\sharp=$} --(8,1);  
\draw [thick, red] (2,0) node [left]{$z'$} -- (3,2);
\draw[thick, red] (5,0) node [left]{$z''$} --(6,2);
\draw [dashed](3,2) to [out=60,in=180] (8,3.4) ;
\draw [dashed](3,2) to [out=60,in=180] (8,3.6) node [right] {$\rbrace \w'$} ;
\draw [dashed](3,2) to [out=60,in=180] (8,3.8) ;
\draw [dashed](6,2) to [out=60,in=180] (8,5.6) node [right]{$\rbrace \w''$};
\draw [dashed](6,2) to [out=60,in=180] (8,5.8);
\end{scope}

\draw [thick] (0,-4) node[left]{$f=$} --(7,-4);
\draw [thick, red] (3,-5) node [left]{$z$} -- (4,-3);
\draw [dashed](3,-5) to [out=-120,in=0] (0,-7.6) ;
\draw [dashed](3,-5) to [out=-120,in=0] (0,-7.8) ;
\draw [dashed](3,-5) to [out=-120,in=0] (0,-8)  node [left]{$\w'+\w'' \Big\lbrace$};;
\draw [dashed](3,-5) to [out=-120,in=0] (0,-8.2);
\draw [dashed](3,-5) to [out=-120,in=0] (0,-8.4);
\draw[ultra thick, <->] (9.5,-4)--(12.5,-4) node[above]{co-separated} node[below]{NS5 res.} -- (15.5,-4);

\begin{scope}[xshift=22cm]
\draw [thick] (0,-4) node[left]{$\ol f_\sharp=$} --(8,-4);  
\draw [thick, red] (2,-5) node [right]{$z'$} -- (3,-3);
\draw[thick, red] (5,-5) node [right]{$z''$} --(6,-3);
\draw [dashed](2,-5) to [out=-120,in=0] (0,-7) ;
\draw [dashed](2,-5) to [out=-120,in=0] (0,-7.2) node [left] {$\w' \lbrace$} ;
\draw [dashed](2,-5) to [out=-120,in=0] (0,-7.4) ;
\draw [dashed](5,-5) to [out=-120,in=0] (0,-9.6) node [left]{$ \w'' \lbrace$};
\draw [dashed](5,-5) to [out=-120,in=0] (0,-9.8);
\end{scope}
\end{tikzpicture}
\caption{}
\label{fig:NS5 resolution fixed points}
\end{figure}
Namely, if $X$ is separated (resp. co-separated), then the tie diagram of $\ol f_\sharp$ is obtained from the one of $f$ by making all the ties connected to the resolving branes $\Zb'$ and $\Zb''$ cross (resp. by avoiding all the crosses).

\end{remark}

Localizing at fixed points and using the localization formula for the pushforward, Theorem \ref{main theorem proof NS5 resolution for stabs} produces explicit formulas relating the fixed point restrictions of the stable envelopes of $X$ and $\overline X$. For clarity, we assume that the resolved brane $\Zb=\Zb_k$ is the $k$-th one, counting from left to right. 

\begin{proposition}[{\cite[Prop. 7.13]{BR}}]
\label{proposition separated NS5 fusion ratios stabs}
Assume that $X$ is separated. Then
\begin{multline*}
  \frac{\Stab^{X}_{\chamb}(f)\Big|_g}{\Stab^{X}_{\chamb}(g)\Big|_g}(a, z, \hbar) \\
  = \sum_{\ol g\in Y^{\At}} \frac{\vartheta(N_{\ol g/Y}^-)}{\vartheta(N_{\ol g/Y^{\hbar}})} (a\hbar^{-\gamma(g)}) 
  \frac{\Stab^{\overline X}_{\chamb}(\ol f_{\sharp})\Big|_{\ol g}}{\Stab^{\overline X}_{\chamb}(\ol g)\Big|_{\ol g}}
  (a, z_1,\dots,z'_{k}=z_k, z''_k=z_k\hbar^{-\w'}, \dots, z_m, \hbar).
\end{multline*}
Here $Y=\ttt{\fs$\w'$\fs$\w'+ \w''$\bs\dots \bs 3\bs2\bs1\bs}$ and $N_{\ol g/Y^{\hbar}}$ is the restriction to $\ol g$ of the tangent class 
$TY^{\hbar}\in K_{\Tt}(Y^{\hbar})$. The multi-index $\gamma(g)$ indicates the shift introduced in Remark \ref{remark torus action on fibers of NS5 resolutions}.
\end{proposition}

\begin{proposition}
\label{prop: co-separated NS5 fusion ratios stabs opp}
    Assume that $X$ is co-separated. Then
\begin{multline*}
  \frac{\Stab^{\opp, X}_{\chamb}(f)\Big|_g}{\Stab^{\opp, X}_{\chamb}(g)\Big|_g}(a, z, \hbar) \\
  = \sum_{\ol g\in Y^{\At}} \frac{\vartheta(N_{\ol g/Y}^-)}{\vartheta(N_{\ol g/Y^{\hbar}})} (a\hbar^{\gamma(g)}) 
  \frac{\Stab^{\opp, \ol X}_{\chamb}(\ol f_{\sharp})\Big|_{\ol g}}{\Stab^{\opp, \ol X}_{\chamb}(\ol g)\Big|_{\ol g}}
  (a, z_1,\dots,z'_{k}=z_k, z''_k=z_k\hbar^{\w'}, \dots, z_m, \hbar).
\end{multline*}
Here $Y=\ttt{\bs1\bs2\bs3\bs\dots\bs$\w'+\w''${\fs}$\w''${\fs}}$ and $N_{\ol g/Y^{\hbar}}$ is the restriction to $\ol g$ of the tangent class
$TY^{\hbar}\in K_{\Tt}(Y^{\hbar})$. The multi-index $\gamma(g)$ indicates the shift introduced in Remark \ref{remark torus action on fibers of NS5 resolutions}.
\begin{proof}
    Let $n$ be the number of $D5$ branes of both $X$ and $\ol X$. Then charge and weights are related by $r_i=n-\w_i$ for all $i$. The proof then follows such equality, \cite[Prop. 7.14]{BR}, and the fact that, as noted in Proposition \ref{prop: duality stable envelopes}, we have
    \[
    \Stab^{\opp, X}_{\chamb}(a,z,\hbar)=\Stab^{X}_{\chamb}(a,z^{-1}\hbar^r,\hbar).
    \]
\end{proof}
\end{proposition}



\section{NS5 resolutions and vertex functions}\label{sec: NS5 and vertex}

Analogous to \S\ref{sec: d5 resolutions}, the main result in this section is Theorem \ref{thm: ns5vertex}, which expresses the compatibility of vertex functions with respect to NS5 resolutions. Despite the similarities with \S\ref{sec: d5 resolutions}, we warn the reader that Theorem \ref{thm: ns5vertex} is considerably more subtle. While the compatibility of vertex functions and D5 resolutions is essentially a term-by-term equality of formal power series, the NS5 property involves evaluating a vertex function at a point on the boundary of the radius of convergence. Hence it is really a statement about analytic continuations of meromorphic functions defined by power series. As such, we do not know a geometric proof of this formula.

Our main tool will be certain scalar $q$-difference equations solved by the vertex functions of bow varieties when all D5 branes have weight $1$. Such bow varieties are isomorphic to the cotangent bundle of a type $A$ partial flag variety.

\subsection{Macdonald equations}
\label{subsec: Macdonald equations}

Let $X$ be a separated or co-separated bow variety with NS5 branes $\Zb_{1},\ldots,\Zb_{m}$ and D5 branes $\Ab_{1},\ldots,\Ab_{n}$ such that $\w(\Ab_{i})=1$ for all $i$. If $X$ is separated, we view it as living on the ``starting side" of mirror symmetry. If $X$ is co-separated, we view it as living on the dual, or $!$, side. This simply means that we use the polarization $\hbar(\alpha+\beta)^{\vee}$ from \S\ref{sec: mirsym vertex normalization} in the separated case and $\alpha+\beta$ in the co-separated case. Nevertheless, to avoid cluttering notation, we will omit the $!$ over all parameters for coseparated $X$. 

Let $\tb_{i}$ be the tautological bundle between the NS5 branes $\Zb_{i}$ and $\Zb_{i+1}$ for $1 \leq i \leq m-1$. These are the tautological bundles which are not necessarily topologically trivial. Let $\lb_{i}=\det \tb_{i}$.

Let
\[
\macdop_{n}=\sum_{i=1}^{n} \left(\prod_{j \neq i} \frac{\hbar q a_{i}-a_{j}}{a_{i}-a_{j}}\right) T_{a_i},
\]
where $T_{a_{i}} a_{j}=q^{\delta_{i,j}} a_{j}$. This is the first Macdonald difference operator in the variables $a_{i}$, see \cite[\S6]{mac}.

For $f \in X^{\At}$, we set
\[
\tilde{\ver}^{X}_{f}:=\exppref_{f} \Phi_{f} V^{X}_{f}.
\]
where $\delta_f$ is the function
\begin{equation}\label{exppref}
\exppref_{f}:=
    \prod_{i=1}^{m-1} \frac{\vartheta((-\hbar^{-1/2})^{\w(\Zb_{i})+\w(\Zb_{i+1})} Q_{i}^{-1} \lb_{i}|_{f})}{\vartheta((-\hbar^{-1/2})^{\w(\Zb_{i})+\w(\Zb_{i+1})} Q_{i}^{-1}) \vartheta(\lb_{i}|_{f})},
\end{equation}
cf. \eqref{eq: delta function}, and $\Phi_{f}:=\Phi\left((q-\hbar^{-1}) \mathcal{P}|_{f}\right)$ is the infinite $q$-Pochhammer symbol associated to the class

\[
\mathcal{P}:=\bigoplus_{\Zb} \Hom(\tb_{\Zb^{-}},\tb_{\Zb^{+}}) \ominus \bigoplus_{\Zb} \Hom(\tb_{\Zb^{+}},\tb_{\Zb^{+}}).
\]

\begin{theorem}[\cite{KorZeit}]\label{thm: qdeflag}
The vertex functions satisfy the $q$-difference equation
 \[
\macdop_{n} \tilde{\ver}^{X}_{f}= e^{X}(Q)\tilde{\ver}^{X}_{f},
\]
   where
    \begin{multline*}
e^{X}(Q)= \\
\begin{cases}
    \sum\limits_{i=1}^{m} \left(\sum\limits_{j=0}^{\w(\Zb_{i})-1} (\hbar q)^{j} \right) (\hbar q)^{\sum\limits_{k=1}^{i-1} \w(\Zb_{k})} \left( \prod\limits_{l=i}^{m-1} \left(-\hbar^{1/2}\right)^{\w(\Zb_{l})+\w(\Zb_{l+1})} Q_{l} \right) & \text{if $X$ is separated} \\[5 mm]
    \sum\limits_{i=1}^{m} \left(\sum\limits_{j=0}^{\w(\Zb_{i})-1} (\hbar q)^{j} \right) (\hbar q)^{\sum\limits_{k=i+1}^{m} \w(\Zb_{k})} \left( \prod\limits_{l=1}^{i-1} \left(-\hbar^{1/2}\right)^{\w(\Zb_{l})+\w(\Zb_{l+1})} Q_{l} \right) & \text{if 
    $X$ is co-separated.}
\end{cases}
    \end{multline*}
\end{theorem}

Notice that $e^{X}(Q)$ is the first elementary symmetric function in a certain set of variables. Koroteev and Zeitlin \cite{KorZeit} also prove that the vertex functions satisfy higher Macdonald difference equations, though we will not need this fact. In the case of complete flag varieties, a geometric proof of Theorem \ref{thm: qdeflag} using nonabelian shift operators and wall-crossing was given by Tamagni in \cite{tamagninonabelian}.

\subsection{Transformation properties}

We retain the assumptions on $X$ from \S\ref{subsec: Macdonald equations}. Recall from \S\ref{sub: vholo} that a fixed point $f \in X^{\At}$ is encoded by a function $\sigma_{f}:\{1,\ldots,n\}\to \{1,\ldots,m\}$. The $\Tt$-character of the restriction of the tautological bundle $\tb_{i}$ to the point $f$ is given by
\[
\tb_{i}|_{f}=\begin{cases} \hbar^{i-m}\sum\limits_{\substack{j \\ \sigma_{f}(j) \leq i}} a_{j} & \text{if $X$ is separated} \\[5 mm]
\sum\limits_{\substack{j \\ \sigma_{f}(j) >i}} \hbar a_{j}  & \text{if $X$ is co-separated,}
\end{cases}
\]
where the index $j$ in the sums sums above and below run over $\{1,2,\ldots,n\}$, the set of D5 branes. So the $\Tt$-character of $\lb_{i}$ is
\begin{equation}\label{lbchar}
\lb_{i}|_{f}=\begin{cases}
    \hbar^{(i-m) \rank(\tb_{i})}\prod\limits_{\substack{j \\ \sigma_{f}(j)\leq i }} a_{j} & \text{if $X$ is separated} \\[5 mm]
   \hbar^{\rank(\tb_{i})}\prod\limits_{\substack{j \\ \sigma_{f}(j)> i }} a_{j} & \text{if $X$ is co-separated.}
\end{cases}
\end{equation}

We easily compute the following.
\begin{lemma}
The following holds:
\[
\frac{T_{a_{i}}(\exppref_{f})}{ \exppref_{f}}= \begin{cases}
    \prod\limits_{l=\sigma_{f}(i)}^{m-1} (-\hbar^{1/2})^{\w(\Zb_{l}) +\w(\Zb_{l+1})} Q_{l} & \text{if $X$ is separated} \\[5 mm]
     \prod\limits_{l=1}^{\sigma_{f}(i)-1} (-\hbar^{1/2})^{\w(\Zb_{l}) +\w(\Zb_{l+1})} Q_{l}  & \text{if $X$ is co-separated.}
\end{cases}
\]

\end{lemma}

\begin{proof}
    This follows directly from the definition of $\exppref_{f}$, \eqref{lbchar}, and \eqref{eq: quasi-period theta function}.
\end{proof}

A straightforward computation shows that
\[
\mathcal{P}|_{f}=
\begin{cases}
  \sum\limits_{i=1}^{m-1} \sum\limits_{\substack{j,k \\ \sigma_{f}(j)\leq i \\ \sigma_{f}(k) \leq i}} \left(\hbar \frac{a_{k}}{a_{j}}- \frac{a_{k}}{a_{j}}\right) + \sum\limits_{1 \leq i < j \leq m} \sum\limits_{\substack{k,l \\ \sigma_{f}(k)=i \\ \sigma_{f}(l)=j}} \hbar \frac{a_{l}}{a_{k}}  & \text{if $X$ is separated} \\[5mm]
\sum\limits_{\substack{j,k\\ \sigma_{f}(j)< \sigma_{f}(k)}} \frac{a_{k}}{a_{j}} & \text{if $X$ is co-separated,}
\end{cases}
\]
and
\begin{equation}\label{eq: tan flag}
TX|_{f} = \begin{cases}
   \sum\limits_{\substack{k,l \\ \sigma_{f}(k)< \sigma_{f}(l)}} \left(\hbar \frac{a_l}{a_k} +\frac{a_k}{a_l}\right)   & \text{if $X$ is separated} \\[5 mm]
    \sum\limits_{\substack{k,l \\ \sigma_{f}(k)< \sigma_{f}(l)}} \left(\frac{a_l}{a_k} +\hbar  \frac{a_k}{a_l}\right) & \text{if $X$ is co-separated.}
\end{cases}
\end{equation}

Recall the chambers $\chamb_{f}$ from \S\ref{sub: vholo}. Let 
\begin{equation}\label{eq: phichamb}
\Phi_{\chamb_{f}}=\Phi((q-\hbar^{-1})N_{f}^{-,\chamb_{f}}),
\end{equation}
where $N_{f}^{-,\chamb_{f}}$ denotes the repelling part of the normal bundle at $f$ with respect to $\chamb_{f}$. The following lemma allows us to replace $\Phi_{f}$ in the $q$-difference equations by $\Phi_{\chamb_{f}}$. The latter has the advantage of being holomorphic in a neighborhood of $0_{\chamb_{f}}$ from \S\ref{sub: vholo}.

\begin{lemma}\label{lem: phitransform}
The following holds:
    \[
\frac{T_{a_{i}}(\Phi_{f})}{\Phi_{f}}= \frac{T_{a_{i}}(\Phi_{\chamb_{f}})}{\Phi_{\chamb_{f}}}.
    \]
\end{lemma}
\begin{proof}
First, assume that $X$ is separated. Then
    \[
\Phi_{f}=\left(\prod_{i=1}^{m-1} \prod_{\substack{j,k \\ \sigma_{f}(j)\leq i \\ \sigma_{f}(k) \leq i}}  \frac{\Phi\left(q \frac{a_{j}}{a_{k}}\right) \Phi\left(\frac{a_{j}}{a_{k}}\right) }{\Phi\left(\hbar^{-1} \frac{a_{j}}{a_{k}}\right) \Phi\left(q \hbar \frac{a_{j}}{a_{k}}\right)}\right) \left(  \prod_{1 \leq i < j \leq m} \prod_{\substack{k,l \\ \sigma_{f}(k)=i \\ \sigma_{f}(l)=j}} \frac{\Phi\left(q \frac{a_{k}}{a_{l}}\right)}{\Phi\left(\hbar^{-1} \frac{a_{k}}{a_{l}}\right)} \right).
    \]
We claim that the leftmost product does not contribute to the transformation property. To see this, first observe that the terms corresponding to $j=k$ depend only on $\hbar$ and $q$ and can thus be neglected. The remaining terms appear in pairs which take the form
    \[
X_{j,k}:=\frac{\Phi\left(q \frac{a_{j}}{a_{k}}\right) \Phi\left(\frac{a_{j}}{a_{k}}\right) }{\Phi\left(\hbar^{-1} \frac{a_{j}}{a_{k}}\right) \Phi\left(q \hbar \frac{a_{j}}{a_{k}}\right)} \frac{\Phi\left(q \frac{a_{k}}{a_{j}}\right) \Phi\left(\frac{a_{k}}{a_{j}}\right) }{\Phi\left(\hbar^{-1} \frac{a_{k}}{a_{j}}\right) \Phi\left(q \hbar \frac{a_{k}}{a_{j}}\right)}.
    \]
    Using the transformation property of $\Phi$, we compute
    \[
\frac{T_{a_{j}}(X_{j,k})}{X_{j,k}}=\frac{\left(1-\hbar^{-1} \frac{a_{j}}{a_{k}}\right) \left(1-q \hbar \frac{a_{j}}{a_{k}}\right)}{\left(1-q \frac{a_{j}}{a_{k}}\right) \left(1- \frac{a_{j}}{a_{k}}\right)} \frac{\left(1-\frac{a_{k}}{a_{j}}\right) \left(1- \frac{a_{k}}{q a_{j}}\right)}{\left(1-\frac{a_{k}}{q \hbar a_{j}}\right)\left(1-\hbar \frac{a_{k}}{a_{j}}\right)} =1.
    \]
Accounting for the rest of the terms, we see that
\[
\frac{T_{a_i}(\Phi_{f})}{\Phi_{f}}=\left(\prod_{\substack{j \\ \sigma_{f}(j)<\sigma_{f}(i)}} \frac{1-\frac{a_{j}}{a_{i}}}{1-\frac{a_{j}}{q \hbar a_{i}}} \right) \left(\prod_{\substack{j \\ \sigma_{f}(j)>\sigma_{f}(i)}} \frac{1-\hbar^{-1}\frac{a_{i}}{a_{j}}}{1-q \frac{a_{i}}{a_{j}}}  \right).
\]
The fact that this equals the transformation of $\Phi_{\chamb_{f}}$ under $T_{a_{i}}$ follows from \eqref{eq: tan flag} and the definition of $\chamb_{f}$.

Now assume that $X$ is co-separated. Then 
\[
\Phi_{f}=\prod_{\substack{j,k \\ \sigma_{f}(j)<\sigma_{f}(k)}}\frac{\Phi\left(q \frac{a_k}{a_j}\right)}{\Phi\left(\hbar^{-1} \frac{a_k}{a_j}\right)}.
\]
So
\[
\frac{T_{a_i}(\Phi_{f})}{\Phi_{f}}=\left(\prod_{\substack{j \\ \sigma_{f}(j)<\sigma_{f}(i)}} \frac{1-\hbar^{-1}\frac{a_i}{a_j}}{1-q\frac{a_i}{a_j}}\right) \left( \prod_{\substack{j \\ \sigma_{f}(j)>\sigma_{f}(i)}} \frac{1-\frac{a_k}{a_i}}{1-\frac{a_k}{q\hbar a_i}} \right),
\]
and the result again follows from \eqref{eq: tan flag} and the definition of $\chamb_{f}$.
\end{proof}



    

\subsection{NS5 property of difference equations}
Now assume that $\w(\Zb_{k})>1$ for some $k$. In addition to $X$ above, we also now consider an NS5 resolution $\ol X$ of $X$ obtained by resolving the brane $\Zb_{k}$. In particular, the bow diagram of $\ol X$ has D5 branes $\Ab_{1},\ldots,\Ab_{n}$ and NS5 branes $ \Zb_{1}, \ldots, \Zb_{k}',\Zb_{k}'', \ldots , \Zb_{m}$ such that $\w(\Zb_{k})=\w( \Zb_{k}')+\w( \Zb_{k}'')$. We write the K\"ahler parameters of $X$ as $Q_{1},\ldots,Q_{m-1}$ and the K\"ahler parameters of $\overline{X}$ as $\ol Q_{1},\ldots, \ol Q_{k-1}, \ol Q_{\text{new}}, \ol Q_{k},\ldots, \ol Q_{m-1}$. 

The $q$-difference equations for $X$ and $\ol X$ are 
\[
\macdop_{n} \tilde{\ver}^{X}_{f}= e^{X}(Q) \tilde{\ver}^{X}_{f}, \qquad \macdop_{n} \tilde{\ver}^{\ol X}_{g}= e^{\ol X}(\ol Q) \tilde{\ver}^{\ol X}_{g}
\]
for $f \in X^{\At}$ and $g \in \ol X^{\At}$. In particular, the $q$-difference operators are the same.

Let $\psi^{*}$ be the specialization
\[
(z_k',  z_k'')\mapsto \left((-\hbar^{1/2})^{-\w(\Zb_{k}'')} z_k, (-\hbar^{1/2})^{\w(\Zb_{k}')} z_k\right).
\]
and the identity on the rest of the K\"ahler parameters. In terms of $\ol Q$, we have
 \begin{align*}
     \psi^{*}: \ol Q_{k-1} &\mapsto  (-\hbar^{1/2})^{\w(\Zb_{k}'')} Q_{k-1} \\
   \ol Q_{\text{new}} &\mapsto (-\hbar^{1/2})^{-\w(\Zb_{k})} \\
   \ol Q_{k} &\mapsto (-\hbar^{1/2})^{\w(\Zb_{k}')} Q_{k},
 \end{align*}
 and $\psi^{*}(\ol Q_{l})=Q_{l}$ otherwise.
   

There is a simple relationship between the two eigenvalues after applying $\psi^{*}$. 

\begin{lemma}\label{lem: eigenvalues}
    Under the specialization $\psi^{*}$ we have $\psi^{*} e^{\ol X}(\ol Q)= e^{X}(Q)$.
\end{lemma}
\begin{proof}
    This is a straightforward computation using the formulas for $e^{X}(Q)$ given in Theorem \ref{thm: qdeflag}.
\end{proof}

Let $f \in X^{\At}$ and let $\ol f \in \ol X^{\At}$ be an NS5 resolution of $f$. We consider the gauge transformed $q$-difference operators defined by
\[
\macdop_{f}:=\macdop_{f}(Q):=\exppref_{f}^{-1} \macdop_{n} \exppref_{f}, \qquad \macdop_{\ol f}:=\macdop_{\ol f}(Q):=\exppref_{\ol f}^{-1} \macdop_{n} \exppref_{\ol f}.
\]
It is implied by Theorem \ref{thm: qdeflag} and Lemma \ref{lem: phitransform} that 
\[
\macdop_{f} \Phi_{\chamb_{f}} \ver^{X}_{f}(Q)=e^{X}(Q) \Phi_{\chamb_{f}} \ver^{X}_{f}(Q)
\]
and 
\[
\macdop_{\ol f} \Phi_{\chamb_{\ol f}} \ver^{\ol X}_{\ol f}(\ol Q)=e^{\ol X}(\ol Q) \Phi_{\chamb_{\ol f}} \ver^{\ol X}_{\ol f}(\ol Q).
\]

\begin{proposition}
Under the specialization $\psi^{*}$, we have
\[
\psi^{*} \macdop_{\ol f} = \macdop_{f}.
\]
\end{proposition}
\begin{proof}
    Assume $X$ is separated (the co-separated case is similar). Then by the quasi-periodicity \eqref{eq: quasiperiod delta} we get
    \[
\macdop_{f}=\sum_{i=1}^{n} \left(\prod_{j \neq i} \frac{\hbar q a_{i}-a_{j}}{a_{i}-a_{j}}  \right) \left(\prod_{l=\sigma_{f}(i)}^{m-1} (-\hbar^{1/2})^{\w^{X}(\Zb_{l})+\w^{X}(\Zb_{l+1})} Q_{l}\right)T_{a_{i}}.
\]
A similar formula holds for $\macdop_{\ol f}$. It is straightforward to check that the two expressions are equal after applying $\psi^{*}$.
\end{proof}

Applying a shift, we get a similar result for the shifted vertex $\MSver$ defined in \eqref{eq: def of MSver}. Let $\mathsf{e}^{X}(Q)=e^{X}(Q)|_{Q_{i}=(-\hbar^{1/2})^{\w(\Zb_{i})-\w(\Zb_{i+1})}Q_{i}}$ and let $\MSmacdop_{f}(Q)=\macdop_{f}(Q)|_{Q_{i}=(-\hbar^{1/2})^{\w(\Zb_{i})-\w(\Zb_{i+1})}Q_{i}}$. For later use, we split up the separated and coseparated cases.

\begin{theorem}\label{thm: sepcompareqdes}
   Assume that $X$ is separated. Let $\psi^{*}$ be the specialization 
   \[
   (z_{k}', z_{k}'')\mapsto(z_k, \hbar^{-\w(\Zb_{k}')} z_k).
   \]
 Then 
   \[
   \MSmacdop_{f}(Q^{-1}) \Phi_{\chamb_{f}} \MSver^{X}_{f}(Q^{-1})=\mathsf{e}^{X}(Q^{-1}) \Phi_{\chamb_{f}} \MSver^{X}_{f}(Q^{-1}).
   \]
   and similarly for $\ol f$. Furthermore, we have
   \[
\psi^{*} \MSmacdop_{\ol f}(\ol Q^{-1})=  \MSmacdop_{f}(Q^{-1}), \qquad \psi^{*}\mathsf{e}^{\ol X}(\ol Q^{-1})=\mathsf{e}^{X}(Q^{-1}).
   \]

\end{theorem}

\begin{theorem}\label{thm: cosepcompareqdes}
    Assume that $X$ is co-separated. Let $\psi^{*}$ be the specialization 
    \[
    (z_{k}', z_{k}'')\mapsto(z_k, \hbar^{\w(\Zb_{k}')} z_k ).
    \]
  Then 
   \[
   \MSmacdop_{f}(Q) \Phi_{\chamb_{f}} \MSver^{X}_{f}(Q)=\mathsf{e}^{X}(Q) \Phi_{\chamb_{f}} \MSver^{X}_{f}(Q).
   \]
and similarly for $\ol f$. Furthermore, we have
\[
\psi^{*} \MSmacdop_{\ol f}(\ol Q)=  \MSmacdop_{f}(Q), \qquad \psi^{*}\mathsf{e}^{\ol X}(\ol Q)=\mathsf{e}^{X}(Q).
\]
\end{theorem}

\subsection{Uniqueness of solutions}

 Let $f \in X^{\At}$ and let $i_{j}$ for $1 \leq j \leq n$ be defined by the properties that $\{i_1,i_2,\ldots,i_n\}=\{1,2,\ldots,n\}$ and $a_{i_j}/a_{i_{j+1}}$ is a repelling weight for $\chamb_{f}$. Expanding $\Phi_{\chamb_{f}} \ver^{X}_{f}(Q)$ as a power series in $a_{i_{j}}/a_{i_{j+1}}$ gives an element of the ring 
 \[
 R:=\mathbb{Q}[[Q_{1},\ldots,Q_{m-1}]][[a_{i_1}/a_{i_2},\ldots,a_{i_{n-1}}/a_{i_n}]].
 \]

We need the following uniqueness result.

\begin{proposition}[\cite{NSmac} Theorem 2.1]\label{prop: uniqueness}
  A solution of the equation 
    \[
\macdop_{f}F=e^{X}(Q) F, \quad F \in R
    \]
is uniquely determined by its constant term in $a$.
\end{proposition}

By a change of variable, this implies an analogous uniqueness result for the $q$-difference equations defined by the operators $\MSmacdop_{f}$.

\subsection{NS5 property for vertex functions}

Now we are ready to relate vertex functions before and after an NS5 resolution. Since the co-separated case proceeds similarly, we will give detailed proofs only for the separated case and will simply state the corresponding co-separated results at the end of the section. 

Let $X$ be a separated bow variety and let $\ol X$ be an NS5 resolution obtained by resolving the brane $\Zb$ of weight $\w>1$ into $\Zb'$ and $\Zb''$ according to the decomposition of weights $\w=\w'+\w''$. Let $f \in X^{\At}$ and let $\ol f \in \ol X^{\At}$ be an NS5 resolution of $f$. Let $\psi^{*}$ be the specialization 
    \begin{align*}
        (z', z'')\mapsto  (z,
        \hbar^{-\w'} z).
    \end{align*}
    Recall the variety $Y$ from Lemma \ref{lemma points in lagrangian variety Z resolution}. Denote $Q'=z'/z''$, which we view as the single K\"ahler variable in the vertex function of $Y$. We compare the limits defined in \S\ref{sec: flag vertex results}.

\begin{lemma}\label{lem: different const NS5}
   In the setting above, assume furthermore that all D5 branes have weight $1$. Then
    \[
\psi^{*} \left( \frac{\MSflaglim^{\ol X}(\ol Q)}{\MSflaglim^{X}(Q)} \MSflaglim^{Y}(Q')^{-1} \right) = 1
    \]
\end{lemma}
\begin{proof}
    After specialization by $\psi^{*}$, it is easy to see from Proposition \ref{prop: msflaglim} that the NS5 branes other than $\Zb$, $\Zb'$, and $\Zb''$ contribute to the formulas for $\MSflaglim^{\ol X}$ and $\MSflaglim^{X}$ in exactly the same ways. It is a straightforward computation to show that contributions from these three branes gives exactly
    \[
    \prod_{j=1}^{\w(\Zb')} \frac{\Phi\left( \hbar^{-1} (q\hbar)^{j-1}\right)}{\Phi\left( (q\hbar)^{j-1}\right)}
    \]
    On the other hand,
    \[
\MSflaglim^{Y}(Q') =\prod_{j=1}^{\w(\Zb')} \frac{\Phi\left(\hbar^{-1} (q\hbar)^{j-1} q^{-\w(\Zb')} Q'^{-1}\right)}{\Phi\left( (q\hbar)^{j-1} q^{-\w(\Zb')}  Q'^{-1} \right)}.
\]
Applying $\psi^{*}$ finishes the proof.
\end{proof}

\begin{lemma}\label{compatiblechambers}
 In the setting above, assume furthermore that all D5 branes have weight $1$. Then ordering $<_{\ol f}$ refines the ordering $<_{f}$ from \S\ref{sub: vholo}. So if a function is holomorphic in a neighborhood of $0_{\chamb_{f}}$, then it is also holomorphic in a neighborhood of $0_{\chamb_{\ol f}}$.
\end{lemma} 

Recall $\Phi_{\chamb_{f}}$ from \eqref{eq: phichamb}. We can also view $\ol f$ as a point in $Y$ equipped with the chamber induced by $\chamb_{\ol f}$. So we define $\Phi_{\chamb_{\ol f},Y}:=\Phi\left((q-\hbar^{-1}) N^{-}_{\ol f / Y}\right)$ where the repelling directions are taken for the chamber $\chamb_{\ol f}$.

\begin{theorem}\label{thm: ns5vertex}
    Let $X$ be a separated bow variety, and let $\ol X$ be an NS5 resolution. Let $f \in X^{\At}$ and let $\ol f \in \ol X$ be an NS5 resolution of $f$. Then 
\begin{equation}\label{eq: sepNS5res}
\MSver^{X}_{f}(Q^{-1})=\Phi_{\chamb_{\ol f},Y}\psi^{*}\left( \MSflaglim^{Y}(Q')^{-1} \MSver^{\ol X}_{\ol f}(\ol Q^{-1})\right).
\end{equation}
\end{theorem}
The well-definedness of $\psi^{*}$, a nontrivial fact, will be justified in the course of the proof.
\begin{proof}

In the proof below, we will omit writing the argument $(Q)$ or $(Q)^{-1}$. For a bow variety $X$ with NS5 branes $\Zb_{i}$, let $k(X):=\sum_{i} (\w(\Zb_{i})-1)$. Note that $k(X)=k(\ol X)+1$ for any NS5 resolution $\ol X$ of $X$.

\textbf{Step 1:} First, we will prove the statement assuming that all D5 branes have charge 1. Assume this for $X$ (hence also for $\ol X$). Actually, in this setting, we will prove a stronger statement: we will prove that for $X$ and $\ol X$, both \eqref{eq: sepNS5res} holds \emph{and} the poles of $\MSver^{X}_{f}$ in $Q$ coincide with the poles of $\MSflaglim^{X}$. We will prove this by induction on $k(\ol X)$.

By Corollary \ref{cor: different tangents NS5} and Corollary \ref{lem: different const NS5}, the desired formula is equivalent to 
\begin{equation}\label{eq: ns5vertex}
\Phi_{\chamb_{f}} \MSflaglim^{X}(Q)^{-1} \MSver^{X}_{f}(Q^{-1}) =\Phi_{\chamb_{\ol f}}  \psi^{*} \left(   \MSflaglim^{\ol X}(\ol Q)^{-1} \MSver^{\ol X}_{\ol f}(\ol Q^{-1})\right)
\end{equation}
where $\chamb_{f}$ and $\chamb_{\ol f}$ are as in \S\ref{sub: vholo}.

As the base case, suppose that $k(\ol X)=0$. In this case, all NS5 branes and D5 branes have weight 1 and $\ol X$ is isomorphic to the cotangent bundle of a complete flag variety. By Proposition \ref{prop:fullflagpoles}, the poles of $\MSver^{\ol X}_{\ol f}$ coincide with those of $\MSflaglim^{\ol X}$. Hence, the evaluation
\[
\psi^{*}\left((\MSflaglim^{\ol X})^{-1} \MSver^{\ol X}_{\ol f} \right)
\]
is well-defined, i.e. the simple pole in the second term is canceled by the simple zero in the first.

By Proposition \ref{Vholomorphic} and Lemma \ref{compatiblechambers}, the functions $\MSver^{X}_{f}$, $\MSver^{\ol X}_{\ol f}$, $\Phi_{\chamb_{\ol f}}$, and  $\Phi_{\chamb_{\ol f}}$ are all holomorphic in a neighborhood of the point $0_{\chamb_{\ol f}}$. By Theorem \ref{thm: sepcompareqdes}, both sides of \eqref{eq: ns5vertex} satisfy the same $q$-difference equation. By Proposition \ref{prop: uniqueness}, it is sufficient to show that the leading terms agree. Specifically, we must show that $\lim_{a\to 0_{\chamb_{\ol f}}}$ of the two sides agree. By definition of $\MSflaglim^{X}$ and $\MSflaglim^{\ol X}$ and since $\lim_{a \to 0_{\chamb_{\ol f}}} \Phi_{\chamb_{\ol f}} = \lim_{a \to 0_{\chamb_{\ol f}}} \Phi_{\chamb_{f}}=1$, this holds.

So \eqref{eq: sepNS5res} holds in this case. Furthermore, from the equality
\[
\MSver^{X}_{f}=\frac{\Phi_{\chamb_{\ol f}}}{\Phi_{\chamb_{f}}} \psi^{*}\left(\frac{\MSflaglim^{X}}{\MSflaglim^{\ol X}} \MSver^{\ol X}_{\ol f} \right)
\]
it is immediate that the poles of $\MSver^{X}_{f}$ in $Q$ coincide with those of $\MSflaglim^{X}$. This proves the base case.

Now suppose the statement to be proven holds for $X$ and $\ol X$ whenever $k(\ol X)< k_{0}$ for some $k_{0} >0$. 

Let $X$ and $\ol X$ be such that $k(\ol X)=k_{0}$. Let $Z$ be an NS5 resolution of $\ol X$. Since $k(Z)=k(X)-1<k_{0}$, the inductive hypothesis implies that the statement holds for the pair $\ol X$ and $Z$. In particular, the poles of $\MSver^{\ol X}_{\ol f}$ coincide with those of $\MSflaglim^{\ol X}$. Hence the substitution $\psi^{*}\left((\MSflaglim^{\ol X})^{-1} \MSver^{\ol X}_{\ol f}\right)$ is well-defined. Now, repeating the exact same argument using Macdonald equations as for the base case shows that \eqref{eq: sepNS5res} holds for $X$ and $\ol X$. It is also immediate that the poles of $\MSver^{X}_{f}$ in $Q$ are as desired.

\textbf{Step 2:} Now we apply the D5 property for vertex functions to prove the theorem for arbitrary bow varieties. We proceed by induction on $l(X):=\sum_{i} (\w(\Ab_{i})-1)$. In Step 1, we proved the theorem when $l(X)=0$, which serve as the base cases.

Let $X$ be arbitrary. Let $Z:=\wt{X}$ be a D5 resolution of $X$. The NS5 branes of $X$ and $Z$ are canonically identified. So let $\ol X$ and $\ol Z$ be NS5 resolutions of $X$ and $Z$ obtained by resolving the same NS5 brane in the same way. It is automatic that $\ol Z$ is a D5 resolution of $\ol X$. Let $f \in X^{\At}$. Let $g \in \ol X$ be an NS5 resolution of $f$. We also have the distinguished D5 resolutions $h:=\tilde{f}_{\sharp}\in Z$ and $k:= \tilde{g}_{\sharp} \in \ol Z$.

The NS5 variable specialization $\psi^{*}$ is the same for the pairs $X, \ol X$ and $Z, \ol Z$. The D5 specialization $\varphi^{*}$ is the same for the pairs $X, Z$ and $\ol X, \ol Z$. We have $\varphi^{*}\MSver^{Z}_{h}=\MSver^{X}_{f}$ and $\varphi^{*}\MSver^{\ol Z}_{k}=\MSver^{\ol X}_{g}$. We also have $\MSver^{Z}_{h}=\Phi_{\chamb_{k},Y} \psi^{*}\left((\MSflaglim^{Y})^{-1} \MSver^{\ol Z}_{k} \right)$ by induction, since $l(Z)=l(X)-1$.

Note also that the space $Y$ relevant for the pair $X, \ol X$ is the same as the space $Y$ for the pair $Z, \ol Z$. Thus $\Phi_{\chamb_{g},Y}=\Phi_{\chamb_k,Y}$.

So
\begin{align*}
\psi^{*}\left( \Phi_{\chamb_{g},Y} (\MSflaglim^{Y})^{-1} \MSver^{\ol X}_{g} \right) &=\psi^{*}\left(\Phi_{\chamb_{k},Y} (\MSflaglim^{Y})^{-1} \varphi^{*} \MSver^{\ol Z}_{k} \right)\\
&=\varphi^{*} \psi^{*}\left( \Phi_{\chamb_{k},Y} (\MSflaglim^{Y})^{-1} \MSver^{\ol Z}_{k}\right)\\
&= \varphi^{*} \MSver^{Z}_{h}\\
&=\MSver^{X}_{f}.
\end{align*}
This concludes the proof.
\end{proof}

When $X$ is co-separated, identical arguments, replacing Theorem \ref{thm: sepcompareqdes} with Theorem \ref{thm: cosepcompareqdes}, give the following.

\begin{theorem}\label{thm: cosepNS5vertex}
     Let $X$ be a co-separated bow variety, and let $\ol X$ be an NS5 resolution. Let $f \in X^{\At}$ and let $\ol f \in \ol X$ be an NS5 resolution of $f$. Let $\psi^{*}$ be the specialization
     \begin{align*}
           (z_{k}', z_{k}'')\mapsto  (z_k, \hbar^{\w(\Zb_{k}')} z_k ).
   \end{align*}
   Then 
\begin{equation}\label{eq: cosepNS5vertex}
\MSver^{X}_{f}(Q)=\Phi_{\chamb_{\ol f},Y}\psi^{*}\left( \MSflaglim^{Y}(Q'^{-1})^{-1} \MSver^{\ol X}_{\ol f}(\ol Q)\right).
\end{equation}
\end{theorem}

\begin{corollary}
     Let $X$ be a bow variety and let $f\in X^{\At}$. The poles of $\ver^{X}_{f}(Q)$ in $Q$ are all simple.
\end{corollary}
\begin{proof}
    The proof of Theorem \ref{thm: ns5vertex} showed this result when all D5 branes have weight 1. The general case now follows by iterating Theorem \ref{thm: d5 vertex}.
\end{proof}

\section{Mirror symmetry}\label{sec: mirsym}

In this section, we will prove the main theorem. Let $X$ be a bow variety with 3d mirror dual $X^{!}$. Fix the standard chambers $\chamb$ and $\chamb^{!}$. For a fixed point $f \in X^{\At}$, let $\Phi^{\pm}_{f}=\Phi\left((q-\hbar^{-1})N^{\pm}_{f}\right)$, where the attracting/repelling parts are taken with respect to $\chamb$. Similarly, for $f^{!} \in (X^{!})^{\At^{!}}$, let $\Phi^{\pm}_{f^{!}}=\Phi\left((q-(\hbar^{!})^{-1})N^{\pm}_{{f}^{!}}\right)$.

We normalize vertex functions $\MSver^{X}(Q)$ and $\MSver^{X^{!}}(Q^{!})$ as in \eqref{eq: def of MSver}. In particular, we use the class $\hbar(\alpha+\beta)^{\vee}$ (resp. $\alpha^{!}+\beta^{!}$) for the polarization in the vertex function of $X$ (resp. $X^{!}$). We normalize stable envelopes as in \eqref{eq: def of msstab}.

Define the mirror map $\mirmap:=\mirmap_{X^{!} \to X}$ by
\begin{align*}
   q^{!} \mapsto q\qquad 
   \hbar^{!} \mapsto \frac{1}{\hbar q}
   \qquad 
   (q^{!})^{-\w(\Zb_{i}^{!})} Q_{i}^{!}\mapsto u_{i}
   \qquad 
   u_{i}^{!} \mapsto  q^{\w(\Zb_{i})} Q_{i}.
\end{align*}

\begin{theorem}\label{thm: mirsym}
Mirror symmetry of vertex functions holds:
\[
\mirmap \left(\Phi^{+}_{f^{!}} \MSver^{X^{!}}_{f^!}(Q^{!})\right)=\sum_{g\in X^{\At}}\MSstab^X_{gf} \Phi^{-}_{g} \MSver^{X}_{g}(Q^{-1}).
\]
\end{theorem}

First, we observe that it is enough to prove Theorem \ref{thm: mirsym} assuming that $X$ is separated and $X^{!}$ is coseparated. Indeed, by Proposition \ref{prop: HW vertex}, vertex functions change slightly under Hanany-Witten transition. On the other hand, so does the map $\mirmap$. These two shifts exactly compensate for each other, so the statement of Theorem \ref{thm: mirsym} is unchanged under Hanany-Witten transition.

The mirror symmetry theorem above can be equivalently stated using the stable envelope of $X^!$. To make this precise, for any bow variety $X$, we set
\begin{equation}\label{eq: def dual of msstab}
    \MSstab^{\opp}_{gf}(a, Q,\hbar):=\left(\frac{\det (\alpha_{\opp})|_{f}}{\det(N_{\chamb_{\opp, f}}^-)}\right)^{1/2}\frac{\Stab_{\chamb_{\opp}}^{\opp}(f)|_{g}}{\Stab_{\chamb_{\opp}}^{\opp}(g)|_{g}}(a, Q,\hbar)\left(\frac{\det (\alpha_{\opp})|_{g}}{\det(N_{\chamb_{\opp},g}^-)}\right)^{-1/2}.
\end{equation}
Then, we have:
\begin{theorem}\label{thm: opmirsym}
Theorem \ref{thm: mirsym} is equivalent to
\[
\kappa^{-1} \left(\Phi^+_{\chamb_{\opp},f} \cdot \MSver_f^X(Q^{-1})\right)=\sum_{g^!\in (X^!)^{\At^!}} \MSstab^{\opp, X^!}_{ g^!f^!}(a^{!},(Q^!)^{-1},\hbar^{!})\Phi^-_{\chamb_{\opp}^!, g^!} \MSver^{X^!}_{g^!}(Q^{!}).
\]
\end{theorem}

\begin{proof}

For a given bow variety $X$ and a polarization $\alpha$, we define the diagonal matrix $L^{X}_{\alpha, \chamb}$ with entries 
\[
(L^{X}_{\alpha, \chamb})_{fg}=\left(\frac{\det(\alpha)|_{f}}{\det(N_{\chamb,f}^{ -})}\right)^{1/2}\delta_{fg}, \quad f,g \in X^{\At}.
\]
Using \eqref{stable envelope restriction}, the statement of Theorem \ref{thm: mirsym} can then be written in matrix notation as 
\[
\mirmap\left( \Phi^{+}_{\chamb^!} \cdot \MSver^{X^{!}}\right)= L^{X}_{\alpha, \chamb} \cdot  S^X_{\chamb} \cdot (L^{X}_{\alpha, \chamb})^{-1} \cdot \Phi^-_{\chamb} \cdot \MSver^X,
\]
where now $\alpha$ is the class \eqref{eq: class alpha for general bow variety}.
We can rewrite the equation above as follows:
\begin{align*}
    \mirmap\left( \Phi^{+}_{\chamb^!} \cdot \MSver^{X^{!}}(Q^{!})\right)
    &= \kappa \left(L^{X^{!}}_{\alpha^!, \chamb^!} \cdot  S^{X^!}_{\chamb^!} \cdot (L^{X^{!}}_{\alpha^!, \chamb^!})^{-1} \right)^{T}\cdot \Phi^-_{\chamb} \cdot \MSver^X(Q^{-1})
    \\
    &= \kappa \left( (L^{X^{!}}_{\alpha^!, \chamb^!})^{-1} \cdot  (S^{\opp, X^!}_{\chamb^!_{\opp}} )^{-1}((Q^!)^{-1})\cdot L^{X^{!}}_{\alpha^!, \chamb^!} \right)\cdot \Phi^-_{\chamb} \cdot \MSver^X(Q^{-1})
    \\
    &= \kappa \left( L^{X^{!}}_{\alpha^!_{\opp}, \chamb^!_{\opp}} \cdot  (S^{\opp, X^!}_{\chamb^!_{\opp}} )^{-1}((Q^!)^{-1})\cdot (L^{X^{!}}_{\alpha^!_{\opp}, \chamb^!_{\opp}})^{-1} \right)\cdot \Phi^-_{\chamb} \cdot \MSver^X(Q^{-1}).
\end{align*}
The first equality follows by Lemma \ref{lemma: mirror symmetry stabs rephrasing}, the second by Proposition \ref{prop: duality stable envelopes}, and the third by the equalities
\[
\frac{\det(\alpha^!)|_{f^!}}{\det(N^-_{f^!})}=\left(\frac{\det(\alpha_{\opp}^!)|_{f^!}}{\det(N^+_{f^!})}\right)^{-1} \qquad \forall f^!\in (X^!)^{\At^!},
\]
which can be checked with a straightforward computation.
Inverting the relevant matrices and noticing that $\Phi^-_{\chamb}=\Phi^+_{\chamb_{\opp}}$, we get 
\[
\kappa^{-1} \left(\Phi^+_{\chamb_{\opp}} \cdot \MSver^X(Q^{-1})\right)= L^{X^{!}}_{\alpha^!_{\opp}, \chamb^!_{\opp}} \cdot  S^{\opp, X^!}_{\chamb^!_{\opp}} ((Q^!)^{-1})\cdot (L^{X^{!}}_{\alpha^!_{\opp}, \chamb^!_{\opp}})^{-1} \Phi^{-}_{\chamb^!_{\opp}} \cdot \MSver^{X^{!}}(Q^!).
\]
The statement of the theorem can now be obtained by expanding both sides in the fixed point basis.

\end{proof}

\subsection{Reduction to weight 1 NS5 branes}

We will need the following lemma.
\begin{lemma}\label{lem: L identity}
Let $f,g, \in X^{\At}$ and let $\ol g$ be an NS5 resolution of $g$. Let $L_f=(L^{X}_{\alpha, \chamb})_{ff}$ where $\chamb$ is the standard chamber and $\alpha$ is the polarization \eqref{eq: class alpha for general bow variety}. Recall $Y$ from \S\ref{sec: equivariant geometry of ns5}. The following equation holds:
\[
\frac{\vartheta(N^{-}_{\ol g/Y})}{\vartheta(N_{\ol g/Y^{\hbar}})}(a \hbar^{-\gamma}) L_{g}^{-1} L_{f} \Phi_{\chamb_{\ol g} ,Y}  = L_{\ol g}^{-1} L_{\ol f_{\sharp}} \Phi^{-}_{\ol g,Y} .
    \]
\end{lemma}
\begin{proof}
    This can be proven by a straightforward computation using the formulas of \eqref{eq: tan flag} and Corollary \ref{cor: different tangents NS5}. In particular, the $\vartheta$ and $\Phi$ terms on each side all cancel up to a monomial in the equivariant parameters. This monomial is accounted for by the $L$ terms.
\end{proof}

\begin{proposition}\label{prop: NS5 induction}
    Let $X$ be a separated bow variety and let $\ol X$ be an NS5 resolution of $X$. If mirror symmetry of vertex functions holds for $\ol X$ and $\ol X^{!}$, then it holds for $X$ and $X^{!}$.
\end{proposition}
\begin{proof}
Recall Proposition \ref{proposition separated NS5 fusion ratios stabs} and Theorem \ref{thm: ns5vertex}:
\[
\frac{\Stab^{X}_{\chamb}(f)\Big|_g}{\Stab^{X}_{\chamb}(g)\Big|_g}(a, z, \hbar) \\
  = \sum_{\bar g\in Y^{\At}} \frac{\vartheta(N_{\bar g/Y}^-)}{\vartheta(N_{\bar g/Y^{\hbar}})} (a\hbar^{-\gamma(g)})  \psi^{*}\left(
  \frac{\Stab^{\overline X}_{\chamb}(\bar f_{\sharp})\Big|_{\bar g}}{\Stab^{\overline X}_{\chamb}(\bar g)\Big|_{\bar g}}\right),
\]
and
\[
\MSver^{X}_{g}(Q^{-1})=\Phi_{\chamb_{\ol g},Y}  \psi^{*} \left(\MSflaglim^{Y}(Q')^{-1} \MSver^{\ol X}_{\ol g}(\ol Q^{-1})\right).
\]
To lighten notation in the lines below, we will omit the arguments $(Q^{-1})$, $(Q')$ and $(\ol Q^{-1})$. We will also omit the subscript $\chamb$.

Starting with the right hand side of Theorem \ref{thm: mirsym} and applying the previous two results, we have
\begin{align*}
&\sum_{g \in X^{\At}} \MSstab^{X}(f)|_{g} \Phi^{-}_{g} \MSver^{X}_{g}
=\sum_{g \in X^{\At}} L_{f} \frac{\Stab^{X}(f)|_{g}}{\Stab^{X}(g)|_{g}} L_{g}^{-1}\Phi^{-}_{g} \MSver^{X}_{g} \\
&=\sum_{g \in X^{\At}} L_{f} \left( \sum_{\ol g \in Y^{\At}} \frac{\vartheta(N_{\bar g/Y}^-)}{\vartheta(N_{\bar g/Y^{\hbar}})} (a\hbar^{-\gamma(g)})  \psi^{*} \left(\frac{\Stab^{\ol X}(\ol f_{\sharp})|_{\ol g}}{\Stab^{\ol X}(\ol g)|_{\ol g}}  \right) \right) L_{g}^{-1}\Phi^{-}_{g} \MSver^{X}_{g}\\
&=\sum_{g \in X^{\At}} L_{f} \sum_{\ol g \in Y^{\At}} \frac{\vartheta(N_{\bar g/Y}^-)}{\vartheta(N_{\bar g/Y^{\hbar}})} (a\hbar^{-\gamma(g)})  \psi^{*} \left(\frac{\Stab^{\ol X}(\ol f_{\sharp})|_{\ol g}}{\Stab^{\ol X}(\ol g)|_{\ol g}} (\MSflaglim^{Y})^{-1} \MSver^{\ol X}_{\ol g}  \right)  L_{g}^{-1}\Phi^{-}_{g} \Phi_{\chamb_{\ol g},Y}.
\end{align*}
Continuing with Lemma \ref{lem: L identity} and Corollary \ref{cor: NS5 stab restrictions}, we obtain
\begin{align*}
&\sum_{g \in X^{\At}}  \sum_{\ol g \in Y^{\At}}  \psi^{*} \left(\frac{\Stab^{\ol X}(\ol f_{\sharp})|_{\ol g}}{\Stab^{\ol X}(\ol g)|_{\ol g}} (\MSflaglim^{Y})^{-1} \MSver^{\ol X}_{\ol g} \right)  \frac{\vartheta(N_{\bar g/Y}^-)}{\vartheta(N_{\bar g/Y^{\hbar}})} (a\hbar^{-\gamma(g)})  L_{f} L_{g}^{-1}\Phi^{-}_{g} \Phi_{\chamb_{\ol g},Y} \\
&=\sum_{g \in X^{\At}} \sum_{\ol g \in Y^{\At}}  \psi^{*} \left(\frac{\Stab^{\ol X}(\ol f_{\sharp})|_{\ol g}}{\Stab^{\ol X}(\ol g)|_{\ol g}} (\MSflaglim^{Y})^{-1} \MSver^{\ol X}_{\ol g} \right) L_{\ol f_{\sharp}}  L_{\ol g}^{-1}   \Phi^{-}_{\ol g} \\
&=\psi^{*} \left( (\MSflaglim^{Y})^{-1} \sum_{k \in \ol{X}^{\At}}\MSstab^{\ol X}(\ol f_{\sharp})|_{k}  \Phi^{-}_{k} 
 \MSver^{\ol X}_{k}     \right).
\end{align*}
Applying mirror symmetry for $\ol X$, this is equal to
\begin{align*}
 &\psi^{*} \left((\MSflaglim^{Y})^{-1} \mirmap_{\ol X^{!} \to \ol X} \Phi^{+}_{(\ol f_{\sharp})^{!}} \MSver^{\ol X^{!}}_{(\ol f_{\sharp})^{!}}  \right).
\end{align*}
Let $\varphi^{*}$ be the variable specialization for the D5 resolution $X^{!} \hookrightarrow \wt{X^{!}}$. One readily checks that $\psi^{*} \mirmap_{\ol X^{!} \to \ol X} = \mirmap_{X^{!} \to X} \varphi^{*}$. So the previous line is equal to
\begin{align*}
 &=\psi^{*}(\MSflaglim^{Y})^{-1}  \psi^{*}  \mirmap_{\ol X^{!} \to \ol X} \Phi^{+}_{(\ol f_{\sharp})^{!}} \MSver^{\wt{X^{!}}}_{\wt{f^{!}_{\sharp}}}  \\
 &=\psi^{*}(\MSflaglim^{Y})^{-1}   \mirmap_{X^{!} \to X} \varphi^{*}  \Phi^{+}_{\wt{f^{!}_{\sharp}}} \MSver^{\wt{X^{!}}}_{\wt{f^{!}_{\sharp}}}.
 \end{align*}

Applying Theorem \ref{thm: d5 vertex}, the latter becomes
\begin{align*}
&\psi^{*} (\MSflaglim^{Y})^{-1}   \mirmap_{X^{!} \to X} \left( \varphi^{*} (\Phi^{+}_{\wt{f^{!}_{\sharp}}}) \MSver^{X^{!}}_{f^{!}} \right) \\
  &=\mirmap_{X^{!} \to X}  \Phi^{+}_{f^{!}} \MSver^{X^{!}}_{f^{!}},
\end{align*}
which is exactly the left hand side of Theorem \ref{thm: mirsym}. This completes the proof.
\end{proof}

\begin{proposition}
    Let $X$ be a separated bow variety and let $\wt X$ be a D5 resolution of $X$. If mirror symmetry of vertex functions holds for $\wt X$ and $(\wt X)^{!}$, then it holds for $X$ and $X^{!}$.
\end{proposition}

\begin{proof}
    Consider the equivalent formulation of mirror symmetry from Theorem \ref{thm: opmirsym}. Since D5 and NS5 resolutions are dual to each other, we must prove that mirror symmetry of vertex functions for the (co-separated) variety $(\wt{X})^{!}= \ol{X^{!}}$ implies mirror symmetry for $X^{!}$. Since $\ol{X^{!}}$ is an NS5 resolution of $X^{!}$, we can repeat the same argument in the proof of Proposition \ref{prop: NS5 induction}, using Theorem \ref{thm: cosepNS5vertex} instead of Theorem \ref{thm: ns5vertex} and Proposition \ref{prop: co-separated NS5 fusion ratios stabs opp} instead of Proposition \ref{proposition separated NS5 fusion ratios stabs}, to prove the statement.
\end{proof}

\begin{proof}[Proof of Theorem \ref{thm: mirsym}]
    Using the trick of \cite[\S5.11]{BR} (which shows that every separated or coseparated bow variety is Hanany-Witten equivalent to a bow variety whose branes all have strictly positive weights) along with Proposition \ref{prop: vertex and HW}, we can without loss of generality assume that all NS5 and D5 brane weights are nonzero. Then using the previous two propositions, mirror symmetry of vertex functions for bow varieties reduces to case when all D5 and NS5 branes have weight one. Such varieties are the self-dual cotangent bundles of complete flag varieties, for which mirror symmetry was proven in \cite{dinkms2}.
\end{proof}

\appendix
\section{Mirror symmetry, D5 resolutions, and NS5 resolutions for \texorpdfstring{$T^*\bbP^1$}{TP1}}
\label{appendix MS}

\setcounter{figure}{0}
\renewcommand\thefigure{A.\arabic{figure}}

\subsection{Setup}

In this appendix we demonstrate all our main results in the simplest nontrivial example, $X=T^{*}\mathbb{P}^{1}$, and explain their connection to known formulas in the theory of $q$-hypergeometric series. Consider the following 3d mirror dual bow varieties: 
\[
X=\ttt{\fs 1\fs 2\bs 1\bs}\qquad X^!=\ttt{\bs 1\bs 2\fs 1\fs}
\]
It is easy to check that they are both isomorphic to $T^*\bbP^1$. In both cases, the GIT presentation is as the quotient of the semistable locus of 
\[
 \{(x,y)\in T^*\Hom(\C^2, \C) \; | \; xy=0\}
\]
by the natural action of $G=\Cs$. A point $(x,y)$ is (semi)-stable iff $x$ is surjective. 

We denote by $\Tt=\Cs_{a_1}\times \Cs_{a_2}\times \Cs_{\hbar}$ (resp. $\Tt^!=\Cs_{a^!_1}\times \Cs_{a^!_2}\times \Cs_{\hbar}$) the torus acting on $X$ (resp. $X^!$). The fixed points of $X$ and $X^!$ can be described via tie diagrams as in Figure \ref{fig: fixexd points TP1}.

\begin{figure}
    \centering
    \begin{tikzpicture}[baseline=0,scale=.3]
\begin{scope}[xshift=0cm, yshift=6cm]
\draw [thick,red] (0.5,0) --(1.5,2); 
\draw[thick] (1,1) node [left] {$f_1=$}--(2.5,1) node [above] {$1$} -- (10,1);
\draw [thick,red](3.5,0) --(4.5,2);  
\draw [thick](4.5,1)--(5.5,1) node [above] {$2$} -- (6.5,1);
\draw [thick,blue](7.5,0) -- (6.5,2);  
\draw [thick](7.5,1) --(8.5,1) node [above] {$1$} -- (9.5,1); 
\draw[thick,blue] (10.5,0) -- (9.5,2);  

\draw [dashed, black](4.5,2.25) to [out=45,in=-225] (9.5,2.25);
\draw [dashed, black](1.5,2.25) to [out=45,in=-225] (6.5,2.25);
\end{scope}
\begin{scope}[xshift=15cm, yshift=6cm]
\draw [thick,red] (0.5,0) --(1.5,2); 
\draw[thick] (1,1) node [left] {$f_2=$}--(2.5,1) node [above] {$1$} -- (10,1);
\draw [thick,red](3.5,0) --(4.5,2);  
\draw [thick](4.5,1)--(5.5,1) node [above] {$2$} -- (6.5,1);
\draw [thick,blue](7.5,0) -- (6.5,2);  
\draw [thick](7.5,1) --(8.5,1) node [above] {$1$} -- (9.5,1); 
\draw[thick,blue] (10.5,0) -- (9.5,2);  

\draw [dashed, black](4.5,2.25) to [out=45,in=-225] (6.5,2.25);
\draw [dashed, black](1.5,2.25) to [out=45,in=-225] (9.5,2.25);
\end{scope}

\begin{scope}[xshift=0cm]
\draw [thick,red] (1.5,0) --(0.5,2); 
\draw[thick] (1,1) node [left] {$f_1^{!}=$}--(2.5,1) node [above] {$1$} -- (10,1);
\draw [thick,red](4.5,0) --(3.5,2);  
\draw [thick](4.5,1)--(5.5,1) node [above] {$2$} -- (6.5,1);
\draw [thick,blue](6.5,0) -- (7.5,2);  
\draw [thick](7.5,1) --(8.5,1) node [above] {$1$} -- (9.5,1); 
\draw[thick,blue] (9.5,0) -- (10.5,2);  

\draw [dashed, black](4.5,-0.25) to [out=-45,in=225] (9.5,-0.25);
\draw [dashed, black](1.5,-0.25) to [out=-45,in=225] (6.5,-0.25);
\end{scope}

\begin{scope}[xshift=15cm]
\draw [thick,red] (1.5,0) --(0.5,2); 
\draw[thick] (1,1) node [left] {$f^!_2=$}--(2.5,1) node [above] {$1$} -- (10,1);
\draw [thick,red](4.5,0) --(3.5,2);  
\draw [thick](4.5,1)--(5.5,1) node [above] {$2$} -- (6.5,1);
\draw [thick,blue](6.5,0) -- (7.5,2);  
\draw [thick](7.5,1) --(8.5,1) node [above] {$1$} -- (9.5,1); 
\draw[thick,blue] (9.5,0) -- (10.5,2);  

\draw [dashed, black](4.5,-0.25) to [out=-45,in=225] (6.5,-0.25);
\draw [dashed, black](1.5,-0.25) to [out=-45,in=225] (9.5,-0.25);
\end{scope}
\end{tikzpicture}
\label{fig: fixexd points TP1}
\caption{Fixed points of $T^*\bbP^1$ (on top) and their dual (on the bottom).}
\end{figure}
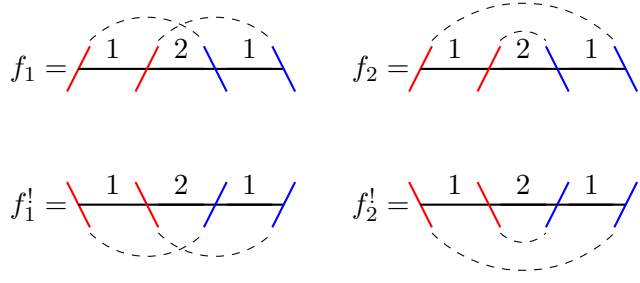

Their tangent spaces are
\[
T_{f_1} X= \frac{a_1}{a_2}+\hbar \frac{a_2}{a_1} \qquad T_{f_2} X= \frac{a_2}{a_1}+\hbar \frac{a_1}{a_2}
\]
\[
T_{f^!_1} X^!= \frac{a^!_2}{a^!_1}+\hbar^! \frac{a^!_1}{a^!_2} \qquad 
T_{f^!_2} X^!= \frac{a^!_1}{a^!_2}+\hbar^! \frac{a^!_2}{a^!_1}.
\]
For the standard chambers, namely $\{a_1<a_2\}$ and  $\{a^!_1<a^!_2\}$, we have $f_1<f_2$ and $f_2^!<f_1^!$.

\subsection{Vertex functions and $q$-hypergeometric series}\label{sec: TP1 vertex}

The vertex functions can be written down using Proposition \ref{prop: vertex formula}. For that, we need to fix a polarization of $X$, which, in agreement with Theorem \ref{thm: mirsym} we set equal to 
\[
\hbar(\alpha+\beta)^\vee= \frac{a_1}{t}+ \frac{a_2}{t}-1\in K_{\Tt}(X).
\]
where $t$ denotes the class of the tautological line bundle. The result for the fixed point $f_1$ is 
\begin{align*}
    \ver_{f_1}
    &=\sum_{d \geq 0} \frac{(1)_{-d} (\frac{a_2}{a_1})_{-d}}{(q\hbar)_{-d} (q\hbar \frac{a_2}{a_1})_{-d}} (\hbar^{-1} q^{-2} Q)^{d}
    \\
    &=\sum_{d \geq 0} \frac{(\hbar^{-1})_{d} (\hbar^{-1} \frac{a_1}{a_2})_{d}}{(q)_{d} (q\frac{a_1}{a_2})_{d}} (q\hbar)^{2d}(\hbar^{-1} q^{-2} Q)^{d}
    \\
    &=\sum_{d \geq 0} \frac{(\hbar^{-1})_{d} (\hbar^{-1} \frac{a_1}{a_2})_{d}}{(q)_{d} (q\frac{a_1}{a_2})_{d}} (\hbar Q)^{d}.
\end{align*}
Where in the second step we used the identity
\[
(x)_{-d}=\frac{(-q/x)^d q^{d(d-1)/2}}{(q/x)_d}.
\]
The computation for $V_{f_2}$ is similar. Overall, we obtain

\begin{align*}
    \ver_{f_1}&=\sum_{d \geq 0} \frac{(\hbar^{-1})_{d} (\hbar^{-1} \frac{a_1}{a_2})_{d}}{(q)_{d} (q\frac{a_1}{a_2})_{d}} (\hbar Q)^{d}=\hyp(\hbar^{-1}, \hbar^{-1}\frac{a_1}{a_2}; q\frac{a_1}{a_2}; \hbar Q)
    \\
    \ver_{f_2}&=\sum_{d \geq 0} \frac{(\hbar^{-1})_{d} (\hbar^{-1} \frac{a_2}{a_1})_{d}}{(q)_{d} (q\frac{a_2}{a_1})_{d}} (\hbar  Q)^{d}=\hyp(\hbar^{-1}, \hbar^{-1}\frac{a_2}{a_1}; q\frac{a_2}{a_1}; \hbar Q).
\end{align*}

\begin{remark}
    Here, $\hyp$ is Heine's basic hyper-geometric series
    \[
    \hyp(a,b,c;z)=\sum_{d\geq 0} \frac{(a)_{d} (b)_{d}}{(q)_{d} (c)_{d}} (z)^{d},
    \]
    see \cite{koelink-q} for a review. More generally, the vertex functions of $T^*\bbP^{n-1}$ recover the hyper-geometric series $\prescript{}{n}{\phi_{n-1}}$ for certain values of the parameters.
\end{remark}

\begin{remark}
    Notice that the vertex functions $V_{f_1}$ and $V_{f_2}$ only differ by the swap of variables $a_1 \leftrightarrow a_2$. This is a special case of Proposition \ref{prop: permute vertex}.
\end{remark}

We now compute the vertex functions of $X^!$. Since $ X$ and  $X^!$ are in the same HW equivalence class, it follows from Corollary \ref{corcosepV} that the vertex functions are equal up to a shift of $Q^{!}$ by a power of $q$ (after the trivial renaming of variables: $Q=Q^{!}$, $a=a^{!}$, $\hbar=\hbar^{!}$). Taking into account that $X^{!}$ uses the opposite of the alpha class as the polarization, namely
\[
\alpha^!+\beta^!=\frac{\hbar^! a_1^!}{t^!}+\frac{\hbar^! a_2^!}{t^!}-\hbar^!,
\]
and using Corollary \ref{corcosepV} (with $n=2$), we get

\begin{align*}
    \ver_{f^!_1}&=\sum_{d \geq 0} \frac{(\hbar^{!-1})_{d} (\hbar^{!-1} \frac{a^!_2}{a^!_1})_{d}}{(q)_{d} (q\frac{a^!_2}{a^!_1})_{d}} (\hbar^! q^{2} Q^!)^{d}=\hyp(\hbar^{!-1}, \hbar^{!-1}\frac{a^!_2}{a^!_1}; q\frac{a^!_2}{a^!_1}; \hbar^! Q^!)\\
    \ver_{f^!_2}&=\sum_{d \geq 0} \frac{(\hbar^{!-1})_{d} (\hbar^{!-1} \frac{a^!_1}{a^!_2})_{d}}{(q)_{d} (q\frac{a^!_1}{a^!_2})_{d}} (\hbar^! q^{2} Q^!)^{d}=\hyp(\hbar^{!-1}, \hbar^{!-1}\frac{a^!_1}{a^!_2}; q\frac{a^!_1}{a^!_2}; \hbar^! Q^!).
\end{align*}

\subsection{Mirror symmetry, summation formula, and connection formula}
\label{app: Mirror symmetry, summation formula, and connection formula}

Notice that $\w(\Zb_i)=\w(\Zb_i^!)=1$ for all $i=1,2$, hence applying \eqref{eq: def of MSver} we get 
\[
\MSver_{f_i}=\ver_{f_i}\qquad \MSver_{f^!_i}=\ver_{f^!_i}.
\]
To remove clutter, we set
\begin{align*}
    u&=a_1/a_2\\
    u^!&=a^!_1/a^!_2\\
    \Phi^-_{f_i}&=\Phi((q-\hbar^{-1})N_{f_i}^{-})\\
    \Phi^+_{f^!_i}&=\Phi((q-(\hbar^{!})^{-1}N_{f^!_i}^+).
\end{align*}
The mirror symmetry map $\mirmap$ is then given by\footnote{Since $q^{!} \mapsto q$, we have chosen not to write the former anywhere.}
\[
\hbar^!\mapsto \frac{1}{\hbar q} \qquad  q^{-1} Q^!\mapsto u \qquad u^!\mapsto qQ.
\]

To make Theorem \ref{thm: mirsym} explicit, we need formulas for the stable envelopes of $X$. Using \eqref{eq: def of msstab}, the polarization\footnote{Note that vertex functions and stable envelopes of $X$ use opposite polarizations.} $\alpha=\hbar(t/a_1 +t/a_2) -\hbar $, and the explicit formulas from \cite[Prop 5.12]{BR}, we obtain:

\begin{align*}
    \MSstab_{f_1f_2}&=\hbar^{-1/2}u^{-1}\frac{\Stab(f_2)|_{f_1}}{\Stab(f_1)|_{f_1}}=\hbar^{-1/2}u^{-1}\frac{\vartheta(\hbar)}{\vartheta(u)}\frac{\vartheta(uQ)}{\vartheta(Q)}
    \\
    \MSstab_{f_2f_1}&=\hbar^{1/2}\frac{\Stab(f_1)|_{f_2}}{\Stab(f_2)|_{f_2}}=0
    \\
    \MSstab_{f_1f_1}&=1
    \\
    \MSstab_{f_2f_2}&=1.
\end{align*}
Therefore, mirror symmetry for vertex functions of $T^*\bbP^1$ can be explicitly stated as follows:
\begin{proposition}
    Mirror symmetry for $T^*\bbP^1$ consists of the following statements:
    \begin{align}
            \mirmap \left(\Phi^+_{f^!_1} \ver^{X^{!}}_{f_1^!}(Q^{!})\right)&=\Phi_{f_1}^-\ver^{X}_{f_1}(Q^{-1})
            \label{eq: mir bbP1 f1}
            \\
            \mirmap \left(\Phi_{f_2^!}^+ \ver^{X^{!}}_{f_2^!}(Q^{!})\right)&=
            \MSstab_{f_1f_2}\Phi_{f_1}^-\ver^{X}_{f_1}(Q^{-1})+ \Phi_{f_2}^-\ver^{X}_{f_2}(Q^{-1}).
            \label{eq: mir bbP1 f2}
    \end{align}
\end{proposition}

In the rest of the section, we provide a classical interpretation of these statements. For this, we recall the following classical results for the hypergeometric series $\hyp$, see for example \cite{Gasper_Rahman_2004}.  

\begin{proposition}[Heine's transformation formula]
\label{prop: Heine's summation formula}
Assume that $|z|<1$ and $|b|<1$. Then the following equation holds:
    \begin{equation*}
    \frac{\Phi(c)}{\Phi(b)}\hyp(a,b;c;z)=\frac{\Phi(az)}{\Phi(z)}\hyp(c/b, z,az,b).
\end{equation*}
\end{proposition}

\begin{proposition}[Watson's connection formula] In the region where the series in both sides are convergent, we have:
\label{prop: connection formula}
\begin{multline}
    \hyp(a,b;c;z)=
    \frac{\Phi(a)\Phi(c/b)}{\Phi(c)\Phi(a/b)}\frac{\Phi(bz)\Phi(q/bz)}{\Phi(z)\Phi(q/z)}\hyp(b, bq/c; bq/a; cq/abz)\\+\frac{\Phi(b)\Phi(c/a)}{\Phi(c)\Phi(b/a)}\frac{\Phi(az)\Phi(q/az)}{\Phi(z)\Phi(q/z)}\hyp(a, aq/c; aq/b; cq/abz).
\end{multline}
\end{proposition}

\begin{proposition}
    Mirror symmetry for $f_1$, namely equation \eqref{eq: mir bbP1 f1}, follows from Heine's transformation formula.
    Mirror symmetry for $f_2$, namely equation \eqref{eq: mir bbP1 f2}, follows from Heine's transformation formula and the connection formula.
\end{proposition}

\begin{proof}
Using all the formulas above, equation \eqref{eq: mir bbP1 f1} can be explicitly written as 
\[
\frac{\Phi(Q^{-1})}{\Phi(\hbar Q^{-1})}\hyp(\hbar q, \hbar Q^{-1}; Q^{-1}; \frac{u}{\hbar})=\frac{\Phi(qu)}{\Phi(\hbar^{-1} u)}\hyp(\hbar^{-1}, \hbar^{-1} u; qu; \hbar Q^{-1}).
\]
Hence, it directly follows from Proposition \ref{prop: Heine's summation formula} upon setting 
\[
a=\hbar q\qquad b=\hbar Q^{-1} \qquad c=Q^{-1} \qquad z=u/\hbar.
\]     
For the second fixed point $f_2$, we proceed as follows. Firstly, notice that \eqref{eq: mir bbP1 f2} claims that 
\begin{multline}
\label{eq: explicit mirror bbP1 f2}
    \frac{\Phi(\hbar^{-1}q^{-1}Q^{-1})}{\Phi( q^{-1}Q^{-1})}\hyp(q\hbar, \hbar q^2Q; q^2Q; \frac{u}{\hbar})=\\
    \hbar^{-1/2}u^{-1}\frac{\vartheta(\hbar)}{\vartheta(u)}\frac{\vartheta(uQ)}{\vartheta(Q)}\frac{\Phi(qu)}{\Phi(\hbar^{-1}u)} \hyp(\hbar^{-1}, \hbar^{-1}u; qu; \hbar Q^{-1})
    \\
    +\frac{\Phi(q\hbar u)}{\Phi(u)} \hyp(\hbar^{-1}, \hbar^{-1}u^{-1}; qu^{-1}; \hbar Q^{-1}).
\end{multline}
Substituting
\[
a=\hbar^{-1}\qquad b=\hbar^{-1}u\qquad c=qu\qquad z=\hbar q^2 Q,
\]
in Proposition \ref{prop: connection formula}, we get 
\begin{multline*}
    \hyp(\hbar^{-1},\hbar^{-1}u;qu;\hbar q^2 Q)=
    \frac{\Phi(\hbar^{-1})\Phi(q\hbar)}{\Phi(qu)\Phi(u^{-1})}\frac{\Phi(uq^2Q)\Phi(u^{-1}q^{-1}Q^{-1})}{\Phi(\hbar q^2 Q)\Phi(\hbar^{-1}q^{-1}Q^{-1})}\hyp(\hbar^{-1}u, \hbar^{-1};qu;  \hbar Q^{-1})
    \\
    +\frac{\Phi(\hbar^{-1}u)\Phi(q\hbar u)}{\Phi(qu)\Phi(u)}\frac{\Phi(q^2 Q)\Phi(q^{-1}Q^{-1})}{\Phi(\hbar q^2 Q)\Phi(\hbar^{-1}q^{-1}Q^{-1})}\hyp(\hbar^{-1}, \hbar^{-1}u^{-1}; qu^{-1}; \hbar Q^{-1}).
\end{multline*}
Multiplying by $\frac{\Phi(\hbar^{-1}q^{-1}Q^{-1})}{\Phi( q^{-1}Q^{-1})}\frac{\Phi(\hbar q^2 Q)}{\Phi(q^2Q)}\frac{\Phi(qu)}{\Phi(\hbar^{-1}u)}$ and noticing that, by definition, $\hyp(a,b;c;z)=\hyp(b,a;c;z)$, we obtain 
\begin{multline*}
\frac{\Phi(\hbar^{-1}q^{-1}Q^{-1})}{\Phi( q^{-1}Q^{-1})}\frac{\Phi(\hbar q^2 Q)}{\Phi(q^2Q)}\frac{\Phi(qu)}{\Phi(\hbar^{-1}u)}\hyp(\hbar^{-1},\hbar^{-1}u;qu;\hbar q^2 Q)=
\\
\frac{\Phi(\hbar^{-1})\Phi(q\hbar)}{\Phi(u^{-1})\Phi(\hbar^{-1}u)}\frac{\Phi(uq^2Q)\Phi(u^{-1}q^{-1}Q^{-1})}{\Phi( q^2 Q)\Phi(q^{-1}Q^{-1})}\hyp(\hbar^{-1}, \hbar^{-1}u;qu;  \hbar Q^{-1})
    \\
    +\frac{\Phi(q\hbar u)}{\Phi(u)}\hyp(\hbar^{-1}, \hbar^{-1}u^{-1}; qu^{-1}; \hbar Q^{-1}) .
\end{multline*}
Using the product presentation of the theta function $\vartheta(x)=x^{1/2}\Phi(qx)\Phi(x^{-1})$, we get 
\[
\frac{\Phi(\hbar^{-1})\Phi(q\hbar)}{\Phi(u^{-1})\Phi(qu)}\frac{\Phi(uq^2Q)\Phi(u^{-1}q^{-1}Q^{-1})}{\Phi( q^2 Q)\Phi(q^{-1}Q^{-1})}=\hbar^{-1/2}u^{-1}\frac{\vartheta(\hbar)}{\vartheta(u)}\frac{\vartheta(uQ)}{\vartheta(Q)}.
\]
Combining the last two equations, we deduce 
\begin{multline*}
    \frac{\Phi(\hbar^{-1}q^{-1}Q^{-1})}{\Phi( q^{-1}Q^{-1})}\frac{\Phi(\hbar q^2 Q)}{\Phi(q^2Q)}\frac{\Phi(qu)}{\Phi(\hbar^{-1}u)}\hyp(\hbar^{-1}, \hbar^{-1}u ; qu; \hbar q^2 Q)=
    \\
    \hbar^{-1/2}u^{-1}\frac{\vartheta(\hbar)}{\vartheta(u)}\frac{\vartheta(uQ)}{\vartheta(Q)}\frac{\Phi(qu)}{\Phi(\hbar^{-1}u)} \hyp(\hbar^{-1}, \hbar^{-1}u; qu; \hbar Q^{-1})
    \\
    +\frac{\Phi(q\hbar u)}{\Phi(u)} \hyp(\hbar^{-1}, \hbar^{-1}u^{-1}; qu^{-1}; \hbar Q^{-1}).
\end{multline*}
The last formula can now be matched with \eqref{eq: explicit mirror bbP1 f2} by applying Proposition \ref{prop: Heine's summation formula} to their first terms on the left hand sides. This concludes the proof.
\end{proof}

\subsection{D5 and NS5 properties}

In this last section, we take advantage of the previous explicit formulas to shed light on the D5 and NS5 properties of the vertex functions in the elementary case of $T^*\bbP^1$. 

We start with the D5 property for the pair 
\[
X=\ttt{\fs 1\fs 2\bs} \qquad \wt X= \ttt{\fs 1\fs 2\bs 1\bs}=T^*\bbP^1
\]
The bow variety $X$ is acted on by $\Tt=\Cs_{a}\times \Cs_{\hbar}$ while $\wt X$ is acted on by the larger torus $\wt{\Tt}=\Cs_{a_1}\times \Cs_{a_2}\times \Cs_{\hbar}$. The fixed locus of $X$ is a singleton, represented by the tie diagram 
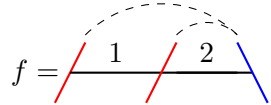
\[
\begin{tikzpicture}[baseline=0,scale=.4]

\draw [thick,red] (0.5,0) --(1.5,2); 
\draw[thick] (1,1) node [left] {$f=$}--(2.5,1) node [above] {$1$} -- (7,1);
\draw [thick,red](3.5,0) --(4.5,2);  
\draw [thick](4.5,1)--(5.5,1) node [above] {$2$} -- (6.5,1);
\draw [thick,blue](7.5,0) -- (6.5,2);  

\draw [dashed, black](4.5,2.25) to [out=45,in=-225] (6.5,2.25);
\draw [dashed, black](1.5,2.25) to [out=45,in=-225] (6.5,2.25);

\end{tikzpicture}
\]
On the other hand, the variety $\wt X$ has two fixed points, displayed in Figure \ref{fig: fixexd points TP1}. The distinguished resolution of $f\in X^{\Tt}$ is $f_2\in (\wt X)^{\wt \Tt}$.

It can be easily checked that $V_f=1$.\footnote{Note that this is non-obvious because the GIT presentation of $X$ is nontrivial, and as a consequence the moduli space of quasimaps are nonzero. Nonetheless, all the associated curve counts are zero except in degree zero.}. One way to see this is to apply Hanany-Witten transition to rewrite $X$ as a bow variety in which all the dimensions are $0$ and invoke Proposition \ref{prop: vertex and HW}. The vertex functions of $\wt X$ are computed in Section \ref{sec: TP1 vertex} (where $\wt X$ was simply denoted by $X$). 
The D5 property, Theorem \ref{thm: d5 vertex} states that 
\[
V_{f}=\left(V_{f_2}\right)\Big|_{\substack{a_1=a\hbar^{-1}\\a_2=a}},
\]
i.e. that 
\[
1=\sum_{d \geq 0} \frac{(\hbar^{-1})_{d} (1)_{d}}{(q)_{d} (q\hbar)_{d}} (\hbar Q)^{d},
\]
which trivially holds because $(1)_d=0$ unless $d=0$. 

We now consider the easiest nontrivial NS5 property for the vertex functions. For this, set 
\[
X=\ttt{\fs 2\bs 1\bs}=T^*\Gr(0,2) \qquad \ol X= \ttt{\fs 1\fs 2\bs 1\bs}=T^*\bbP^1
\]
The only fixed point in $X$ has tie diagram 
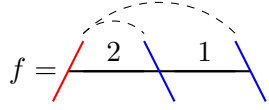
\[
\begin{tikzpicture}[baseline=0,scale=.4]
\draw[thick] (4,1) node [left] {$f=$}--(4,1) -- (10,1);
\draw [thick,red](3.5,0) --(4.5,2);  
\draw [thick](4.5,1)--(5.5,1) node [above] {$2$} -- (6.5,1);
\draw [thick,blue](7.5,0) -- (6.5,2);  
\draw [thick](7.5,1) --(8.5,1) node [above] {$1$} -- (9.5,1); 
\draw[thick,blue] (10.5,0) -- (9.5,2);  

\draw [dashed, black](4.5,2.25) to [out=45,in=-225] (9.5,2.25);
\draw [dashed, black](4.5,2.25) to [out=45,in=-225] (6.5,2.25);
\end{tikzpicture}
\]
while the two fixed points in $\ol X$ are, once again, those in Figure \ref{fig: fixexd points TP1}. By Proposition \ref{prop: vertex formula}, or by using Hanany-Witten transition and Proposition \ref{prop: vertex and HW}, the vertex function of $f$ is just the constant function $1$. Both fixed points $f_1,f_2 \in \ol X$ are NS5 resolutions of $f$. So Theorem \ref{thm: ns5vertex} gives two equations to verify. From Proposition \ref{prop: msflaglim}, we have 
\[
\MSflaglim^{Y}(Q)=\frac{\Phi\left( Q^{-1} \right)}{\Phi\left(\hbar Q^{-1}\right)}
\]
Theorem \ref{thm: ns5vertex} states that
\[
1=\frac{\Phi\left(q \frac{a_1}{a_2}\right)}{\Phi\left(\hbar^{-1} \frac{a_1}{a_2}\right)} \left(\frac{\Phi\left(\hbar Q^{-1} \right)}{\Phi\left( Q^{-1}\right)}\sum_{d \geq 0} \frac{(\hbar^{-1})_{d} (\hbar^{-1} \frac{a_1}{a_2})_{d}}{(q)_{d} (q\frac{a_1}{a_2})_{d}} (\hbar Q^{-1})^{d}\right)\Bigg|_{Q=\hbar}
\]
and
\[
1= \frac{\Phi\left(q \frac{a_2}{a_1}\right)}{\Phi\left(\hbar^{-1} \frac{a_2}{a_1}\right)} 
 \left(\frac{\Phi\left(\hbar  Q^{-1} \right)}{\Phi\left( Q^{-1}\right)}\sum_{d \geq 0} \frac{(\hbar^{-1})_{d} (\hbar^{-1} \frac{a_2}{a_1})_{d}}{(q)_{d} (q\frac{a_2}{a_1})_{d}} (\hbar Q^{-1})^{d}\right)\Bigg|_{Q=\hbar}
\]
Since the series defining the vertex functions converge when $|\hbar Q^{-1}|<1$, we must first apply Proposition \ref{prop: Heine's summation formula} to obtain the analytic continuation. For the first equality, this gives
\begin{multline*}
\frac{\Phi\left(q \frac{a_1}{a_2}\right)}{\Phi\left(\hbar^{-1} \frac{a_1}{a_2}\right)} \left(\frac{\Phi\left(\hbar Q^{-1} \right)}{\Phi\left( Q^{-1}\right)} \sum_{d \geq 0} \frac{(\hbar^{-1})_{d} (\hbar^{-1} \frac{a_1}{a_2})_{d}}{(q)_{d} (q\frac{a_1}{a_2})_{d}} (\hbar Q^{-1})^{d}\right)\Bigg|_{Q=\hbar} \\
=\frac{\Phi\left(q \frac{a_1}{a_2}\right)}{\Phi\left(\hbar^{-1} \frac{a_1}{a_2}\right)} \left(\frac{\Phi\left(\hbar Q^{-1} \right)}{\Phi\left( Q^{-1}\right)} \frac{\Phi\left( Q^{-1} \right)}{\Phi\left( \hbar Q^{-1}\right)} \frac{\Phi\left( \hbar^{-1} \frac{a_1}{a_2}\right)}{\Phi\left( q \frac{a_1}{a_2}\right)} \hyp(\hbar q, \hbar Q^{-1},Q^{-1};\frac{a_1}{\hbar a_2})\right)\Bigg|_{Q=\hbar}
\end{multline*}
which is manifestly equal to $1$. The second equality follows similarly. 

Note that for Heine's transformation to be valid, we must assume that $|a_1/a_2|<1$ in the first case and $|a_2/a_1|<1$ in the second. This parallels the fact that $\ver_{f_1}$ is holomorphic in a neighborhood of of $a_1/a_2=0$ but not $a_2/a_1=0$, while the opposite is true for $\ver_{f_2}$. Likewise, the prefactors $\Phi_{\chamb_{f_{i}}}$ also reflect this.

\printbibliography

@article{aganagic2016elliptic,
    title={Elliptic stable envelopes},
    author={Mina Aganagic and Andrei Okounkov},
    year={2021},
    journal={J. Amer. Math. Soc.},
    volume={34},
    pages={79--133}
}

@book{maulik2012quantum,
	title={Quantum Groups and Quantum Cohomology},
	author={Davesh Maulik and Andrei Okounkov},
	year={2019},
	series={Ast{\'e}risque},
    volume={408},
    publisher={Soci{\'e}t{\'e} Math\'ematique de France}
}

@article{Nakajimaquiver,
	author = {Nakajima, Hiraku},
	year = {1998},
	month = {02},
	pages = {},
	title = {Quiver varieties and Kac-Moody algebras},
	volume = {91},
	journal = {Duke Mathematical Journal},
}

@misc{ginzburg1995elliptic,
	title={Elliptic Algebras and Equivariant Elliptic Cohomology},
	author={Victor Ginzburg and Mikhail Kapranov and Eric Vasserot},
	year={1995},
	eprint={q-alg/9505012},
	archivePrefix={arXiv},
}

@article{Ganter_2014,
author = {Ganter, Nora},
year = {2012},
month = {06},
pages = {},
title = {The elliptic Weyl character formula},
volume = {150},
journal = {Compositio Mathematica},
%doi = {10.1112/S0010437X1300777X}
}

@inproceedings{Grojnowski2007EllipticCD,
	title={Elliptic Cohomology: Delocalised equivariant elliptic cohomology (with an introduction by Matthew Ando and Haynes Miller)},
	author={Ian Grojnowski},
	year={2007}
}

@article{Nakajima_Takayama,
  eprint={1606.02002},
  author = {Nakajima, Hiraku and Takayama, Yuuya},
  journal = {Selecta Mathematica},
  volume={ 23},
  pages={2553--2633},
  year={2017},
 title = {Cherkis bow varieties and Coulomb branches of quiver gauge theories of affine type $A$},
}

@misc{okounkov2017enumerative,
      title={Enumerative geometry and geometric representation theory}, 
      author={Andrei Okounkov},
      year={2017},
      eprint={1701.00713},
      archivePrefix={arXiv},
      %primaryClass={math.AG}
}

@misc{okounkov2020inductiveI,
      title={Inductive construction of stable envelopes and applications, I. Actions of tori. Elliptic cohomology and K-theory}, 
      author={Andrei Okounkov},
      year={2020},
      eprint={2007.09094},
      archivePrefix={arXiv},
      %primaryClass={math.AG}
}

@misc{rimanyi2020bow,
      title={Bow varieties---geometry, combinatorics, characteristic classes}, 
      author={Rich{\'a}rd Rim{\'a}nyi and Yiyan Shou},
      note = {To appear in Comm. in Analysis and Geometry},
      year ={2022},
      eprint={2012.07814},
      archivePrefix={arXiv},
      %primaryClass={math.AG}
}

@thesis{Shou,
	author = {Shou, Yiyan},
	title = {Bow Varieties—Geometry, Combinatorics, Characteristic Classes},
	type = {Ph.D. Thesis},
	institution = {University of North Carolina at Chapel Hill},
	date = {2021},
	%url = {},
}

@article{takayama_2016,
   title={Nahm's Equations, Quiver Varieties and Parabolic Sheaves.},
   volume={52},
   %ISSN={0020-9910},
   number={1},
   journal={Publ. Res. Inst. Math. Sci.},
   %publisher={Springer Science and Business Media LLC},
   author={Takayama, Yuuya},
   year={2016},
   pages={1-41}
}

@misc{botta2021shuffle,
      title={Shuffle products for elliptic stable envelopes of Nakajima varieties}, 
      author={Tommaso Maria Botta},
      year={2021},
      eprint={2104.00976},
      archivePrefix={arXiv},
      %primaryClass={math.AG}
}

@incollection{Okounkov_lectures,
      author         = "Okounkov, Andrei",
      title          = "{Lectures on K-theoretic computations in enumerative
                        geometry}",
    booktitle= {Geometry of Moduli Spaces and Representation Theory},
    publisher={American Mathematical Society},
    series={IAS/Park City Mathematics Series},
      year           = "2017",
      volume={24}
}

@article{cherkis1,
author = {Cherkis, Sergey},
year = {2008},
month = {05},
pages = {},
title = {Moduli Spaces of Instantons on the Taub-NUT Space},
volume = {290},
journal = {Communications in Mathematical Physics},
%doi = {10.1007/s00220-009-0863-8}
}

@article{cherkis2,
author = {Cherkis, Sergey},
year = {2009},
month = {02},
pages = {},
title = {Instantons on the Taub-NUT space},
volume = {14},
journal = {Advances in Theoretical and Mathematical Physics},
%doi = {10.4310/ATMP.2010.v14.n2.a7}
}

@article{cherkis3,
author = {Cherkis, Sergey},
year = {2010},
month = {06},
pages = {},
title = {Instantons on Gravitons},
volume = {306},
journal = {Communications in Mathematical Physics},
%doi = {10.1007/s00220-011-1293-y}
}

@article{kamnitzer,
author = {Kamnitzer, Joel},
year = {2022},
month = {07},
pages = {},
title = {Symplectic resolutions, symplectic duality, and Coulomb branches},
volume = {54},
journal = {Bulletin of the London Mathematical Society},
%doi = {10.1112/blms.12711}
}

@article{IntSei,
    author = "Intriligator, Kenneth and Seiberg, N.",
    title = "{Mirror symmetry in three-dimensional gauge theories}",
     %doi = "10.1016/0370-2693(96)01088-X",
    journal = "Phys. Lett. B",
    volume = "387",
    pages = "513--519",
    year = "1996"
}

@unpublished{Okounkov_video,
    author = {Andrei Okounkov},
    title ={Enumerative symplectic duality},
    note = {Workshop on Mathematical Physics,  https:\-//www.youtube.com/watch?v=9PIE9pUhnkk, ICTP-SAIFR},
    year ={2018}
}

@article{KorZeit,
    author = "Koroteev, Peter and Zeitlin, Anton M.",
    title = "{qKZ/tRS Duality via Quantum K-Theoretic Counts}",
    %doi = "10.4310/MRL.2021.v28.n2.a5",
    journal = "Math. Res. Lett.",
    volume = "28",
    number = "2",
    pages = "435--470",
    year = "2021"
}

@article{qm,
title = {Stable quasimaps to GIT quotients},
journal = {J. Geom. Phys.},
volume = {75},
pages = {17 - 47},
year = {2014},
author = {Ionuţ Ciocan-Fontanine and Bumsig Kim and Davesh Maulik},
}

@book{Gasper_Rahman_2004,
place={Cambridge},
edition={2}, 
series={Encyclopedia of Mathematics and its Applications}, 
title={Basic Hypergeometric Series}, 
publisher={Cambridge University Press}, 
author={Gasper, George and Rahman, Mizan}, 
year={2004}, 
collection={Encyclopedia of Mathematics and its Applications}}

@online{NSmac,
       author = {{Noumi}, Masatoshi and {Shiraishi}, Jun'ichi},
        title = "{A direct approach to the bispectral problem for the Ruijsenaars-Macdonald q-difference operators}",
      journal = {arXiv e-prints},
     keywords = {Mathematics - Quantum Algebra},
         year = 2012,
        month = jun,
          eid = {arXiv:1206.5364},
        pages = {arXiv:1206.5364},
archivePrefix = {arXiv},
       eprint = {1206.5364},
 %primaryClass = {math.QA},
       %adsurl = {https://ui.adsabs.harvard.edu/abs/2012arXiv1206.5364N},
      adsnote = {Provided by the SAO/NASA Astrophysics Data System}
}

@article{BF97,
   title={The intrinsic normal cone},
   volume={128},
   ISSN={1432-1297},
   %url={http://dx.doi.org/10.1007/s002220050136},
   %doi={10.1007/s002220050136},
   number={1},
   journal={Inventiones Mathematicae},
   publisher={Springer Science and Business Media LLC},
   author={Behrend, K. and Fantechi, B.},
   year={1997},
   month=mar, pages={45–88} }

@article{dinkinsD5Vertex,
      author = {{Dinkins}, Hunter},
        title = "{On the vertex functions of type A quiver varieties}",
      journal = {Lett Math Phys},
volume={114},
number={23},
         year = {2024},
        month = {2} ,
        %doi={10.1007/s11005-024-01774-3}
}

@misc{tamagninonabelian,
      title={Nonabelian shift operators and shifted Yangians}, 
      author={Spencer Tamagni},
      year={2024},
      eprint={2412.17906},
      archivePrefix={arXiv},
      %primaryClass={math.AG},
      %url={https://arxiv.org/abs/2412.17906}, 
}

@article{KorZeit2,
author = {Koroteev, Peter and Zeitlin, Anton},
year = {2023},
month = {09},
pages = {},
title = {3D Mirror Symmetry for Instanton Moduli Spaces},
volume = {403},
journal = {Communications in Mathematical Physics},
%doi = {10.1007/s00220-023-04831-5}
}

@article{dinkms2,
       author = {{Dinkins}, Hunter
       },
        title = {3d mirror symmetry of the cotangent bundle of the full flag variety},
     journal={Letters in Mathematical Physics},
     year={2022},
     volume={112},
     number={100},
       %doi={/10.1007/s11005-022-01593-4}
}

@Book{ mac,
author = { Macdonald, I. G. },
title = { Symmetric functions and Hall polynomials},
isbn = { 0198535309 },
publisher = { Clarendon Press ; Oxford University Press Oxford : New York },
year = { 1979 },
type = { Book },
language = { English },
subjects = { Abelian groups.; Finite groups.; Hall polynomials.; Symmetric functions. },
life-dates = { 1979 -  },
catalogue-url = { https://nla.gov.au/nla.cat-vn2653038 },
}

@book{BrualdiRyser,
place={Cambridge}, 
series={Encyclopedia of Mathematics and its Applications}, 
title={Combinatorial Matrix Theory}, 
publisher={Cambridge University Press}, 
author={Brualdi, Richard A. and Ryser, Herbert J.}, 
year={1991}, 
collection={Encyclopedia of Mathematics and its Applications}
}

@article{dinksmir4,
author = {Dinkins, Hunter and Smirnov, Andrey},
year = {2022},
month = {07},
pages = {72},
title = {Euler characteristic of stable envelopes},
volume = {28},
journal = {Selecta Mathematica},
%doi = {10.1007/s00029-022-00788-w}
}

@misc{koelink-q,
      title={q-special functions, basic hypergeometric series and operators}, 
      author={Erik Koelink},
      year={2018},
      eprint={1808.03441},
      archivePrefix={arXiv},
      %primaryClass={math.CA},
      %url={https://arxiv.org/abs/1808.03441}, 
}

@incollection{BFS,
      author         = {Braverman, Alexander and Michael Finkelberg and Jun'ichi Shiraishi},
      title          = {Macdonald polynomials, Laumon spaces and perverse coherent sheaves},
    booktitle= {Perspectives in Representation Theory},
    publisher={American Mathematical Society},
    series={Contemporary Mathematics},
      year           = {2014},
 isbn={9780821891704},
  lccn={2013035921},
}

@article{Heine,
%url = {https://doi.org/10.1515/crll.1846.32.210},
title = {Über die Reihe 1+ . (Aus einem Schreiben an Lejeune Dirichlet).},
author = {E. Heine},
pages = {210--212},
volume = {1846},
number = {32},
journal = {Journal für die reine und angewandte Mathematik},
%doi = {doi:10.1515/crll.1846.32.210},
year = {1846}
}

@article{manolache_working,
   title={Virtual Classes for the Working Mathematician},
   ISSN={1815-0659},
   %url={http://dx.doi.org/10.3842/SIGMA.2020.026},
   %doi={10.3842/sigma.2020.026},
   journal={Symmetry, Integrability and Geometry: Methods and Applications},
   publisher={SIGMA (Symmetry, Integrability and Geometry: Methods and Application)},
   author={Battistell, Luca and Carocci, Francesca and Manolache, Cristina},
   year={2020},
   month=apr }

@article{LMS,
   title={Basic Properties of Non-Stationary Ruijsenaars Functions},
   ISSN={1815-0659},
   %url={http://dx.doi.org/10.3842/SIGMA.2020.105},
   %doi={10.3842/sigma.2020.105},
   journal={Symmetry, Integrability and Geometry: Methods and Applications},
   publisher={SIGMA (Symmetry, Integrability and Geometry: Methods and Application)},
   author={Langmann, Edwin and Noumi, Masatoshi and Shiraishi, Junichi},
   year={2020},
   month=oct }

@article{lee_quantumktheory,
    author = "Lee, Y. -P.",
    title = "{Quantum k-theory I: foundations}",
    eprint = "math/0105014",
    archivePrefix = "arXiv",
    journal = "Duke Math. J.",
    volume = "121",
    pages = "389--424",
    year = "2004"
}

@article{SV,
author = {Smirnov, Andrey and Varchenko, Alexander},
year = {2025},
month = {05},
pages = {},
title = {The p-adic approximations of vertex functions via 3D mirror symmetry},
volume = {115},
journal = {Letters in Mathematical Physics},
%doi = {10.1007/s11005-025-01944-x}
}

@article{PT,
author = {Rahul Pandharipande and Richard P Thomas},
title = {{The $3$–fold vertex via stable pairs}},
volume = {13},
journal = {Geometry \& Topology},
number = {4},
publisher = {MSP},
pages = {1835 -- 1876},
keywords = {curve, Gromov–Witten, threefold, toric},
year = {2009},
%doi = {10.2140/gt.2009.13.1835},
%URL = {https://doi.org/10.2140/gt.2009.13.1835}
}

@article{Liuqmstab, 
title={Quasimaps and stable pairs}, 
volume={9}, 
%doi={10.1017/fms.2021.25}, 
journal={Forum of Mathematics, Sigma}, 
author={Liu, Henry}, 
year={2021},
pages={e32}
}

@article{fusion,
author = {d'Andecy, Loïc},
year = {2017},
month = {12},
pages = {1-36},
title = {Fusion Formulas and Fusion Procedure for the Yang-Baxter Equation},
volume = {20},
journal = {Algebras and Representation Theory},
%doi = {10.1007/s10468-017-9692-1}
}

@article {FRT,
    AUTHOR = {Reshetikhin, N. Yu. and Takhtadzhyan, L. A. and Faddeev, L.
              D.},
     TITLE = {Quantization of {L}ie groups and {L}ie algebras},
   JOURNAL = {Algebra i Analiz},
  FJOURNAL = {Algebra i Analiz},
    VOLUME = {1},
      YEAR = {1989},
    NUMBER = {1},
     PAGES = {178--206},
      ISSN = {0234-0852},
   MRCLASS = {17B65 (17B35 22E46 58F07 81D07 82A69)},
  MRNUMBER = {1015339 (90j:17039)},
MRREVIEWER = {Ya. S. So{\u\i}bel{\cprime}man},
}

@article{dinksmir,
author = {Smirnov, Andrey and Dinkins, Hunter},
year = {2020},
month = {09},
pages = {},
title = {Characters of tangent spaces at torus fixed points and 3d-mirror symmetry},
volume = {110},
journal = {Letters in Mathematical Physics},
%doi = {10.1007/s11005-020-01292-y}
}

@article{dinkms1,
       author = {{Dinkins}, Hunter},
        title = "{Symplectic Duality of $T^*Gr(k,n)$}",
      journal = {Mathematical Research Letters},
      year={2021},
      volume={29},
      number={3}
}

@article{NoumiSano,
author = {Noumi, Masatoshi and Sano, Ayako},
year = {2021},
month = {08},
pages = {},
title = {An infinite family of higher-order difference operators that commute with Ruijsenaars operators of type A},
volume = {111},
journal = {Letters in Mathematical Physics},
%doi = {10.1007/s11005-021-01435-9}
}

@misc{BR,
      title={Bow varieties: Stable envelopes and their 3d mirror symmetry}, 
      author={Tommaso Maria Botta and Richard Rimanyi},
      year={2023},
      eprint={2308.07300},
      archivePrefix={arXiv},
      %primaryClass={math.AG},
      %url={https://arxiv.org/abs/2308.07300}, 
      note={To appear in: Duke Mathematical Journal},
}

@misc{BDCohaYangians,
      title={Okounkov's conjecture via BPS Lie algebras}, 
      author={Tommaso Maria Botta and Ben Davison},
      year={2024},
      eprint={2312.14008},
      archivePrefix={arXiv},
      primaryClass={math.RT},
      url={https://arxiv.org/abs/2312.14008}, 
}

@inproceedings{CRC,
   author = {Jim Bryan and Tom Graber},
     title = {The crepant resolution conjecture},
    booktitle ={Algebraic Geometry: Seattle 2005} ,
    year = {2009},
pages={23-42},
publisher={Amer. Math. Soc.}
}

@article{BCR,
  title={A proof of the Donaldson–Thomas crepant resolution conjecture},
  author={Sjoerd Viktor Beentjes and John Calabrese and J{\o}rgen Vold Rennemo},
  journal={Inventiones mathematicae},
  year={2018},
  volume={229},
  pages={451 - 562},
  url={https://api.semanticscholar.org/CorpusID:119175177}
}

\end{document}